%% file: CDG_VII-2022-August-arch.tex
\documentclass{amsart}

\usepackage[utf8]{inputenc}
\usepackage{microtype}

\usepackage[backref,bookmarks=true,colorlinks=true,pdfstartview=FitV,linkcolor=blue, citecolor=red,urlcolor=green]{hyperref}  

\usepackage{mathrsfs}

\usepackage{amssymb,amsmath,amscd}
\usepackage{framed}

 \usepackage{stmaryrd}
\usepackage{url}

\usepackage[numbers]{natbib} 

\usepackage{subfig}
\usepackage{graphicx}
\usepackage{lpic}
\usepackage{etoolbox}

\makeatletter
\patchcmd{\lp@dimens}
{\epsfig{figure=\lp@epsfile}}
{\includegraphics{\lp@epsfile}}
{}{}
\patchcmd{\lp@dimens}
{\epsfig{figure=\lp@epsfile, width=\lp@tmpx, height=\lp@tmpy}}
{\includegraphics[width=\lp@tmpx,height=\lp@tmpy]{\lp@epsfile}}
{}{}
\makeatother

\usepackage[export]{adjustbox}

\usepackage{xcolor}

 \newcommand{\purple}[1]{\textcolor{black}{#1}}
\newcommand{\blue}[1]{\textcolor{black}{#1}} 
\newcommand{\newch}[1]{\textcolor{black}{#1}} 
 \newcommand{\oldch}[1]{\textcolor{black}{#1}} 
 
\numberwithin{equation}{section}

\input{lstylesCDG}

	\newtheorem{theorem}{Theorem}[section]
\newtheorem{proposition}[theorem]{Proposition}
\newtheorem{lemma}[theorem]{Lemma}
\newtheorem{corollary}[theorem]{Corollary}

\theoremstyle{definition}
\newtheorem{definition}{Definition}[section]

\newtheorem{criterion}{Criterion}[section]

\theoremstyle{remark}
\newtheorem{remark}{Remark}[section]

\begin{document}



\title{Tameness of Margulis space-times with parabolics}


\author{Suhyoung Choi} 
\address{Department of Mathematical Sciences, KAIST,
		305-701, Daejeon, Republic of Korea}
\email{schoi@math.kaist.ac.kr}
\thanks{Choi was supported by the Mid-career Researcher Program through 
the NRF grant NRF-2013R1A1A2056698 funded by the MEST.} 

\author{Todd Drumm}
\address{} 

\author{William Goldman} 
\address{Department of Mathematics, University of Maryland, 
	20742-4015, College Park MD, USA.} 
\email{wmg@math.umd.edu}
\thanks{Goldman was partially supported by the NSF Grant DMS-1709791.}

\dedicatory{Dedicated to the memory of Todd Drumm} 

%

%




%


\begin{abstract} 
	Let $\Lspace$ be a flat Lorentzian space of signature $(2, 1)$. 
	A Margulis space-time  is a noncompact complete Lorentz flat $3$-manifold $\Lspace/\Gamma$ with a free holonomy group $\Gamma$ of rank $\bg, \bg \geq 2$. 
	We consider the case when $\Gamma$ contains a parabolic element.
\purple{We obtain a characterization of proper $\Gamma$-actions in terms of Margulis and Drumm-Charette invariants.} 
	We show that $\Lspace/\Gamma$ is homeomorphic to the interior of a compact handlebody of genus $\bg$ generalizing our earlier result. 
	Also, we obtain a bordification of the Margulis space-time with parabolics 
	by adding a real projective surface at infinity 
	giving us a compactification as a manifold 
	relative to parabolic end neighborhoods.  
	Our method is to estimate the translational
	parts of the affine transformation group and use some $3$-manifold topology. 

\end{abstract} 
	
	

\dedicatory{Dedicated to the memory of Todd Drumm.}






\date{\today}

\maketitle

\input{CDG-text2022-August.tex}

\bibliographystyle{plain} 
\bibliography{cdg}

\end{document}

%% file: lstylesCDG.tex
\newcommand\bR{{\mathbb{R}}}

\newcommand\bZ{{\mathbb Z}}

\newcommand\bV{{\mathbb V}}

\newcommand\bH{{\mathbb{H}}}

\newcommand{\bP}{\mathbb{P}}
\newcommand{\bfb}{\mathbf{b}}
\newcommand{\bfu}{\mathbf{u}}

\newcommand\SI{{\mathbb{S}}}

\newcommand\Bd{{\rm bd}}
\newcommand\clo{{\rm Cl}}
\newcommand\bdd{{\mathbf{d}}}

\newcommand\ra{\rightarrow}

\newcommand\emp{\emptyset}
\newcommand\eps{\epsilon}

\newcommand\Aff{{\mathbf{Aff}}}

\newcommand\Idd{{\rm I}}

\newcommand\bu{{\mathbf{u}}}

\newcommand\Isom{{\mathbf{Isom}}}

\newcommand\SL{{\mathsf{SL}}}

\newcommand\SO{{\mathsf{SO}}}

\newcommand\PGL{{\mathsf{PGL}}}

\newcommand\GL{{\mathsf{GL}}}

\DeclareMathOperator{\rank}{rank}

\newcommand\Ss{{\mathbb{S}}}
\newcommand\bbD{{\mathbb{D}}}

\newcommand{\deltxt}[1]{}
\newcommand{\KFAR}[1]{}

\newcommand\Lspace{\mathsf E}
\newcommand\Uu{\mathsf U} 

\newcommand\SOto{{\mathsf{SO}}(2,1)}
\newcommand\bnu{{\mathbf \nu}}

\newcommand\Ker{{\mathrm{Ker}}}

\newcommand\PSO{{\mathsf{PSO}}}
\newcommand\Sf{{\mathsf{S}}}

\newcommand\Bs{{\mathsf{B}}}
\newcommand\Us{{\mathsf{U}}}

\newcommand{\vv}{\mathbf{v}}
\newcommand{\vw}{\mathbf{w}}
\newcommand{\va}{\mathbf{a}}
\newcommand{\vb}{\mathbf{b}}
\newcommand{\vc}{\mathbf{c}}

\newcommand{\vp}{\mathbf{p}}
\newcommand{\vx}{\mathbf{x}}
\newcommand{\vu}{\mathbf{u}}
\newcommand{\vy}{\mathbf{y}}
\newcommand{\vi}{\mathbf{i}}
\newcommand{\vj}{\mathbf{j}}
\newcommand{\vk}{\mathbf{k}}

\newcommand{\CH}{\mathcal{CH}}

\newcommand{\tbV}{\widetilde{\mathbf{V}}}
\newcommand{\bg}{\mathsf{g}}

\newcommand{\llrrparen}[1]{
	\left(\mkern-4mu\left(#1\right)\mkern-4mu\right)}
\newcommand{\llrrV}[1]{
	\left|\mkern-2mu\left|#1\right|\mkern-2mu\right|}

\newcommand{\rf}{\mathrm{fiber}}

\newcommand{\Paren}[1]{
	\left(#1\right)}



%% file: CDG-text2022-August.tex
\section{Introduction}
Let $\Isom^{+}(\Lspace)$ denote the group of orientation-preserving Lorentzian isometries on 
the oriented flat Lorentzian space $\Lspace$
of the signature $(2, 1)$. 
Here, we have an exact sequence 
\[1 \ra \bR^{2, 1} \ra \Isom^{+}(\Lspace) \stackrel{\mathcal{L}}{\longrightarrow} \SO(2, 1) \ra 1\]
where $\mathcal{L}$ is the homomorphism taking the linear parts of
the isometries. 
A {\em parabolic} of $\Isom^{+}(\Lspace)$ is an element whose linear part is a parabolic element of $\SO(2, 1)$.

A discrete affine group $\Gamma$ acting properly on $\Lspace$ is either solvable or is free of rank $\geq 2$.
(See Goldman-Labourie \cite{GL12}.)  While we will assume that $\Gamma$ is a free group of rank $\geq 2$,
we say that $\Gamma$ is a {\em proper affine free group of rank $\geq 2$}. 

We will often require $\mathcal{L}(\Gamma) \subset \SO(2, 1)^o$ for the subgroup $\SO(2, 1)^o$ of 
$\SO(2, 1)$ acting on the positive cone. Here, 
$\mathcal{L}(\Gamma)$ acts properly discontinuously and freely on 
a hyperbolic plane $\bH^2$ formed by positive rays in the cone. 
We say that  
$\Gamma$ is 
a \hypertarget{term-pad}{{\em proper affine hyperbolic group of rank $\bg$ with linear parts in $\SO(2,1)^o$}}
\begin{itemize}
\item if it acts properly discontinuously faithfully and freely on $\Lspace$, and 
\item $\mathcal{L}(\Gamma)$ is a free group of rank $\bg, \bg\geq 2$ in 
$\SO(2, 1)^o$,
acting freely and discretely on $\bH^{2}$. 
\end{itemize} 
It will be sufficient to prove tameness in this case. 


A \hypertarget{term-rps}{{\em real projective structure}} on a manifold 
is given by a maximal atlas of charts to $\bR P^n, n\geq 1,$ 
with transition maps in $\PGL(n+1, \bR)$. 
A {\em real projective manifold} is a manifold with a real projective structure.

\begin{theorem} \label{thm:main} 
Suppose that $\Gamma$ is a proper affine free group of rank $\bg, \bg \geq 2$, with parabolics and linear parts in $\SO(2,1)^o$. 
Then 
\begin{itemize}
\item $\Lspace/\Gamma$ is diffeomorphic to the interior of a compact handlebody of genus $\bg$. 
\item Moreover, it is the interior of a real projective $3$-manifold $M$ 
with boundary equal to a totally geodesic real projective surface, and $M$ deformation retracts 
to a compact handlebody obtained by removing a union of finitely 
many end-neighborhoods homeomorphic to solid tori. 
\end{itemize}
\end{theorem}

These real projective surfaces are from the paper of Goldman \cite{G87}. 
The second item is the so-called {\em relative compactification}. 

For all cases of Margulis space-times, we have
\begin{corollary} \label{cor:main2} 
Let $\Gamma$ be a proper affine free group  of rank $\geq 2$ with parabolics. Then 
$\Lspace/\Gamma$ is diffeomorphic to the interior of a compact handlebody of genus $\bg$. 
Moreover, it is the interior of a real projective $3$-manifold $M$ with boundary equal to a totally geodesic real projective surface, 
and $M$ deformation-retracts to 
a compact handlebody obtained by removing a union of finitely many end neighborhoods homeomorphic to
solid tori.
\end{corollary} 

\purple{
We denote by $\Ss$ the sphere of directions in $\Lspace$, by $\Ss_{+}$ the space of directions of 
positive time-like directions and by $\Ss_{-}$ the space of directions of negative time-like directions. We will consider $\Ss_+$ as the projectivization of
$\Ss_+ \cup \Ss_-$. 
Then  the quotient space of $\Ss_+$ under $\Gamma$ is a complete hyperbolic surface $\Sf$. 
Let $\mathcal{P}_{\pi_1(\Sf)}$ denote
 the set of parabolic  elements  and the identity element of $\pi_{1}(\Sf)$. 
We denote by $l_{\Ss_+}(g)$ the length of the shortest closed geodesic in $\Ss_+/\Gamma$ corresponding to the element $g \in \Gamma$. 
By Theorem 4.1 of Charette-Drumm \cite{CD05} generalizing the Margulis opposite sign lemma \cite{Margulis84}, 
we will need the following criterion in this paper for our group $\Gamma$.
\begin{criterion} \label{cr:positive}
Let $\Gamma$ be an isometry group acting on $\Lspace$, and let $\alpha(g) \in \bR$ for 
$g\in \Gamma$  denote the Margulis invariant of $g$. 
$\Gamma$ satisfies the following conditions: 
\begin{itemize} 
\item $\alpha(\gamma)> 0$ for every $\gamma \in \pi_{1}(\Sf)\setminus \mathcal{P}_{\pi_1(\Sf)}$,
\item every $\gamma$, $\gamma \in  \mathcal{P}_{\pi_1(\Sf)} \setminus \{\Idd\}$, 
has the positive Charette-Drumm invariant, and 
\item $\alpha(g) \geq c_{\Sf \setminus E} 
l_{\Ss_+}(g)$ for every $g$ realized as a closed geodesic 
in $\Sf \setminus E$ for the union $E$ of mutually disjoint cusp neighborhoods
for a positive constant $c_{\Sf \setminus E}$ depending on $\Sf \setminus E$. 
\end{itemize} 
\end{criterion} 
Of course, we can assume the negativity also
since the change of the orientation of $\Lspace$ changes the signs of 
Margulis invariants and Charette Drumm invariants by \cite{Drumm92} and  \cite{CD05}.
} 


\purple{
\begin{proposition} \label{prop:forward} 
Suppose that $\Gamma$ acts properly on $\Lspace$. Then 
Criterion \ref{cr:positive} holds up to changing the orientation of $\Lspace$. 
\end{proposition} 
\begin{proof} 
This is proved by Theorem 4.1 of \cite{CD05} and Lemma \ref{lem:outsidecusp}. 
\end{proof} 
} 

\newch{Let $\Uu S$ denote the inverse image of the 
projection $\Uu \Sigma \ra \Sigma$ for the the subset $S\subset \Sigma$ 
and the unit tangent bundle $\Uu \Sigma$ of 
a hyperbolic surface $\Sigma$.
Let $\Uu \Lspace$ denote the bundle of unit space-like vectors over 
$\Lspace$.
}
\purple{
\begin{lemma}\label{lem:outsidecusp} 
		Suppose that $\Gamma$ acts properly on $\Lspace$. 
		Let $E'$ be the union of cusp neighborhoods in an $\eps$-thin part of $\Sf$. 
		Then there exists a constant $c^{(1.4)}_{\Sf \setminus E'}$ in
$(0, 1)$ depending on $E'$ such that 
		for any closed curve $g$ realized as a closed geodesic in $\Sf\setminus E'$
		\[    c^{(1.4)}_{\Sf \setminus E'}  l_{\Ss_+}(g) \leq \alpha(g) \leq \frac{1}{c^{(1.4)}_{\Sf \setminus E'}}  l_{\Ss_+ }(g). \]	
\end{lemma}
\begin{proof} 
		Consider the geodesic currents supported in a compact 
		set $\Uu \Sf\setminus \Uu E'$. Then the argument of Goldman-Labourie \cite{GL12} applies 
		to this collection.  We have a conjugacy homeomorphism from 
		the set of geodesic currents on $\Uu \Sf\setminus \Uu E'$ with
		a compact set of \newch{neutral} geodesic currents on $\Us \Lspace/\Gamma$.  \newch{The length of each of these currents gives us the Margulis invariant}. 
\end{proof} 
} 

\purple{
We prove the following characterization of a proper action of $\Gamma$ in terms of 
Margulis and Charette-Drumm invariants.
\begin{theorem} \label{thm:equivalence} 
An affine finitely generated  free group $\Gamma$
of rank $\geq 2$ acts properly discontinuously on
$\Lspace$ if and only if Criterion \ref{cr:positive} holds up to a change of 
the orientation of $\Lspace$. 
\end{theorem} 
The forward part is Proposition \ref{prop:forward}. 
The converse follows from the main result Theorem \ref{thm:unifest} of Section \ref{sec:orbit}. The proof is given at the end of Section \ref{sub:accumulate} 
}

We mention that the tameness of geometrically finite hyperbolic manifolds was first shown by 
 Marden \cite{Marden74} and later by Thurston \cite{Thurston84p}.
 (See Epstein-Marden \cite{EM06}.)
 Let $\bH^3$ denote the hyperbolic $3$-space. 
We take the convex hull $\CH(\Lambda)$ in $\bH^3$ 
of the limit set $\Lambda$ of the Kleinian group $\Gamma$,
 and there is a deformation retraction of 
$\bH^{3}/\Gamma$ to the compact or finite volume 
$\CH(\Lambda)/\Gamma$ having a thick and thin decomposition.  The paper here follows some of Marden's ideas. 
(See also Beardon-Maskit \cite{BM}.)

Also, the approaches here are using thick and thin decomposition ideas of hyperbolic manifolds as suggested by Canary. 
However, we cannot find a canonical type of decomposition yet and artificially construct the \hyperlink{term-pbr}{parabolic regions}. 
Only canonically defined regions in analogy to Margulis thin parts in the hyperbolic manifold theory 
is the regions bounded by parabolic cylinders. (See Section \ref{sub:paraact}).



\oldch{
Note that the tameness of Margulis space-times without parabolics was shown
by Choi and Goldman \cite{CG17} and Danciger, Gu\'eritaud, and Kassel \cite{DGK16}. Danciger, Gu\'eritaud, and Kassel have also announced a proof \cite{DGKp} for the tameness of Margulis space-times with parabolics, extending \cite{DGK162}. In addition, they give a proof \cite{DGKp} of the crooked plane conjecture in this setting, extending their proof in the setting without parabolics from \cite{DGK162}. 
Their methods, based on the deformation theory of hyperbolic surfaces, seem very different than those of the present paper.  
}

Differently from them, we directly obtain $3$-dimensional compactification 
relative to parabolic regions. 
We estimate by integrals the asymptotics of 
translation vectors of the affine holonomies. 
This is done by using the differential 
form version of the cocycles and estimating with geodesic flows on the vector bundles over
the unit tangent bundle of the hyperbolic surface, 
the uniform Anosov nature of the flow \eqref{eqn:PhitC}, and 
\newch{the estimation} of 
the cusp contributions in Appendix \ref{app:1form}.
(See also Goldman-Labourie \cite{GL12}.)
In the cusp neighborhoods, we replace the $1$-form with the standard cusp $1$-form
and use this to estimate the growth of the cocycles. 
We use the exponential decreasing of a component of 
the differential form \purple{along the geodesic flows}. 
Then we use estimates of the integration of the standard cusp $1$-forms
\newch{ in Section \ref{sub:transvector}}.

Using this and the $3$-manifold theory, we show that properly embedded disks
and parabolic regions in $\Lspace$ meet the inverse images of 
compact submanifolds in the Margulis space-time in compact subsets
and find fundamental domains.


Since there are many proper affine actions of discrete groups not based on
Lie algebraic situations as in \cite{DGK16}, \cite{DGK162}, \cite{DGK163},
and \cite{DGKp},  
we hope that our method can generalize to these spaces with \newch{parabolics} 
providing many points of view.
(See Smilga \cite{Smilga2016p},  \cite{Smilga2018}, and \cite{Smilga2016} for example.)



The paper has three parts: the first two sections \ref{sec:prelim} and \ref{sec:CD} are 
preliminary. Appendices \ref{app:A} and \ref{app:1form} are only dependent on 
these two sections. Then the main argument parts follow: 
Section \ref{sec:orbit} discusses the geometry of the proper 
affine action, and Section \ref{sec:MP} discusses the topology of the quotient space. 

In Section \ref{sec:prelim}, we review some projective geometry of Margulis space-times, the hyperbolic geometry of surfaces, 
Hausdorff convergences, and the Poincar\'e polyhedron theorem. 

In Section \ref{sec:CD}, we first review the proper action of parabolic elements on the Lorentz space $\bR^{2, 1}$.
We analyze the corresponding Lie algebra and vector fields. We introduce a canonical parabolic coordinate system of 
$\bR^{2, 1}$. 
In Section \ref{sub:propaffine}, we generalize the theory of Margulis invariants by Goldman, Labourie, and Margulis \cite{GLM09} and Ghosh and Treib \cite{GT}
to groups with parabolics. That is, we introduce Charette-Drumm invariants which generalize the Margulis invariants for parabolic elements. 
In Section \ref{subsec:parabolic}, we will study the \hyperlink{term-pbr}{parabolic regions} and their ruled boundary components.


In Section \ref{sec:orbit}, we will study the limit sets. We show that any sequence of the translation vectors
of elements of $\Gamma$, i.e., cocycle elements, 
will accumulate in terms of directions only to $\Ss_{0}:= \Ss\setminus \Ss_{+}\setminus \Ss_{-}$.
In key result Corollary \ref{cor:conv1}, we will prove that the limit points of a sequence 
of images of a compact set in $\bR^{2, 1}$ under elements of $\Gamma$ are in $\Ss_{0}$.  \purple{We will also prove the converse part of 
the equivalence of the properness of the action 
and Criterion \ref{cr:positive}, i.e., Theorem \ref{thm:equivalence}. }

In Section \ref{sec:MP}, we will find the fundamental domain for $M$ bounded by a finite union of properly embedded 
smooth surfaces showing that
$M$ is tame. We prove our main results 
Theorem \ref{thm:main} and Corollary \ref{cor:main2} here. 
We make use of parabolic regions bounded by parabolic ruled surfaces. 
We avoid using almost crooked planes as in \cite{CG17}.
Instead, we are using disks that are partially ruled in parabolic regions
to understand the intersections with parabolic regions. 
We will outline this major section in the beginning.

In Appendix \ref{app:A}, we will prove facts about the parabolic regions. 

In Appendix \ref{app:1form}, we will show how to modify $1$-forms representing homology classes.  
We give estimates of some needed integrals here. 

%


\subsection*{Acknowledgements} 
We thank Virginie Charette, Jeffrey Danciger, Michael Kapovich, and Fanny Kassel for helpful comments. 
We thank Richard Canary for the idea to pursue the proof here similar to ones of Marden \cite{Marden74} 
and Thurston \cite{Thurston84p} in the Kleinian group theory where they separate the parabolic regions.
We thank the MSRI for the hospitality where this work was partially
carried out during the program ``Dynamics on Moduli Spaces of Geometric Structures'' in 2015.
Also, we initiated the work during the conference ``Exotic Geometric Structures'' 
at the ICERM, Brown University, on September 16-20, 2013.

We do apologize for the length of the article. We felt that the shortening might confuse the readers since we use many ideas in a novel way. 
Also, dividing the paper seemed a bit unethical and to be a disservice to the mathematical community. 

During the preparation of this manuscript, our coauthor Todd Drumm tragically passed away.
Todd pioneered the field by developing the geometric approach to Margulis's breakthough discovery  \cite{Margulis83} and \cite{Margulis84} 
of proper affine actions of nonabelian free groups. We miss him dearly and dedicate this work to his lasting memory.
%



\section{Preliminary} \label{sec:prelim}

We will state some necessary facts here, mostly from the paper \cite{CG17}. 
Let $\Lspace$ denote the oriented flat Lorentzian space-time given 
as an affine space with a bilinear inner-product given by 
\[ \Bs(\vx, \vx) := x_1^2 + x_2^2 - x_3^2, \, \vx = (x_{1}, x_{2}, x_{3}).\] 
A {\em Lorentzian norm} $\llrrV{\vx}$ is given as $\Bs(\vx, \vx)^{\frac{1}{2}}$ where 
$(-1)^{\frac{1}{2}} = i$. 
We will fix a standard orientation on $\Lspace$ and the associated vector space in this paper. 
Hence, $\Lspace$ denote an oriented Lorentz space-time. 

A {\em Margulis space-time} is a manifold of the form $\Lspace/\Gamma$ where 
$\Gamma$ is a proper affine free subgroup of $\Isom(\Lspace)$ of rank 
$\mathbf{g}, \mathbf{g} \geq 2$. 
Elements of $\PSO(2, 1)$ are hyperbolic, parabolic, or elliptic. 
An element of $\Isom(\Lspace)$ is said to be {\em hyperbolic, parabolic, or elliptic }
if its linear part is so.


The topological boundary $\Bd_X A$ of a subset $A$ in another topological space 
$X$ is given as $\clo(A)$ with the set of interior points of $A$ removed. 
We denote by manifold boundary $\partial A$ and the interior $A^o$ of 
a manifold {$A$}  as usual. 
We define the manifold boundary $\partial A := \clo(A)\setminus A^o$ 
for any $i$-dimensional manifold $A$
with $i$-dimensional manifold closure $\clo(A)$, $i=1,2,3$, in a topological space $X$.  

\subsection{The projective geometry of the Margulis space-time}

Let $V$ be a vector space. 
Define $\bP(V)$ as $V\setminus\{0\}/\sim$ where $\vx \sim \vy$ iff $\vx = s \vy$ for $s \in \bR\setminus\{0\}$. 
Denote by $\PGL(V)$ the group of automorphisms induced by $\GL(V)$ on $\bP(V)$. 

Define the projective sphere 
$\SI(V):=V\setminus\{0\}/\sim_+ \hbox{ where } \vx \sim_+ \vy \hbox{ iff } \vx = s\vy \hbox{ for } s \in \bR_+$.
There is a double cover $\SI(V) \ra \bP(V)$ with the deck transformation group 
generated by the antipodal map \hypertarget{term-A}{$\mathcal{A}:\SI(V) \ra \SI(V)$}. 
We will denote by $\llrrparen{\vv} $ the equivalence 
class of $\vv$. 
Let $a_{-}=\mathcal{A}(a)$ denote the antipodal point of $a$. 
Also, given a set $A \subset \SI(V)$, we define 
$A_{-} = \mathcal{A}(A)$. 
Let $\SL_\pm(V)$ denote the group of linear maps of determinant $\pm 1$. 
$\SL_\pm(V)$ acts on $\SI(V)$ effectively and transitively. 

We embed $\Lspace$ as an \hypertarget{term-hemi}{ open hemisphere} 
	in $\SI(\bR^4)$
by sending \[ (x_1, x_2, x_3) \hbox{ to }\,\llrrparen{ 1, x_1, x_2, x_3}\,
\hbox{ for } x_1, x_2, x_3 \in \bR.\]  
The boundary of $\Lspace$ is a great sphere $\Ss$ given by $x_0 =0$. 
The rays of the  positive cone end in an open disk $\Ss_+ \subset \Ss$, and 
the rays of the negative cone end in an open disk $\Ss_- \subset \Ss$
where $\mathcal{A}(\Ss_{\pm}) = \Ss_{\mp}$.
The closure of $\Lspace$ is a $3$-hemisphere \hypertarget{not-H}{$\mathcal{H}$} bounded by
$\Ss$.

The group $\Isom^{+}(\Lspace)$ of orientation-preserving 
isometries acts on $\Lspace$ as a group of affine transformations 
and hence extends to a group $\SL_\pm(\bR^{4})$ of projective automorphisms of $\SI(\bR^{4})$. It restricts to 
the projective automorphism groups of $\mathcal H$
and of $\Ss$ and $\Ss_\pm$ respectively. 






 \subsection{Thin parts of hyperbolic surfaces} \label{sub:thin} 
 As a subgroup of $\SL_\pm(\bR^3) \subset \SL_\pm(\bR^{4})$, the Lorentz group 
 $\SO(2,1)$ acts on $\Ss_+ \cup \Ss_-$ where $\SO(2,1)^o$ is the subgroup acting on $\Ss_+$
 and is an index two subgroup.  
 The space $\Ss_+\cup \Ss_-$ carries a $\SO(2,1)$-invariant hyperbolic metric, 
 and $\SO(2, 1)^o$ acting on $\Ss_+$ forms a Beltrami-Klein model of the hyperbolic plane. 
 We denote the complete Beltrami-Klein metric by $d_{\Ss_+}$. 
 
 Given a nonelementary discrete subgroup $\Gamma$ of $\SO(2, 1)^o$ acting freely on $\Ss_+$,
 we obtain a complete orientable hyperbolic surface $\Sf := \Ss_+/\Gamma$
 with the covering map $p_{\Sf}: \Ss_+\ra \Sf$. 
An {\em end neighborhood} of a manifold $M$ is 
 a component $U$ of the complement of a compact subset of $M$
 that has a noncompact closure $\clo(U)$.

Let $\eps > 0$ be the Margulis constant.
Recall that the {\em {\rm (}$\epsilon$-{\rm )}thin part} of $\Sf$ is the set of points through which
essential loops \newch{with lengths} $< \eps$ pass. 
The thin part is a union of open annuli. 
For a parabolic element, there is an embedded annulus that is a component of 
the thin part. It is a component of $\Sf\setminus c$ for a simple closed curve $c$, and 
a horodisk $H$ in the hyperbolic plane covers it. Here, 
$H/\langle g \rangle$ is isometric to the end-neighborhood
for a parabolic isometry $g$ acting on $H$. 
This end-neighborhood is called a {\em cusp neighborhood}. 
For $\epsilon > 0$, 
a parabolic  ($\epsilon$-)end-neighborhood is 
a component of the $\epsilon$-thin part of $\Sf$ that is an end-neighborhood. 

We choose a union $E$ of disjoint open cusp-neighborhoods in $\Sf$ in an $\eps$-thin part of $\Sf$
and its inverse image $\mathscr{H}$ in $\Ss_{+}$ which 
is a union of mutually disjoint horodisks.

\subsubsection{\purple{Divergence functions}} \label{subsub:div} 

\begin{definition}\label{defn:curvdist} 
\newch{ 
Let $\tilde g:I \ra \Sf$ be an arclength-parameterized geodesic and 
let $g: I \ra \Sf$ be a freely homotopic arc 
which is a  closed arc whenever $\tilde g$ is closed. 
Suppose that there exists a continuous map $A:I \times \bR \ra \Sf$ so that 
\begin{itemize} 
\item $A(t, 0) = \tilde g(t)$ for each $t\in I$, 
\item Define $A_t(s) := A(t, s)$ for each $t\in I, s\in \bR$. Then 
$A_t$ is an arclength-parameterized geodesic perpendicular to $\tilde g$ at 
$\tilde g(t)$ for each $t\in I$, and 
\item $A(t, s_t) = g(t)$ for some $s_t$ for each $t\in I$. 
\end{itemize} 
Then we say that we can {\em project $g$ to $\tilde g$}
by the {\em perpendicular family of geodesics} $A_t$.
If $|s_t| < \eps$ for all $t$, then we say that $g$ is at 
a {\em $d_{\Ss_+}$-distance} $< \eps$. 
The correspondence $g(t) \ra \tilde g(t)$ for $t \in I$ 
to be called  the {\em perpendicular projection},
and the geodesic \newch{segment} between $g(t)$ to $\tilde g(t)$ for each $t$
is called the {\em perpendicular projection path}
and its length $s_t$ the {\em perpendicular distance} at $t$.    
}
\end{definition}
Of course, the family of perpendicular geodesics may not be uniquely determined,
but we make choices. 

We call the $f$ defined as below the {\em divergence function} from $g_1$ to $g_2$.

\begin{lemma}\label{lem:divergence} 
\newch{Let $g_1(t)$ and $g_2(t)$, $t\in [0, l]$, 
 denote the parameterization of geodesics $g_1$ and $g_2$ where $g_1$ is arclength parameterized. 
Suppose that we can project $g_2$ to $g_1$ by a perpendicular family of geodesics $A_t$. We orient these by the forward directions.}  
\begin{itemize} 
\item 	\newch{We orient $A_t$  so that the frame of its tangent vector and
that of $g_1$ is positively oriented at $A_t(0)= g(t)$ for each $t\in  I$.
Define $f(t)$ to be the oriented path length on $A_t$ from $g_1(t)$ to $g_2(t)$. }
\item 	\newch{Let $e_+:= f(l)$ and $e_-:=f(0)$.} 
\item 	Let $\alpha_+$ and $\alpha_-$ denote 
$\pi/2$ minus the respective angles 
at the forward endpoint $v_+$ and 
the starting endpoint $v_-$ of $g_2$ 
	made by $A_0$ and $A_l$ and $g_2$, respectively. 
\end{itemize} 
\newch{Assume $l \geq 1$. 
	Then the following hold\/{\em :}}
	\begin{itemize} 
		\item[{\em (i)}] 
If $|f(0)|, |f(l)| \leq C$, then $|f(t)| < C$ for $ 0 < t < l$.  
Furthermore, $|f|$ has at most onel minimum. 
	\item[{\em (ii)}] The 
 integral of $\newch{|f(t)|}$ over $[0, l]$ is less than $\newch{2|f(0)|+ 2|f(l)|}$. 
	\item[{\em (iii)}] $\newch{\sum_{i=2}^{m-1} |f(t_i)| \leq  2|f(t_1)|+ 2|f(t_m)|}$ if $t_1, \dots, t_m$, 
$t_i \leq t_{i+1}$ for each $i=1, \dots, m-1$, $m \geq 4$,
satisfies $|t_{i+1} - t_i| \geq 1$. 
	\item[{\em (iv)}] 
For the family of functions $l \geq 1$, 
	$F_l: \newch{\bR^2 \ra \bR^2}$ sending $(\alpha_+, \alpha_-)$ to 
	$(e_+, e_-)$ for  each $l \geq 1$ is $3.3$ times a function
decreasing the max norm provided $\newch{|}\alpha_\pm\newch{|} \leq  1/20$.  
\end{itemize} 	
	\end{lemma} 
\begin{proof} 
(i)	
\newch{
We can show by \cite{geodesicII}: 
\begin{multline} 
f(t)= g(y(t)) \hbox{ for } \\
g(y):= \frac{1}{2} (\log (1+y)) -\frac{1}{2} \log(1-y)) \hbox{ and } y(t)= \pm  \frac{\pm c_- s_+ \sinh(t) + c_+ s_- \sinh(l - t)}{c_-c_+ \sinh(l) }  
\end{multline} 
where $c_i = \cosh(|e_i|), s_i= \sinh(|e_i|)$, $i=-,+$. 
Notices that open geodesics become disjoint 
if only one of the endpoints is changed. 
We may assume that $e_-$ and $e_+$ are positive
since old $|g(y(t))|$ is bounded above by the new $|g(y(t))|$ when we change 
all signs to be positive.
We need to consider the case when the signs 
are $+$ without loss of generality. 
}
\newch{
Now $g$ has the expression as a Taylor series of $y$ with only odd powers: 
\[ g(y)=  y + \frac{y^3}{3} + \frac{y^5}{5} + \dots. \]
We see that $y$ as a function of $t$ can have exactly one interior minimum with only non-negative values
or else it is strictly decreasing with some negative values. 
Since this property holds for the odd powers of $y$ with the identical 
interior minimum point and zeros,  our result follows for $e_-, e_+ \geq 0$.
For other cases, we use hyperbolic trigonometry. 
}

(ii) 
%
\newch{
For (ii) and (iii), we can still look at $y(t)$ with positive coefficients only since 
we are seeking the upper bounds. 
We denote by $\tilde y$ the expression obtained 
from $y$ by respectively replacing terms $\sinh(t)$ and $\sinh(l-t)$ by strictly larger
$\frac{1}{2}\exp(t)$ and $\frac{1}{2}\exp(l-t)$ for $0 \leq t \leq l$. 
That is, 
\[\tilde y(t) : = \frac{c_- s_+ e^t + c_+ s_- e^{l - t}}{2 c_-c_+ \sinh(l) }. \]
Now, 
\[\tilde y(l) = \frac{e^l \tanh |e_+| + \tanh |e_-|}{2 \sinh l} 
\hbox{ and } 
\tilde y(0) =  \frac{\tanh |e_+|  + e^l \tanh |e_-|}{2 \sinh l} 
\]
Using $ \tanh(x)< x$ for $x> 0$, and the fact that 
$1/(2\sinh(l)) < 0.5$ 
and $e^l/(2 \sinh(l)) < 1.2$ for $l \geq 1$ while they from strictly decreasing functions of $l$, 
we can show 
\begin{equation}\label{eqn:tildey}
\tilde y(l) < (1.2)|e_+| + (0.5) |e_-|
\hbox{ and }\tilde y(0) < (0. 5)|e_+| + (1.2) |e_-|.
\end{equation} 
By hyperbolic right triangle rules, we can show 
$|e_+|, |e_-| < 0.26$ provided $|\alpha_\pm| < 0.2$ for $l \geq 1$
by considering the contrapositive and the worst cases
since it is again enough to consider the case $e_+, e_- \geq 0$.
Hence $\tilde y(l), \tilde y(0) < 0.5$ and $\tilde y(t) < 0.5$ by the convexity of $\tilde y$.
}

\newch{
Since 
$g$ is strictly increasing, and $0 < y(t) < \tilde y(t)$ for $t> 0$, 
we obtain 
\[\int_o^l |g(y(t))| dt \leq 
\int_0^l |g(\tilde y(t))| dt\] 
provided $0< \tilde y(t) < 1$. 
Since 
the Taylor series becomes a sum of terms that are 
postive number times $\exp(ml + nt)$ for $m, n\in \bZ$, 
we obtain by a term-by-term argument
\[  \int_o^l |g(y(t))| dt  \leq 
\int_0^l |g(\tilde y(t))| dt \leq |g(\tilde y(l))|+|g(\tilde y(0))|.\]
Since $g(x) < 1.1 x $ for $0< x < 0.5$ by the convexity of $g$, 
\eqref{eqn:tildey} implies
\[ 
 |g(\tilde y(l)|+|g(\tilde y(0))| < 2(e_- + e_+) = 2|f(0)| + 2|f(l)|
\] 
}
%
%
%
%
%

(iii) $\sum_{i=2}^{m-1} |f(t_i)|$ is smaller than the integral of $|f|$ over $t_1$ to $t_m$
since we can break up $|f|$ into parts as above and use the step functions dominated by $|f|$. 
\newch{(We may skip an interval containing the unique minimal point.) 
Hence, the sum is smaller than the twice of the sum of $|f(t_1)|$ and $|f(t_m)|$ by (ii). }

(iv) 
Here again, we can look only at the cases when $e_+, e_-\geq 0$
and $\alpha_- \leq 0, \alpha_+ \geq 0$: Replacing the segments at $v_+, v_-$ 
with ones with positive $e_+, e_-$, we can show by hyperbolic geometry that
the max norm of old $(\alpha_+, \alpha_-)$  is greater or
equal to that of new one while $(e_+, e_-)$ does not change.
In \cite{Vanalysis},  we compute the map 
$[0, 1) \times (-1, 0] \ra \bR_+ \times \bR_+$ 
which sends 
\[(x_-, x_+) = (\cos(\pi/2+\alpha_-), \cos(\pi/2+\alpha_+))
=(-\sin(\alpha_-), -\sin(\alpha_+)) \mapsto (e_-, e_+).\]
We computed \newch{by analytic continuation} 
\begin{multline} \label{eqn:e1e2}
e_-= \log \left({\frac{(x_- \coth(l)+x_+
   \textrm{csch}(l))}{{\sqrt{1-x_-^2}}}} +\sqrt{1+\frac{(x_- \coth(l)+x_+ \text{csch}(l))^2}{1-x_-^2}}\right) \\ 
e_+ = \log \left({\frac{(x_+ \coth (l)+{x_-}
   \text{csch}(l))}{\sqrt{1-x_+^2}}}+\sqrt{1+\frac{(x_+ \coth
   (l)+x_- \text{csch}(l))^2}{1-x_+^2}}\right),
\end{multline} 
\newch{where there is a symmetry switching $(e_-, x_-, x_+)$ with $(e_+, x_+, x_-)$,
and we modified the computations in \cite{Vanalysis}
to obtain an analytic continuation when $x_+, x_-$ are very small. }
We use the series 
\begin{multline} 
\log(y + \sqrt{y^2+1}) = \log\left(\sqrt{y^2+1}\right) + \log\left(1 + \frac{y}{\sqrt{1+y^2}}\right)= \\
\frac{1}{2}\log(1+y^2) + \left(\sum_{n=1}^\infty \frac{(-1)^{n+1}}{n}\left(\frac{y}{\sqrt{y^2+1}}\right)^n\right),
\end{multline} 
which is always absolutely convergent. 
\newch{
We may plug into this
\[y = \frac{x_- \coth(l)+x_+ \textrm{csch}(l))}{\sqrt{1-x_-^2}} \hbox{ and }  
\frac{x_+\coth (l)+{x_-} \text{csch}(l)}{\sqrt{1-x_+^2}}, 
\] 
to obtain $e_-$ and $e_+$ respectively in \eqref{eqn:e1e2}. 
}
\newch{
Since $|x_+|, |x_-| < 1/\sqrt{2}$,
$|e_-|$ and $|e_+|$ respectively are bounded above by 
\begin{multline}
\frac{1}{2} \log (1+ 2 v^2) + \left(  \sum_{n=1}^{\infty}  \frac{(-1)^{n+1}}{n } 
(\sqrt{2}v)^n \right)  = \frac{1}{2} \log (1+ 2 v^2)  + \log(1 + \sqrt{2}v)
 \\
\hbox{ for } 
v = (|x_-| \coth(l)+ |x_+| \textrm{csch}(l)) \hbox{ and } 
(|x_+| \coth (l)+{|x_-|} \text{csch}(l)). 
\end{multline} 
}
By the Taylor analysis to order $1$ \newch{ and the Lagrange form of the error}, the function is 
smaller than \newch{$\frac{3}{2} v $ for $v < 1/8$. 
(See \cite{Calculus}.)}
Since $\coth(1) \leq 1.32$ and $\mathrm{csch}(1) \leq 0.86$,
it follows that 
$(x_-, x_+) \mapsto (e_-, e_+)$ is
\newch{ $\frac{3}{2}(1.32+ 0.86)$ times a norm-nonincreasing function} in terms of max norms provided $\max\{|x_-|, |x_+|\} <  \frac{1}{8\times 2.18}$. 
\newch{Since $x \ra \sin(x)$ is a strictly convex for $0\leq  x< 1/(8 \times 2.18)$, 
we take angles to satisfy $|\alpha_-|, |\alpha_+| \leq 0.05 < \arcsin\left( \frac{1}{8\times 2.18}\right)$.
Then since $\arcsin(\alpha)<1.00056 \alpha$, $0 \leq \alpha  < 0.05$, 
we are done.
(See \cite{Vanalysis}.)}
	
	\end{proof}

A {\em broken geodesic} is a path consisting of 
parameterized geodesics except for 
isolated sets of points. 
For a broken geodesic, a {\em vertex} is a nonsmooth point of it. 
A {\em turning angle} at a vertex is the angle that the tangent vector
the ending geodesic and one for the starting geodesic makes 
at the vertex. 
Since we are on an oriented surface $\Sf$, we can say that 
the path can turn right or left at the vertex. 
The left-turning angle will be considered positive, 
and the right-turning  angle will be considered negative.  

\begin{lemma} \label{lem:perturb} 
\newch{Let $g$ be a closed curve in $\Sf$ consisting of geodesic segments.}
Suppose that $g$ is not parabolic. 
Suppose that the turning angles at vertices are within 
$(- \delta,  \delta)$. 
Assume that $\delta < 1/40. $ 
\newch{For the closed geodesic $\tilde g$ freely homotopic to $g$,
suppose that each geodesic segment of $g$ has a projected image with 
the length at least $1$. 
Then $\tilde g$ has  
an arclength parameterization $\tilde g(t)$ with following properties{\em :} }
\begin{itemize} 
\item There is a corresponding perpendicular parametrization $g(t)$ of $g$ 
so that $d_{\Ss_+}(g(t), \tilde g(t))  \leq \eps $ for 
$0 < \eps \leq 6.6\delta$. 
\item Let $\zeta$ be a bounded $1$-form  defined on 
a compact subset $K$.  Let $C_K$ denote the maximum value of the norm of 
$\zeta$.  Let $\alpha$ be a union of mutually disjoint geodesic subarcs in
a gedoesic subarc in $g$, going into $K$, corresponding to 
a union $\tilde \alpha$ of subarcs in $\tilde g$ where every perpendicular 
geodesic path between them \newch{is} also going into $K$.
Then the absolute value of 
the difference of respective integrals of $\zeta$ on $\alpha$ and $\tilde \alpha$ is less than  $4C_K\eps$.
\end{itemize} 
\end{lemma} 
\begin{proof} 
\newch{Let $\tilde g:I \ra \Sf$ denote the closed geodesic. }
	We draw the perpendicular lines at points of $\tilde g$ passing through broken points of $g$. 
A vertex $g(t_0)$ is {\em good} is the geodesic segments ending there has angles in $(\pi/2-2\delta, \pi/2+2\delta)$ 
with the perpendicular line to $\tilde g$ at $\tilde g(t_0)$.  
\newch{A geodesic segment $e$ is {\em good at $v$} 
if it satisfies the condition for $e$ for that side.}
We let $f: I \ra \bR$ be a function given by 
sending $t $ to the perpendicular distance if 
$g(t)$ is in the right side of $\tilde g$ and to 
$(-1)$ times that if $g(t)$ is in the left side. 

\newch{We prove by induction on the number of vertices. }
If a vertex of $g$ corresponds to  a  local maximum or the local minimum of the perpendicular distance function, then it is a good vertex
\newch{since the turning angles are within $(-\delta, \delta)$. }
\newch{
Since $g$ is closed, there are at least two good vertices.
For a broken geodesic, 
a local maximum of $|f|$ cannot occur in the interior point of a segment
by hyperbolic geometry, 
but a local minimum of $|f|$ can occur. }

\newch{
We  consider a maximal subarc $m$ in $g$ with no good interior vertex  and
$f$ is either increasing or decreasing. 
Assume that the number of geodesic segments in $m$ is $\geq 2$.
Let $v$ be the vertex with the maximal $|f|$-value on $m$. 
Here, $v$ is good since $m$ is maximal. 
Suppose that the end vertex $v'$  of the first geodesic segment $e$ in $m$ next to $v$
has the same sign of the corresponding $f$-values. 
Then $e$ is good at $v$ and $v'$ by elementary hyperbolic 
geometry using the hyperbolic right triangle 
with vertices $v$ and $v'$ and 
the right angle on the perpedicular line to $\tilde g$ passing $v$. 
Then the perpendicular distance function  to $e$ is given by above Lemma \ref{lem:divergence} and hence $f$-values of $e$ are in $(-6.6\delta, 6.6\delta)$. 
Hence, so is $m$ since we have the deceasing or increasing function where 
$v$ has the maximum $|f|$-value. 
}

\newch{
Suppose that $f(v)$ and $f(v')$ have different signs. $v'$ is not a local minimum 
or a local maximum. 
Now consider  $m'$ given by $m$ 
with the edge $e^o$ and $v$ removed. 
Then the $f$-values have the same signs on $m'$ and the maximal $|f|$-value occurs 
at the other end which must be a good vertex also. 
Now the above applies and $f$-values on $m'$ are in $(-6.6\delta, 6.6\delta)$. 
For $e$, we use the hyperbolic triangle with the vertex $v$ and the two vertices 
that are perpendicular projections $v_1$ and $v'_1$ 
of $v$ and $v'$ on $\tilde g$ respectively. 
Let $e'$ be the edge opposite $v_1$. 
Now $e'$ is good at $v$ and $v'$ since the angle sum of the triangle must be $< \pi$.
Lemma \ref{lem:divergence} shows that 
$f(v) \in (-6.6\delta, 6.6\delta)$. 
Since $f$ on $e$ is strictly decreasing or increasing, we have the result for $e$.
}

\newch{
We do these processes of estimation for such maximal subarcs.  
A segment $e$ with a local minimum of $|f|$ in its interior can occur after the process ends. The vertices of $e$ can be a vertex of such maximal subarcs or 
a good vertex. 
We need to work with quadrilateral
obtained by projecting $e$ to $\tilde g$ and the corresponding sides. 
We can use a reflection by the geodesic containing the shortest segment 
between $e$ and its projection to $\tilde g$ and compare. 
We can show that 
either both angles at $v$ and $v'$ satisfy the premises of Lemma \ref{lem:divergence} 
or  $|f|$-values are both less than $6.6\delta$. 
Since the perpendicular distance functions have good vertices or a segment with 
an interior local minimum of $|f|$ in between. 
}

\newch{Suppose that the number of segments in $m$ is $1$. 
Take $v$ as above. 
If both endpoints are good with respect to the segment, then we are done
by Lemma \ref{lem:divergence}. 
Otherwise, we can extend this segment at the other end point which is not good. 
If $|f|$ becomes zero, then we can use as above 
the right triangle with the hypothenuse obtained by 
extending the segment until $|f|$ becomes zero.
If not, then there is a local minimum point where we can directly use 
Lemma \ref{lem:divergence}. 
}

\newch{
The last item follows by using the divergence function. 
We obtain the bounds by (ii) of Lemma \ref{lem:divergence}.
}
	\end{proof}

\begin{lemma}\label{lem:vertical}
Let $l$ be a maximal geodesic in a horodisk $B$ in  the upper half-space model
given by $y > 1$. 
Suppose that the difference of the $x$-coordinates of the endpoints is $t$. 
Then the angle $\theta$ that $l$ makes with the vertical line
satisfies $\theta(t) = \pi/2 - \arctan(t/2)$.
Also, $t \mapsto t \theta(t)$ is a strictly increasing function for $t \in (0, \infty)$. 
$t \theta(t) < 2$, and the limit is $2$ as $t \ra \infty$. 
	\end{lemma} 
\begin{proof} 
	The lemma follows from elementary geometry since the geodesics are circles
perpendicular to $y=0$ \newch{in the upper half-space model}. 
(See \cite{Calculus}.)
	\end{proof}

\subsection{Hausdorff limits} \label{sub:Hlimit}
The projective sphere $\SI^3$ is a compact metric space, 
and has a natural standard metric \hypertarget{term-bdd}{$\bdd$}.  
For a compact set $A \subset \SI^3$, we define 
\[\bdd(x, A) = \inf\{\bdd(x, y) | y \in A \}.\] 
We define the $\eps$-$\bdd$-neighborhood 
$N_{\bdd, \eps}(A):= \{x| \bdd(x, A) < \eps \}$ for a point or 
a compact set $A$. 
We define the \hypertarget{term-hm}{{\em Hausdorff distance}} between two compact sets $A$ and $B$
as follows: 
\[ \bdd_{H}(A, B) = \inf\{\delta| \delta> 0, B \subset N_{\bdd, \delta}(A), A \subset N_{\bdd, \delta}(B) \}. \]

A sequence $\{A_{i}\}$ of compact sets 
{\em converges} to a compact subset $A$ if $\{\bdd_{H}(A_{i}, A)\} \ra 0$. 
The limit $A$ is characterized as follows if it exists: 
\[ A := \{a\in \SI^3|\, a \hbox{ is a limit point of some sequence } \{a_i| a_i \in A_i\} \}.   \]
See Proposition E.12 of \cite{BP92} for proof of this and Proposition \ref{prop:BP} 
since the Chabauty topology for compact spaces is the Hausdorff topology. 
(See Munkres \cite{Munkres75} also.)

\begin{proposition}[Benedetti-Petronio \cite{BP92}] \label{prop:BP} 
	A sequence $\{A_i\}$ of compact sets converges to $A$ in the Hausdorff topology if and only if 
	both of the following hold\/{\em :}
	\begin{itemize} 
		\item If there is a sequence $\{x_{i_{j}}\}$, $x_{i_j} \in A_{i_j}$, where $x_{i_j} \ra x$ for 
		$i_j\ra \infty$, then $x \in A$. 
		\item If $x \in A$, then there exists a sequence $\{x_{i}\}$, $x_i \in A_i$, such that $\{x_i\} \ra x$. 
		\end{itemize} 
\end{proposition}

Immediately we obtain
\begin{corollary}\label{cor:geoc} 
	Suppose that a sequence $g_i$ of projective automorphisms of $\SI^3$
	converges to a projective automorphism $g$, 
	and $\{K_i\} \ra K$ for a sequence $K_i$ of compact sets. 
	Then $\{g_i(K_i)\}  \ra g(K)$.
\end{corollary}

For example, a sequence of closed hemispheres will have a subsequence converging to 
a closed hemisphere.

\subsection{ The Poincar\'e polyhedron theorem} \label{sub:poincare} 

\begin{definition} \label{defn:matching} 
Let $\tilde N$ be an oriented manifold with empty or nonempty boundary on which a free group $\Gamma$ acts properly and freely. 
	Let $\mathcal{S}$ be a finite generating set $\{\gamma_{1}, \dots, \gamma_{2{\mathbf{g}}}\}$ in $\Gamma$
	with $\gamma_{i+\mathbf{g}} = \gamma_{i}^{-1}$ for indices in $\bZ/2{\mathbf{g}}\bZ$. 
	The collection of codimension-one submanifolds 
	$A_{1}, \dots, A_{2{\mathbf{g}}}$ 
	satisfying the following properties is called a {\em matching collection of sets} under $\mathcal{S}$:
	\begin{itemize}
	\item $\tilde N$ is a union of two submanifolds $\tilde N_{1}$ and 
	$\tilde N \setminus \tilde N_1^o$ 
	with $A_{1}\cup \dots \cup A_{\mathbf{g}} \subset \Bd_{\tilde N} \tilde N_{1}$ 
	for $i \in \bZ/2{\mathbf{g}}\bZ$. 
	\item $A_{i}$ is oriented by the boundary orientation from $\tilde N_{1}$. 
		\item $\gamma_{i}(A_{i}) = A_{i+{\mathbf{g}}}$ for $i\in \bZ/2{\mathbf{g}}\bZ$, 
		\item $ \gamma_{k}(A_{l})\cap A_{m} = \emp$ for $(k,l,m) \ne (i, i, i+{\mathbf{g}})$, and
		\item $\gamma_{i}$ is orientation-preserving for each $i \in \bZ/2{\mathbf{g}}\bZ$ and is orientation-reversing for $A_i$ and $A_{i+\mathbf{g}}$. 
	\end{itemize} 
	\end{definition} 

The following is a version of the Poincar\'e polyhedron theorem.
We generalize Theorem 4.14 of Epstein-Petronio \cite{EP94}.
Here, we drop their distance lower-bound conditions, without which 
we can easily find counter-examples. 
However, we replace the condition with exhaustion by compact submanifolds where 
the lower-bounds hold. Thus, we give a proof.
But we did not fully generalize the theorem by allowing 
sides of codimension $\geq 2$. 
\begin{proposition}[Poincar\'e] \label{prop:Poincare}
	Let $N$ be a connected manifold with empty or nonempty boundary
	covered by a manifold $\tilde N$ with a free deck transformation group $\Gamma_N$. 
	\begin{itemize} 
	\item Let $F$ be a connected codimension-zero submanifold with 
	boundary in $\tilde N$ that is 
	a union of mutually-disjoint, codimension-one, properly-embedded, 
	two-sided submanifolds
	 $A_1, \dots, A_{2{\mathbf{g}}}$ with boundary orientation. 
	\item Let $N_i \subset N$, $i=1, 2, \dots$, be an exhausting sequence of 
	compact submanifolds of $N$
	where $N_i \subset N_{i+1}$ for $i=1, 2, \dots$, 
	and the inverse image $\tilde N_i$ of $N_i$ 
	in $\tilde N$ is connected. 
	\item Let $\mathcal{S}$ be a finite generating subset of $\Gamma_N$ and 
	$\{A_1, \dots, A_{2{\mathbf{g}}}\}$ is matched under $\mathcal{S}$.
	\item $F \cap \tilde N_i$ is compact, and $F \cap \tilde N_i \cap A_j \ne \emp$ 
	for each $i$ and $j$. 
	\end{itemize} 
	Then $F$ is a fundamental domain of $\tilde N$ under $\Gamma_N$. 
\end{proposition}
\begin{proof} 
	We define $X' := \bigsqcup_{\gamma \in \Gamma_N} \gamma(F)/\sim$ 
	where we introduce an equivalence relation $\sim$ on
	$\bigsqcup_{\gamma \in \Gamma} \gamma(F)$
	given by 
	\begin{equation}
	x \in \gamma_{1}(F) \sim y \in \gamma_{2}(F) \hbox{ iff }  
	\begin{cases} 
	\, x =y \hbox{ and } \gamma_{1}\gamma_{2}^{-1}\in \mathcal{S}  
	\hbox{, or else } \nonumber \\
	\, x = y  \hbox{ and } \gamma_{1} = \gamma_{2}. 
	\end{cases}
	\end{equation}
	Thus, 
	\[X':=\bigsqcup_{\gamma \in \Gamma} \gamma(F)/\sim\] 
	is an open manifold	immersing into $N$. 
	We give a complete Riemannian metric on $N$ where each $\partial N_i$ is strictly convex.
	This induces a $\Gamma$-invariant Riemannian path-metric on $X'$ 
	and one on $F$. 
	
	Let $F_i = \tilde N_i \cap F$, a compact submanifold 
	bounded by $A_j \cap \tilde N_i$ for $j=1, \dots, 2{\mathbf{g}}$ by a generic perturbation of $N_i$ by small amounts. 
		We define $X_i' := \bigsqcup_{\gamma \in \Gamma_N} \gamma(F_i)/\sim$ 
	where we introduce an equivalence relation $\sim$ on
	$\bigsqcup_{\gamma \in \Gamma} \gamma(F_i)$
	given by 
	\begin{equation}
	x \in \gamma_{1}(F_i) \sim y \in \gamma_{2}(F_i) \hbox{ iff }  
	\begin{cases} 
	\, x =y \hbox{ and } \gamma_{1}\gamma_{2}^{-1}\in \mathcal{S}  
	\hbox{, or else } \nonumber \\
	\, x = y  \hbox{ and } \gamma_{1} = \gamma_{2}. 
	\end{cases}
	\end{equation}
	
	We restrict the above Riemannian metric to $X'_i$ as a submanifold of 
	$X'$ and obtain a $\Gamma$-invariant path metric $d_i$. 
	We claim that $d_i$ is metrically complete:
	Since $F\cap \tilde N_i$ is compact by the premise, 
	$A_j \cap \tilde N_i$ is a compact subset. 
	For every point in $x \in A_j \cap \tilde N_{i}$, the pathwise $d_i$-distance in $\tilde N_{i}$ to 
	$A_k \cap \tilde N_{i}$, $k \ne j$ is bounded below by a positive number $\delta_i$.
	Hence, each point of $X'_i$ has a normal $d_i$-ball $B'_i$ of fixed radius $\delta_{i}$
	in the union of at most two images of $F$ mapping isometric 
	to a $\delta_{i}$-$d_i$-ball $B_i$ in $N_i$. 
 Thus, given any Cauchy sequence $x_i$ in $X'_i$, 
 suppose that 
 \[d_i(x_k, x_l) < \delta_{i}/3 \hbox{ for } l, k > L \hbox{ for some } L.\] 
 Then $d_i(x_j, x_{L+1}) < \delta_{i}/3$ for $j > L$. 
 Since the ball of radius $\delta_{i}/3$ is in a union of two compact 
 sets,  $x_i$ converges to a point of the $\delta_{i}$-$d_i$-ball with center 
 $x_{L+1}$. 
	Hence, $X'_i$ has a metrically complete path-metric $d_i$. 
	
	There is a natural local isometry $X'_i \ra \tilde N_i$ given by 
	sending $\gamma(F_i)$ to $\gamma(F_i)$ for each $\gamma$. 
	Since $\{\gamma(F_i)|\gamma \in \Gamma \}$
	is a locally finite collection of 
	compact sets in $\tilde N_i$, the map is proper. 
The image in $\tilde N_i$ is open since each $\delta_i$-ball 
is in the image of at most two sets of the form $\gamma(F_i)$. 
Since $\tilde N_i$ is connected, the openness and closedness 
show that $X'_i \ra \tilde N_i$ is surjective. 
	Therefore, $X'_i \ra N_i$ is a covering map 
	being a proper local homeomorphism. 
	Now, 
	$\tilde N_i$ and $X'_i$ are covers of $N_i$ 
	with the identical deck transformation groups. We conclude
	$ X'_i \ra \tilde N_i $
	is a homeomorphism. 
	
	There is a natural embedding $X'_i \ra X'$. We identify $X'_i$ with its image. 
	We may identify $X'$ with $\bigsqcup_{i=1}^{\infty} X'_i$. 
	Since $\tilde N = \bigcup_{i=1}^{\infty} \tilde N_i$ holds, 
	$X' \ra \tilde N$ is a homeomorphism, and $F$ is the fundamental domain. 
\end{proof}


\section{Margulis invariants and Charette-Drumm invariants} \label{sec:CD}

We will first discuss parabolic group action in Section \ref{sub:parabolic} and 
then discuss Charette-Drumm invariant ensuring their proper action 
in Section \ref{sub:propaffine}. In Section \ref{subsec:parabolic}, we will 
introduce the parabolic ruled surfaces in $\Lspace$ and the region bounded by them. 
We will also provide two transversal foliations on the regions.

\subsection{Parabolic action} \label{sub:parabolic}


\subsubsection{Understanding parabolic actions} 


Let $V$ be a Lorentzian vector space of $\dim_{\bR} V = 3$ with the inner product 
$\Bs$. 
A linear endomorphism $N: V\ra V$ is a {\em skew-adjoint endomorphism} of $V$
 if 
 \[ \Bs(N\vx, \vy) = - \Bs(\vx, N \vy).\]

\begin{lemma}\label{lem:lem1} 
Suppose that $N$  is a skew-adjoint endomorphism of $V$  and $\vx \in V$.  Then
$\Bs(N \vx, \vx) = 0$.
\end{lemma}
\begin{proof}
$\Bs(N \vx, \vx) = -\Bs(\vx, N \vx)  = -\Bs(N \vx, \vx)$ by symmetry.
Thus we obtain $\Bs(N \vx, \vx)  = 0$ as claimed. 
\end{proof}

\begin{lemma}\label{lem:lem2}
Suppose that $N$  is a nonzero  nilpotent  skew-adjoint  endomorphism.   Then
$\rank(N ) = 2$.
\end{lemma}
\begin{proof}  Since $N$ is nilpotent,  it is non-invertible and so $\rank(N ) < 3$.
We have $\rank(N ) > 0$.
Assume $\rank(N ) = 1$. Then $\dim \Ker(N ) = 3 - 1 = 2$. Since $\dim(V ) = 3$, one of the following holds:
$ N (V )  \cap \Ker(N ) = \{0\}$, or
 $ N (V )  \subset \Ker(N )$.
If $N (V ) \cap \Ker(N ) = \{0\}$, 
then the restriction  of $N$  to $N (V )$ is nonzero, contradicting nilpotency.  Thus, $N (V ) \subset \Ker(N )$, 
that  is, $N^{2}  = 0$.
Then there exists $\vv \in V$ with $N \vv \ne 0$. Since $N^{2} \vv = 0$, the set 
$\{\vv, N \vv\}$ is linearly independent.  Complete $\{N \vv\}$ to a basis $\{N \vv, \vw\}$ of $\mathrm{Ker}(N )$. The set 
$\{\vv, N \vv, \vw\}$ is a basis for $V$.

\begin{itemize} 
\item Lemma \ref{lem:lem1} implies $\Bs(N \vv, \vv) = 0$.
\item $N^2  = 0$ implies $\Bs(N \vv, N \vv) = -\Bs(N^2 \vv, \vv) = 0$.
\item $\Bs(N \vv, \vw) = -\Bs(\vv, N \vw) = 0$ since $N \vw = 0$.
\end{itemize} 
Thus, $N \vv$ is a nonzero vector orthogonal to all of $V$, contradicting nondegeneracy.
Hence, $\rank(N ) = 2$ as claimed. 	

\end{proof}

\begin{lemma}\label{lem:lem3}  $N^{2}  \ne 0$.
\end{lemma} 
\begin{proof}  Lemma \ref{lem:lem2} implies that  $\dim\Ker(N ) = 1$ and $\dim N(V) = 2$. If $N^{2}  = 0$, then
$N (V ) \subset \Ker(N )$, a contradiction.
\end{proof} 


\begin{lemma}\label{lem:lem4} 
$ N (V ) = \Ker(N^{2})$ and $N^{2}(V ) = \Ker(N)$.
\end{lemma} 
\begin{proof} 
$\dim(V ) = 3$ and the nilpotency implies $N^3  = 0$. By Lemma \ref{lem:lem3}, the invariant flag
\begin{equation} \label{eqn:eqn1} 
V \supset N (V ) \supset N^{2} (V ) \supset \{0\} 
\end{equation}  
is maximal;  that  is,
$\dim V  /N (V )  = \dim N (V )/N^{2}(V )  = 1.$
Now, $N^{3}   = 0$ implies that  $N (V )  \subset \Ker(N^{2} )$ and  $N^{2}(V )  \subset \Ker(N )$.
Hence, the invariant flag
\begin{equation}\label{eqn:eqn2} 
V \supset \Ker(N^{2}) \supset \Ker(N) \supset \{0\}
\end{equation} 
is maximal.   
It follows that  the flags \eqref{eqn:eqn1} and \eqref{eqn:eqn2} are equal, as claimed. 
\end{proof} 

\begin{lemma}\label{lem:lem5}
$\Ker(N )$ is null.
\end{lemma} 
\begin{proof}  
Lemma \ref{lem:lem4} implies $\Ker(N ) = N^{2}(V )$. Since $N$  is skew-adjoint and $N^{4}  = 0$,
\[\Bs(N^{2}(V ), N^{2}(V ))  \subset \Bs(N^{3} (V ),N(V))   = \{0\} \]   
as desired. 
\end{proof} 

\begin{lemma}\label{lem:lem6}  
$\Ker(N ) = N (V )^{\perp}$ and $N (V ) = \Ker(N )^{\perp}$. 
\end{lemma} 
\begin{proof} 
$\Bs(N (V ), \Ker(N ))  = \Bs(V, N(\Ker(N )))  = \{0\}$
so that  $\Ker(N ) \subset N (V )^{\perp}$ and 
$N (V ) \subset \Ker(N )^{\perp}$.  
Since $\Ker(N )$ and $N (V )^{\perp}$ 
each have dimension $1$, and $N (V )$ and $\Ker(N )^{\perp}$ each have dimension $2$, the lemma
follows.	
\end{proof} 

We find a canonical generator  for the line $\Ker(N )$ given $N$, together  with a time-orientation.

\begin{lemma}\label{lem:lem7} 
 There exists unique $\vc \in \Ker(N )$ such that\/{\rm :}
 \begin{itemize}
\item $\vc \ne 0$ is a causal null-vector.
\item $\vc = N (\vb)$ for a unit-space-like $\vb \in V$  {\rm (}that  is, $\Bs(\vb, \vb) = 1${\rm ).}
\end{itemize} 
Furthermore, the following hold\/{\rm :} 
\begin{itemize}
\item  $\vb$ is unique up to addition of $\lambda\vc$, $\lambda \in \bR -\{0\}$.
\item We can choose the unique null vector $\va$ so that $N(\va) = \vb$.
\item  $\Bs(\va, \vb)=0=\Bs(\vb, \vc), \Bs(\va, \vc) = -1$. 
\item $\va, \vb, \vc$ form a basis. 
\item The Lorentz metric has an expression $g: = dy^{2} - 2  dx dz $
with respect to the coordinate system given by $\va, \vb, \vc$. 
\end{itemize} 
\end{lemma} 
\begin{proof}  
Lemma \ref{lem:lem4} implies that  $N$ defines an isomorphism (of one-dimensional vector spaces)
\begin{equation} \label{eqn:barN}
\bar N: N (V )/\Ker(N )  \ra   N^{2}(V )  = \Ker(N ).
\end{equation}

Now, 
$B| N(V) \times N(V)$ is factored to the maps
\begin{align*} 
& N(V) \times N(V) \ra N(V)/N(V)^\perp \times N(V)/N(V)^\perp \hbox{ and } \\
& \hat B:  N(V)/N(V)^\perp \times N(V)/N(V)^\perp  \ra \bR. 
\end{align*} 
Lemma  \ref{lem:lem6}  implies that  
the second map is 
\[\hat B: N(V)/\Ker N \times N(V)/\Ker N \ra \bR.\] 
Since $N(V)/\Ker N$ is a one-dimensional vector space, 
the quadratic map $\hat B$ is a square of an isomorphism 
$N(V)/\Ker N \ra \bR$. 
Hence, the restriction  to 
$N (V)$ of the quadratic  form $\vu \ra \Bs(\vu, \vu)$ 
is the square of an isomorphism 
$N (V )/\Ker(N ) \ra \bR$ composed with the quotient map
$N (V ) \ra N (V )/\Ker(N ).$

Recall $\dim \Ker(N)=1$. 
Since $\bar N$ is injective, 
the set of unit-space-like vectors in $N (V )$ is the union of two cosets of $\Ker(N )$, mapped  by $N$  
to two nonzero vectors in $\Ker(N)$. 
By Lemma \ref{lem:lem5}, the image is null. 
The image is a causal vector in $\Ker(N )$ or a non-causal vector in $\Ker(N )$. 
Take the causal one to be $\vc$. Since the image has only two vectors, 
$\vc$ is the unique one.

By \eqref{eqn:barN},
$\vb$ can be chosen to be any {in $N(V)$} in the coset of $\Ker(N)$, and hence $\vb$ can be changed 
to $\vb + c_{0} \vc$ since $\vc$ generates $\Ker(N)$. 

By Lemma \ref{lem:lem1}, $\Bs(\vb,\vc)=\Bs(N(\vc), \vc) = 0$.     

The subspace $N^{-1}(\vb)$ is a line since $\dim \Ker(N) = 1$
and is parallel to a null space and does not pass $0$ since $\vb \ne 0$. 
Hence, it meets a null cone at {the} unique point. Call this $\va$. 
By Lemma \ref{lem:lem1}, $\Bs(\va, \vb)=0$.          

Finally, \[\Bs(\va, \vc) = \Bs(\va, N^{2}(\va)) = -\Bs(N(\va), N(\va)) = -\Bs(\vb, \vb)= -1.\] 
The last statement follows by $\Bs$-values which also implies the independence. 
\end{proof} 

\begin{definition}\label{defn:adoptF} 
	Let $N$ be a nilpotent skew adjoint endomorphism. 
We will call the frame $\va, \vb, \vc$ satisfying the above properties: 
\begin{itemize} 
\item $\vb = N(\va), \vc = N(\vb)$.
\item $\va, \vc$ are null and $\vb$ is of unit space-like.  
\item $\Bs(\va, \vb)=0=\Bs(\vb, \vc), \Bs(\va, \vc) = -1$. 
\end{itemize}
the {\em adopted frame} of $N$.  
We will say that $N$ is {\em accordant} if the adopted frame has the standard orientation. 
\end{definition}
Corollary \ref{cor:parac} shows that associated with $N$, there is a one-parameter family of frames. 
However, we remark that the orientation of $\{\va, \vb, \vc\}$ is determined by $N$ as we can see 
from exchanging $N$ with $-N$ has the orientation-reversing effect. 



\begin{corollary} \label{cor:parac}
	Let $N$ be a nilpotent skew adjoint endomorphism. Then
the Lorentzian vectors $\va, \vb, \vc$ 
 satisfying the property that 
\begin{itemize} 
\item  $\Bs(\va, \vb)=0=\Bs(\vb, \vc), \Bs(\va, \vc) = -1$, 
\item $ \vc = N(\vb), \vb = N(\va)$, and 
\item $\vb$ is a unit space-like vector, $\vc \in \Ker N$ is {causally null},  and $\va$ is null
\end{itemize}
are determined up to changes 
$\vb \ra \vb + c_0 \vc, \va \ra \va + c_0 \vb + \frac{c_0^2}{2} \vc$ with respect to the  
a  skew-symmetric nilpotent endomorphism $N$ \oldch{and} $\Bs: V \times V \ra \bR$.  
Furthermore, the adopted frame for $N$ is determined only up to these changes and translations. 
\end{corollary} 
\begin{proof} 
By Lemma \ref{lem:lem7}, we can only change  
$\vb \mapsto \vb + c_{0}\vc, \va \mapsto \va + c_{0}\vb + d_{0}\vc$. 
Since $\Bs( \va + c_0 \vb + d_0 \vc,  \vb + c_{0}\vc) = -c_{0} + c_0=0$, and
\[\Bs(\va + c_{0}\vb + d_{0}\vc, \va + c_{0}\vb + d_{0}\vc)= c_{0}^{2} - 2 d_{0} =0,\] 
this is proved. 
\end{proof} 

\KFAR{
For example, suppose that
\[N = 
\left(
\begin{array}{ccc}
0  & 1  & 1  \\
-1  & 0  & 0  \\
1  & 0  & 0   
\end{array}
\right).
\]
is a  nilpotent  element  in $\mathrm{so}(2, 1)$ with  the  standard (diagonal) inner product.  Then we may choose
 \[\vc = 
\left(
\begin{array}{c}
 0      \\
  -1     \\
   1  
\end{array}
\right), 
\vb = 
\left(
\begin{array}{c}
 1     \\
  -z     \\
   z  
\end{array}
\right),
\va = 
\left(
\begin{array}{c}
z     \\
 y     \\
1-y  
\end{array}
\right)
\]
where $\vc=N(\vb), \vb= N(\va)$, 
and $\vc$ is the desired vector.
We can choose $\va$ to be so that $N(\va)= \vb$.
We choose $\va$ to be null by setting 
$y = \frac{1}{2}(1- z^2)$. 
The choice is canonical except for the choice of $z$. 
Here, the basis adapted  to the invariant flag is:
\begin{equation} \label{eqn:invflag}
\left(
\begin{array}{c}
 z     \\
  y     \\
   1-y  
\end{array}
\right)  \stackrel{N}{\longrightarrow}
\left(
\begin{array}{c}
 1    \\
  -z     \\
   z  
\end{array}
\right) \stackrel{N}{\longrightarrow}
\left(
\begin{array}{c}
 0    \\
  -1     \\
   1  
\end{array}
\right) \stackrel{N}{\longrightarrow}
\left(
\begin{array}{c}
 0    \\
  0  \\
  0 
\end{array}
\right).
\end{equation} 

}

\subsubsection{The action of the parabolic transformations} \label{sub:paraact} 

We represent an affine transformation with the formula
 $\vx \mapsto A \vx + \vw$, 
$\vx \in \bR^{2, 1}$ by 
the matrix 
\[
\left(
\begin{array}{cc}
 A   & \vw  \\
 0   & 1  
\end{array}
\right).
\]

{Let $N$ be an accordant  nilpotent element of the Lie algebra of $\SO(2, 1)$:} 
Let us use the frame $\vc, \vb, \va$ on $\Lspace$ obtained 
by Corollary \ref{cor:parac}
as the vectors parallel to $x$-, $y$-, and $z$-axes respectively. 
Then the bilinear form $\Bs$ takes the matrix form
\begin{equation}\label{eqn:Bform}  
 \left(
\begin{array}{ccc}
	0 & 0 & -1 \\
	0 & 1 & 0 \\
	-1 & 0 & 0
\end{array}
\right).
\end{equation}

Let $\gamma$ be a parabolic transformation $\Lspace \ra \Lspace$.
{Then it must be of the form 
\begin{equation}\label{eqn:Phit0}
	\Phi(t) := \exp t
	\left(
	\begin{array}{cc}
		N & \vec{v}_\gamma \\
		0 & 0 
	\end{array}
	\right) \hbox{ for an accordant nilpotent skew adjoint element } N. 
	\end{equation} }
{Using the frame given by Corollary \ref{cor:parac}
and shifting the origin by translation by $(t, v_1, v_2), t \in \bR$ when $\vec{v}_\gamma$ can be written 
as $(v_1, v_2, \mu)$ with respect to the frame,} 
we obtain an affine coordinate system so that $\gamma$ lies in a one-parameter group 
\begin{equation}\label{eqn:Phit}
 \Phi(t) := \exp t
\left(
\begin{array}{cccc}
0  & 1  & 0  &0\\
0  & 0  & 1  & 0\\
0  & 0  &  0 & \mu \\
0 &  0 & 0 & 0 
\end{array}
\right)
= 
\left(
\begin{array}{cccc}
 1 & t  & t^{2}/2 & \mu t^{3}/6  \\
0   & 1  & t  & \mu t^{2}/2  \\
0  &  0  & 1 & \mu t \\ 
0 & 0 & 0 & 1   
\end{array}
\right)
\end{equation} 
for $\mu \in \bR$
where
$\Phi(t): \Lspace \ra \Lspace$ is generated by 
a vector field 
\[ \phi:= y \partial_{x} + z \partial_{y} + \mu \partial_{z} \hbox{ where } 
\Bs(\phi, \phi) = z^{2} - 2 \mu y. \]


{For a parabolic element $\gamma$ and  $t \in \bR$, 
	we define $\gamma^t := \exp(t \eta)$ where $\gamma = \exp(\eta)$ for a unique Lie algebra element $\eta$ of  $\Isom^{+}(\Lspace)$.} 

\begin{definition} \label{defn:parac} 
	For any parabolic element $\gamma$, the coordinate system where it can be written in the form \eqref{eqn:Phit} with 
	the	adopted frame for accordant nilpotent $N$ where $\gamma = \exp(tN), t \in \bR$ is  
		called a {\em parabolic coordinate system adopted to $\gamma$.}
		Furthermore, $\gamma$ is called {\em accordant} if $t > 0$.
\end{definition}

\begin{proposition}\label{prop:paracoord} 
	Any parabolic element $\gamma$ has a parabolic coordinate system. 
	All other parabolic coordinate system for $\gamma$ is obtained by changing it by 
   a $2$-dimensional parameter family of isometries generated
	by the $1$-parameter family of translations along unique eigen-direction and the frame change given in Corollary \ref{cor:parac}. 
\end{proposition}

\begin{proof} 
	The existence of the coordinate frame is already given.
The fact that the $2$-dimensional family of isometries preserves the form \eqref{eqn:Phit} is already shown
in Corollary \ref{cor:parac} and near \eqref{eqn:Phit}. 
Also, from near \eqref{eqn:Phit} we obtain the translations must be the one-parameter ones along the unique eigen-direction.
\end{proof}

This one-parameter subgroup $\{\Phi(t), t \in \bR\}$ leaves invariant the two polynomials
\begin{align*} 
F_{2}(x, y, z) &:= z^{2}  - 2 \mu y \\
F_{3}(x, y, z)     &:= z^{3}  - 3\mu yz + 3 \mu^{2} x, 
\end{align*}
and the diffeomorphism $F (x, y, z) := (F_{3}(x, y, z), F_{2}(x, y, z), z)$ satisfies
\begin{equation} \label{eqn:Phicoor} 
F \circ \Phi(t) \circ F^{-1}  :
\left(x, y, z\right) \ra \left(x, y, z+\mu t \right)
\end{equation} 
All the orbits  are twisted cubic curves.  In particular, every 
cyclic parabolic group leaves invariant no line and no plane for $\mu \ne 0$.
\begin{figure}[h]

\begin{center}
\includegraphics[width=11cm]{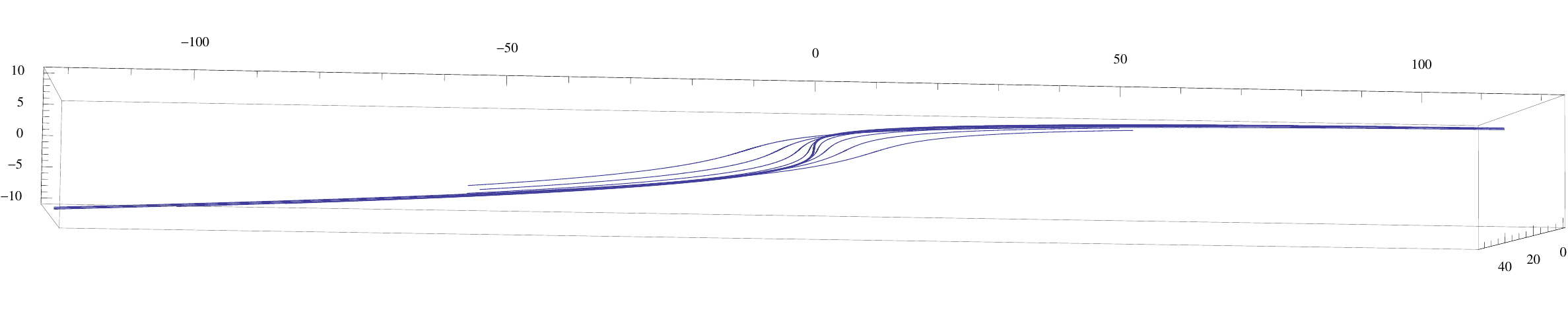}
\end{center} 
\caption{A number of orbits drawn horizontally.} 
\label{fig:cubicc}
\end{figure}

Now, $Q:= F_{2}$ is {the} unique quadratic 
$\phi$-invariant function on $\Lspace$ up to adding constants and scalar multiplications. 
If $Q(p) < 0$ for $p\in \Lspace$, then the trajectory $\Phi(t)(p)$ is time-like. 
If $Q(p) > 0$, then $\Phi(t)(p)$ 
is space-like. In addition, if $Q(p)=0$, then $\Phi(t)(p)$ is a null-curve. 
The region $Q < k$ is defined canonically for $\gamma$ for $k \in \bR$. 
($k$ can be negative.)  
The region is a parabolic cylinder in the parabolic coordinate system of $\gamma$. 
We will call this a \hypertarget{term-pc}{{\em parabolic cylinder}} for $\gamma$.


\KFAR{
We have $\gamma = \Phi(t_{1})$ for some $t_{1}>0$. 
\begin{equation} \label{eqn:Sp}
S'_{p} := \Phi(t_{1})(p) - p = 
\left(
\begin{array}{c}
 t_{1}y + (t_{1}^{2}/2)z + \mu t_{1}^{3}/6   \\
  t_{1}z + \mu t_{1}^{2}/2   \\
  \mu t_{1}   
\end{array}
\right)
\end{equation}
and $\Bs(S'_{p}, S'_{p}) = t^{2}(Q(p) - \mu^{2}t^{2}/12)$. 
} 


\begin{remark} \label{rem:canonical} 
	The expression \eqref{eqn:Phit} can change by conjugation by a dilatation so that $\mu {=\pm1}$.
	However, a dilatation is not a Lorentz isometry. 
%
\end{remark}


\begin{definition}\label{defn:semicircle}
A \hypertarget{term-sct}{{\em semicircle  tangent }} to $\partial \Ss_{+}$ at $p \in \partial \Ss_{-}$ is the closure of 
a component {of} $S\setminus \{p, p_{-}\}$ of the great circle $S$ tangent to $\partial \Ss_{+}$ at $p$
which does not meet $\Ss_{+}$. 
An \hypertarget{term-gs}{{\em accordant great segment}} $\zeta_{p}$ to {$\partial \Ss_{+}$} is an open semicircle tangent to $\partial \Ss_{+}$ 
starting from $x$ in the direction of the orientation of $\partial \Ss_{+}$. 
(See Section 3.4 of \cite{CG17}.)
\end{definition}
We may refer to them as being
{\em positively oriented} since 
we need to alter the construction when we change the orientation.

\begin{remark}\label{rem:accors}  
	In the parabolic coordinate system of $\Lspace$ for a parabolic $\gamma$, 
	 $\Ss_+$ is given by $\llrrparen{x, y, z, 0}$ in $\Ss$ with $y^2- 2xz < 0$ with $x > 0$. 
	Then 
	\[\overline{\llrrparen{1, 0, 0, 0}\llrrparen{0, 1, 0, 0}} 
	\cup \overline{\llrrparen{0, 1, 0, 0}\llrrparen{-1, 0, 0, 0}}\]
	easily shown to be the accordant great segment 
	$\clo(\zeta_{\llrrparen{1, 0, 0, 0}})$
	to the boundary of $\Ss_+$ with the induced orientation. 
\end{remark} 


	For the following if $\gamma$ is not accordant, we need to use $\gamma^{-1}$.

\begin{proposition} \label{prop:orbit} 
Let $\gamma$ be  {accordant parabolic transformation.}
	We use the parabolic coordinate system of $\gamma$ so that $\gamma$ is of the form \eqref{eqn:Phit} with $\mu > 0$.
Then the following hold\,{\rm :} 
\begin{itemize} 
\item $\langle \gamma \rangle$ acts properly on $\Lspace$. 
\item 
The orbit $\{ \gamma^{n}(p)\}$, $p\in\Lspace$, converges to 
{the} unique fixed point $x_\gamma$ in $\partial \Ss_+$
as $n \ra \infty$ 
and converges to its antipode $x_{\gamma -} \in \partial \Ss_-$ as $n \ra -\infty$. 
\item The orbit lies on the \hyperlink{term-pc}{parabolic cylinder} 
\[P_{p} := \{x \in \Lspace| Q(x) = Q(p)\},\] 
where $\gamma$ acts on. 
\item The set of lines in $\Lspace$ parallel to the vector $\vx_\gamma$ in 
the direction of $x_\gamma$ 
foliates each parabolic cylinder and gives us equivalence classes. 
$P_{p}/\sim$ can be identified with a real line $\bR$. The action of $\gamma$ on $P_{p}/\sim$ corresponds to 
a translation action on $\bR$. 
\item $P_{p}$ can be compactified to a compact subspace in $\SI^{3}$
homeomorphic to a $2$-sphere by adding 
the great segment 
$\clo(\zeta_{\mathbf{x}_\gamma})$
accordant to $\partial \Ss_+$. 
\end{itemize} 
\end{proposition} 
\begin{proof} 
%
We have $x_{\gamma}$ equal to $\llrrparen{1, 0, 0, 0}$ in this coordinate 
system. 	
The properness follows since  $\mu t^{3}/6$ dominates all other terms. The second item follows since $F_{2}$ is an invariant. 
Since $F_2$ is $\Phi_t$-invariant, $\gamma$ acts on the parabolic cylinder 
determined by $F_2$. The third item follows by projecting to the $z$-value. 
The fourth item is straightforward from the third item. 

Let $H_{0}$ be a great sphere given by $x=0$ in $\SI^{3}$. 
For each line $l$ in the parabolic cylinder, ${\{\gamma^{t}(\clo(l)) \cap H_{0}, t \in \bR\}}$
is a parabola compactified by a single point
$\llrrparen{0, 1, 0, 0}$ {as we can see using \eqref{eqn:Phit}.}
Let $H_{+}$ be the upper hemisphere bounded by $H_{0}$ and $H_{-}$ the lower hemisphere. 
We have geometric convergence\newch{:}
\begin{align}
\{{\gamma^{t}}(\clo(l)) \cap H_{+}  \}&\ra 
\overline{\llrrparen{1, 0, 0, 0}\llrrparen{0, 1, 0, 0}} \hbox{ as } 
t \ra \infty \hbox{ or } t \ra -\infty\newch{,}  \label{eqn:segment1}\\
\{{\gamma^{t}}(\clo(l)) \cap H_{-}\} &\ra 
\overline{\llrrparen{-1, 0, 0, 0}\llrrparen{0, 1, 0, 0}} \hbox{ as } t \ra \infty \hbox{ or } t \ra -\infty. 
 \label{eqn:segment2}
\end{align} 
Hence, by Remark \ref{rem:accors},  
\[\{{\gamma^{t}}(\clo(l))\} \ra \clo(\zeta_{\llrrparen{1, 0, 0, 0}})
\hbox{ as } t \ra \infty \hbox{ or } t \ra -\infty.\]
For any sequence of points $x_i$ on $P_p$, 
$x_i \in {\gamma^{t_i}}(\clo(l))$ for some $t_i \in \bR$. 
If $|t_i|$ is bounded, $\{x_i\}$ can accumulate only on $P_p$. 
If $|t_i|$ is unbounded, then $\{x_i\}$ can accumulate to 
$\clo(\zeta_{x_\gamma})$ by the above paragraph.
The final part follows. 
\end{proof}

\subsection{Proper affine deformations and Margulis and Charette-Drumm invariants} \label{sub:propaffine}

Let $\Sf$ be a complete orientable hyperbolic surface with $\chi(\Sf) < 0$ and possibly some cusps. 
Let $h:\pi_{1}(\Sf) \ra \SO(2, 1)^{o}$ be 
a discrete irreducible faithful representation. 
Now, the image is allowed to have parabolic elements. 
Each nonparabolic element $\gamma$ of $\pi_{1}(\Sf)\setminus\{\Idd\}$ is represented 
by {the} unique closed geodesic in $\Sf:=\Ss_{+}/h(\pi_{1}(\Sf))$ and hence 
is hyperbolic. 
Let $\Gamma$ be a proper affine deformation of $h(\pi_{1}(\Sf))$.
For nonparabolic $\gamma \in \Gamma\setminus\{\Idd\}$, 
we define 
\begin{itemize} 
\item $\vx_{+}(\gamma)$ as an eigenvector of $\mathcal{L}(\gamma)$ in the causally null directions with {the} eigenvalue $> 1$, 
\item $\vx_{-}(\gamma)$ as one of $\mathcal{L}(\gamma)$ with {the} eigenvalue $< 1$, and 
\item 
$\vx_0(\gamma)$ as a space-like positive eigenvector of $\mathcal{L}(\gamma)$ of {the} eigenvalue $1$
which is given by 
\[\vx_0(\gamma) = \frac{\vx_-(\gamma)\times \vx_+(\gamma)}{\llrrV{\vx_-(\gamma)\times \vx_+(\gamma)}}.\]
\end{itemize} 
Here, $\times$ is the Lorentzian cross-product, and
$\vx_+(\gamma)$ and $\vx_-(\gamma)$ are well-defined up to choices of sizes; 
however, $\vx_0(\gamma)$ is well-defined since it has a unit Lorentz norm. 
They define the Margulis invariant 
\begin{equation}\label{eqn:alpha}
\alpha(\gamma) = \Bs(\gamma(x) -x , \vx_0(\gamma)), x\in \Lspace 
\end{equation}
where the value is independent of the choice of $x$. 




In general, 
an {\em affine deformation} of a homomorphism $h:\pi_1(\Sf) \ra \SO(2, 1)$
is a homomorphism $h_{\bfb}:\pi_1(\Sf) \ra \Isom^{+}(\Lspace)$ given by 
$ h_{\bfb}(g)(x) = h(g)x + \bfb(g) $
for a cocycle $\bfb: \pi_1(\Sf) \ra \bR^{2, 1}$
in $Z^1(\pi_1(\Sf), \bR^{2, 1}_{h})$.  
The vector space of coboundary is denoted by $B^1(\pi_1(\Sf), \bR^{2, 1}_{h})$. 
As usual, we define 
\[H^1(\pi_1(\Sf), \bR^{2, 1}_{h}):= \frac{Z^1(\pi_1(\Sf), 
	\bR^{2, 1}_{h})}{B^1(\pi_1(\Sf), \bR^{2, 1}_{h})}. \]



Let $[\bfu]$ be the class of a cocycle in $H^{1}(\pi_{1}(\Sf), \bR^{2,1}_{h})$
with $\bfu \in Z^{1}(\pi_{1}(\Sf), \bR^{2,1}_{h}).$
Let $h_{\bfu}$ denote the affine deformation of $h$ according to a cocycle $\bfu$ in $[\bfu]$, 
and let $\Gamma_{\bfu}$ be the affine deformation $h_{\bfu}(\pi_{1}(\Sf))$. 
There is a function $\alpha_{\bfu}: \pi_{1}(\Sf)\setminus \mathcal{P}_{\pi_1(\Sf)}\ra \bR$ with the following properties:
\begin{itemize}
\item $\alpha_{\bfu}(\gamma^{n}) = |n| \alpha_{\bfu}(\gamma), n \in \bZ.$
\item $\alpha_{\bfu}(\gamma) = 0$ if and only if $h_{\bfu}(\gamma) $ fixes a point. 
\item The function $\alpha_{\bfu}$ depends linearly on $\bfu$. 
\item If $h_{\bfu}(\pi_{1}(\Sf))$ acts properly and freely on $\Lspace$, then 
$|\alpha_{\bfu}(\gamma)|$ is the Lorentz length of the unique space-like 
closed geodesic in 
$\Lspace/h_{\bfu}(\pi_{1}(\Sf))$ corresponding to $\gamma$. 
(See Goldman-Labourie-Margulis \cite{GLM09}.)
\end{itemize}

Charette and Drumm generalized the Margulis invariants for parabolic elements in \cite{CD05},
where the values are given only as ``positive'' or ``negative''.
Let $g\in \Gamma$ be a parabolic or hyperbolic element of an affine deformation of a linear group in $\SO(2, 1)^o$. 

\begin{definition} \label{defn:CD}
	An eigenvector $\vv$ of eigenvalue $1$ of a linear hyperbolic or parabolic transformation 
	$g$ is said to be {\em positive} relative to $g$ if 
	$\{\vv, \vx, \mathcal{L}(g)\vx\}$ is positively oriented when $\vx$ is any null or time-like vector which is not an eigenvector of $g$. 
\end{definition} 

It is easy to verify $\vv$ is positive with respect to $g$ if and only if $-\vv$ is positive with respect to $g^{-1}$. 
Let $F(\mathcal{L}(g))$ be the oriented one-dimensional space of eigenvectors of $\mathcal{L}(g)$ of eigenvalue $1$. 
Define $\tilde \alpha(\gamma): F( \mathcal{L}(\gamma)) \ra \bR $ by 
\[\tilde \alpha(\gamma)(\cdot) = \Bs(\gamma(x) - x, \cdot) \] where
$x \in \Lspace$ is any chosen point. 
Drumm \cite{Drumm92} also shows
\begin{equation} \label{eqn:alphacd} 
\alpha(\gamma) = \tilde \alpha(\gamma)(\vx^0(\mathcal{L}(\gamma)))
\end{equation} 
provided $\gamma$ is hyperbolic. 

By Definition \ref{defn:CD}, 
components of $F( \mathcal{L}(\gamma)) \setminus\{0\}$ have well-defined signs.
We say that the \hypertarget{term-CD}{{\em Charette-Drumm invariant}} $cd(\gamma) >0$ if 
$\tilde \alpha(\gamma)$ is positive on positive eigenvectors in
$F( \mathcal{L}(\gamma)) \setminus\{0\}$. 

Also, we note $cd(\gamma)> 0$ iff $cd(\gamma^{-1})>0 $. 

\begin{lemma}[Charette-Drumm \cite{CD05}]  \label{lem:CD}
	Let $\gamma \in \Gamma$ be a parabolic or hyperbolic element. 
	\begin{itemize} 
		\item $\tilde \alpha(\gamma)=  \Bs(\gamma(x) - x, \cdot)$ is independent of the choice of $x$. 
		\item $\tilde \alpha(\gamma) = 0$ if and only if $\gamma$ has a fixed point in $\Lspace$. 
		\item For any $\eta \in \Aff(\Lspace)$, $\tilde \alpha(\eta \gamma \eta^{-1})(\eta(\vv)) = \tilde \alpha(\gamma)(\vv)$ for 
		$\vv \in F(\mathcal{L}(\gamma))$. 
		\item For any $n \in \bZ$, $\vv \in F(\mathcal{L}(\gamma))$, $\tilde \alpha(\gamma^n)(\vv) = |n| \tilde \alpha(\gamma)(\vv)$. 
	\end{itemize} 
\end{lemma}


In the \hyperlink{term-pcs}{parabolic coordinate system} of $\gamma$, 
we obtain 
\begin{equation}\label{eqn:alphap} 
\tilde \alpha(\gamma)(x, 0, 0) = - \mu t x  
\end{equation}
for $\mu, t$ {given} for $\gamma$ as in 
\eqref{eqn:Phit} 
in Section \ref{sub:parabolic}. 

\begin{lemma} \label{lem:pex} 
	Let $\gamma$ be defined by \eqref{eqn:Phit} for \oldch{$t > 0$}
	in the  \oldch{accordant} parabolic coordinate system for $\gamma$. 
	Then the following holds: 
	\begin{itemize} 
	\item $\mu  > 0$ if and only if
	$\gamma$ has a positive Charette-Drumm invariant.
	\item $\mu  < 0$ if and only if
	$\gamma$ has a negative Charette-Drumm invariant.
	\item $\mu \ne 0$ if and only if 
	$\langle \gamma \rangle$ acts properly on $\Lspace$. 
	\end{itemize} 
\end{lemma}
\begin{proof} 
\oldch{	We prove the first item: 
	Choose	$\vx = (a, 0, c)$ with $ac > 0, a > 0$ 
	so that $\vx$ is a {causal} time-like vector. 
	Then $\{\vi, \vx, \mathcal{L}(\gamma)\vx\}$ 
	is a negatively 
	oriented frame, and $\vi$ is the negative null eigenvector of 
	$\mathcal{L}(\gamma)$ by Definition \ref{defn:CD}. 
	By \eqref{eqn:alphap}, the first item follows. 
The second item follows by the contrapositive of the first item. 
The final part follows by Proposition \ref{prop:orbit} and Lemma \ref{lem:CD}
and reversing the orientation of $\Lspace$. 
} 
	
\end{proof}

\subsection{Parabolic region and two transversal foliations on them} \label{subsec:parabolic}

\begin{figure}[h]
	
	\begin{center}
		\includegraphics[height=5cm]{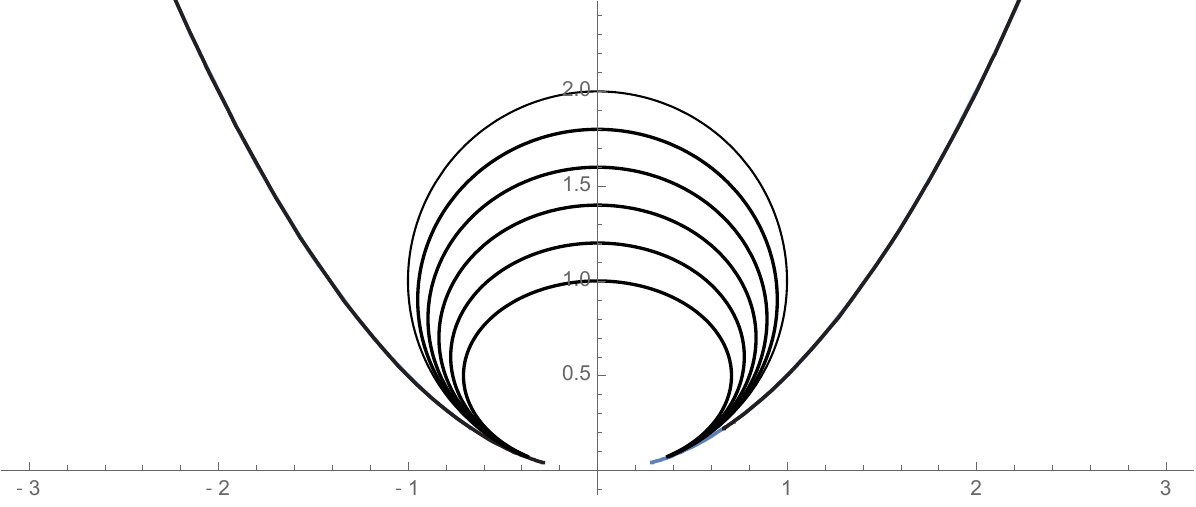}
	\end{center} 
	\caption{ This shows the projective action of a one-dimensional parabolic group 
		on $\Ss_{+}$ with boundary represented as a parabola.
		We use the affine patch where $x>0$ in the coordinate system.
		We normalize the homogeneous coordinates by setting $x =1$. 
		The parabola $2z = y^2$ describes the boundary of 
		$\Ss_+$ given by $y^2 < 2 xz$. 
		See Remark \ref{rem:accors}. } 
	\label{fig:parabolic}
\end{figure}

\subsubsection{Parabolic regions} 

Let $g$ be a parabolic element with the expression \eqref{eqn:Phit} for $t > 0$
under the {parabolic} coordinate system of Section \ref{sub:paraact}. 
Assume that the \hyperlink{term-CD}{Charette-Drumm invariant} of $g$ is positive. That is, $\mu > 0$ by Lemma \ref{lem:pex}. 
Recall from Section \ref{sub:paraact} that 
\[F_{2}(x, y, z) =z^{2} - 2 \mu y  \hbox{ and } F_{3}(x, y, z)     = z^{3}  - 3\mu yz + 3 \mu^{2} x \] 
are invariants of $g^{t}$.  
Recall that  $\Phi(t): {\mathbf{E}} \ra {\mathbf{E}}$ is generated by 
a vector field 
\[ \phi:= y \partial_{x} + z \partial_{y} + \mu \partial_{z} \]
with the square of the Lorentzian norm $\llrrV{\phi}^{2} = z^{2} - 2 \mu y$. 

The equation $F_{2}(x, y, z) = T$ gives us a parabolic cylinder $P_T$ in the $x$-direction 
with the parabola in the $yz$-plane. 
The vector field $\phi$ satisfies
\[\phi(x, y_{0}, 0) = (y_{0}, 0, \mu) \hbox{ for all } x \hbox{ and } T = -2 \mu y_{0}.\]

Since we are looking for a $g^{t}$-invariant ruled surface, 
we take a line $l$ tangent to $P_T$ in the direction of $\vx=(a, 0, c)$ starting at $(0, y_{0}, 0)$. 
Since $\llrrparen{ \vx}\,\, \in \Ss_{+}$ by \newch{the} premise, 
we obtain $2ac > 0$ with $a > 0, c> 0$ under the parabolic coordinate system 
with the quadratic form \eqref{eqn:Bform}.  
(See Figure \ref{fig:parabolic}.)

We define $\Psi(t, s) = g^{t}(l(s))$ so  that 
\[l(s) = (0, y_{0}, 0) + s(a, 0, c)= (sa, y_{0}, sc), \phi(l(s)) = (y_{0}, sc, \mu).\] 
Thus, $\phi$ is never parallel to $(a, 0, c)$ unless $s=0$. 
We choose $(a, 0, c)$, $c \ne 0$, not parallel to $(y_{0}, 0, \mu)$, i.e., 
\[ \frac{a}{c} \ne \frac{y_{0}}{\mu}.\] 
Then $\phi| l$ is never parallel to the tangent vectors to $l$. 
Since $Dg^t(\phi) = \phi$, $\phi$ is never parallel to tangent vectors 
to $g^t(l)$, it follows that $\Psi$ is an immersion in $\Lspace$.

Let $\mathcal{H}_{s_0, \kappa_{1}, \kappa_{2}}$ be the space of compact segments $u$ passing $\Lspace$ with 
the following property: 
\begin{itemize} 
	\item $u$ has an antipodal pair of endpoints in $\Ss_+$ 
	and in the antipodal set $\Ss_{-}$ and
	\item $u\cap \Lspace$ is equivalent under $g^t$ for some $t$ to a line $l(s)$ given by 
	\begin{multline} \label{eqn:ls} 
		l(s) = (sa, y_0, sc) \hbox{ for } y_0\geq s_0, a, c> 0, \frac{\kappa_{1}a}{c} \leq \frac{y_0}{\mu} \leq \frac{\kappa_{2}a}{c}, \hbox{ and }  \\
		a^2+ c^2 =1  \hbox{ for some pair } 0 < \kappa_{1}\leq \kappa_{2} < 1
		\hbox{ and } s_0 > 0. 
	\end{multline} 
\end{itemize} 
This space has a metric 
coming from the Hausdorff metric $\bdd_H$.

{We will prove the following in Appendix \ref{app:A}.} 

\begin{theorem}\label{thm:ruled} 
	Let $g$, $\mathcal{L}(g) \in \SO(2, 1)^{o}$, be {an accordant} parabolic element acting properly on $\Lspace$ 
	{with the positive Charette-Drumm invariant}.
	Let $l$ be a line in $\mathcal{H}_{s_0, \kappa_1, \kappa_2}$
	for the \hyperlink{term-pcs}{parabolic coordinate system} for $g$.
	Then   
	\begin{itemize}
		\item For each time-like line $l$ in the ruling of $S$, 
		\[\{g^{t}(\clo(l))\} \ra \clo(\zeta_{{\mathbf{x}_{\infty}}}) \hbox{ as } t \ra \infty
		\hbox{ and }  t \ra -\infty\]
		geometrically.
		\item For any \hyperlink{term-bdd}{$\eps$-$\bdd$}-neighborhood $N$ of $\clo(\zeta_{{\mathbf{x}_{\infty}}}) \subset \Ss$, 
		we can find such a ruled surface $S$ in $N \cap \Lspace$. 
		\item there exists a $\{g^{t}, t \in \bR\}$-invariant surface $S$ ruled by time-like lines containing $l^o$ 
		properly embedded in $\Lspace$ with boundary 
		\[ \clo(S) \setminus S = \{g^{t}({\mathbf{x}})| t \in \bR\} \cup \{g^{t}({\mathbf{x}}_{-})| t \in \bR\} \cup \clo(\zeta_{{\mathbf{x}}_{\infty}})\] 
		for a point ${\mathbf{x}} \in \Ss_{+}$, and ${\mathbf{x}}_{\infty}$ is 
		a parabolic fixed point of $g$ in $\partial \Ss_{+}$
		respectively. 
		{Furthermore, there exists a domain $R$ homeomorphic to a $3$-cell in $\Lspace$ whose topological boundary in the hemisphere $\mathcal{H}$ equals $\clo(S)$.
			Also, $R/\langle g \rangle$ is homeomorphic  to a solid torus.}
	\end{itemize} 
\end{theorem}

\begin{definition} \label{defn:para} 
	In Theorem \ref{thm:ruled},  
	the surface denoted  by $S$ is called a {\em parabolic ruled surface}. 
	(Compare with parabolic cylinders in Section \ref{sub:paraact}.)
	The open region $R$ in $\Lspace$ bounded by a parabolic ruled surface 
	is called the \hypertarget{term-pbr}{{\em parabolic region}}.
	The generator of the parabolic group acting on a parabolic ruled surface 
	fixes a point $p \in \partial \Ss_{+}$.
	
	An immersed image $S/\langle g \rangle$ 
	of the surfaces in a manifold $\Lspace/\Gamma$ is also 
	called a {\em parabolic ruled surface}. 
	The embedded image $R/\langle g \rangle$ of $R$ in a manifold $\Lspace/\Gamma$ is called a
	{\em parabolic region}. 
	
	
	We can choose the  parabolic surface and the parabolic regions so that they are in 
	the $\eps$-$\bdd$-neighborhood $N$ of $\bigcup_{x\in a}{\clo(\zeta_{x})} \subset \Ss$ by the last item of Theorem \ref{thm:ruled}. 
	Then we call the parabolic region \hypertarget{term-far}{{\em $\frac{1}{\eps}$-far away}} from the compact parts. The isometrically embedded images of such surfaces in 
	$\Lspace/\Gamma$ or $\Lspace$ are described in the same manner.

\end{definition}

\subsubsection{Two transversal foliations} 

\begin{figure}[h!] 
	
	%
	
	\includegraphics[height=8.0cm, clip, trim={3.5cm 0.0cm 4.5cm 0.0cm}]{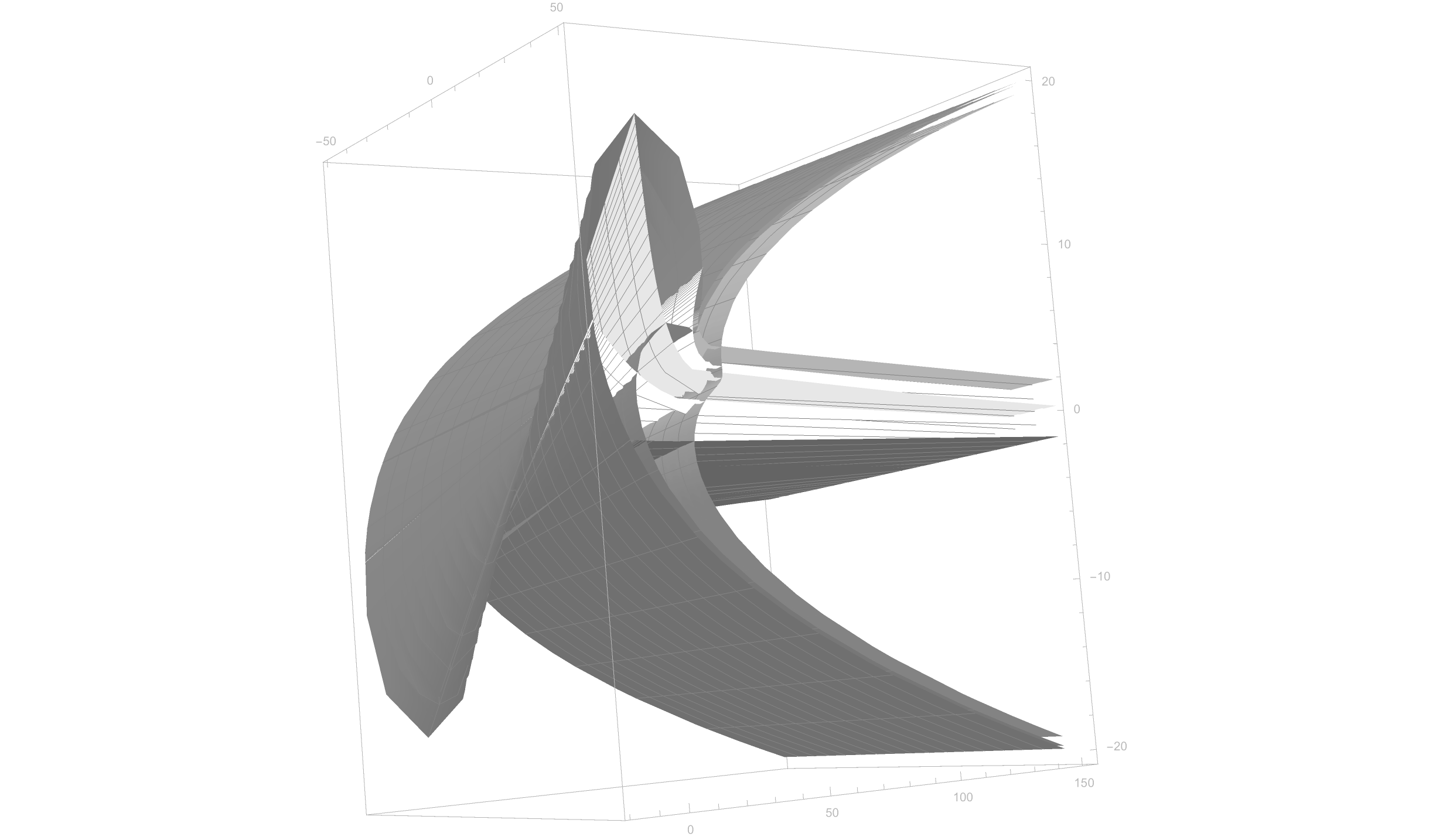}
	
	\caption{ Three darker leaves of foliation $\mathcal{S}_{f, r_{0}}$ and \newch{five}  transversal light-gray leaves of 
		$\mathcal{D}_{f, r_{0}}$ where 
		$f(\rho) = \frac{3}{4}\frac{r}{\sqrt{1-r^{2}}}$ and $\mu=1$.  See \cite{FoliationsV2}.}
	
\end{figure}

Assume \[0 < \kappa_{1}\leq \kappa_{2} < 1.\]
Let $f:(0, 1) \ra \bR$ be a strictly increasing smooth function satisfying 
\[\kappa_{1} \mu \frac{r}{\sqrt{1-r^{2}}}  \leq f(\rho) \leq \kappa_{2}\mu \frac{r}{\sqrt{1-r^{2}}}.\]
Let $\mathcal{H}_{f}$ be the space of compact segments $u$ passing $\Lspace$ with 
the following property: 
\begin{itemize} 
	\item $u$ has an antipodal pair of endpoints in $\Ss_+$ 
	and in  $\Ss_{-}$, 
	\item $u\cap \Lspace$ is equivalent under $g^t$ for some $t$ to a line $l(s)$ given by 
	$l_{f, r}(s) = (sa, y_{f}(\rho), sc), s\in \bR$, where  
	\[ y_{f}(\rho):=  f(\rho), a = r, c = \sqrt{1-r^{2}},  r \in (0, 1).\] 
\end{itemize} 
For fixed $r \in (0, 1 )$, 
let $S_{f, r}$ denote the parabolic ruled surface given by 
\[\bigcup_{t, s \in \bR}g^{t}(l_{f, r}(s)).\] 
Define $D_{f, r_0, t}$ for fixed $t \in \bR$ denote the surface 
\[ \bigcup_{ s \in \bR, r \in [r_{0}, 1)}g^{t}(l_{f, r}(s)).\]

{We will prove the following in Appendix \ref{app:A}.} 

\begin{theorem} \label{thm:Sr} 
	Let $r_{0}\in (0, 1)$. Then the following hold\/{\rm :} 
	\begin{itemize} 
		\item The surfaces $S_{f, r}$ for $r \in [r_{0}, 1)$ are properly embedded 
		leaves of a foliation $\tilde{\mathcal S}_{f, r_{0}}$ of the region $R_{f, r_{0}}$, closed in $\Lspace$, bounded by $S_{f, r_{0}}$
		where $g^{t}$ acts on. 
		\item $\{D_{f, r_0, t}, t \in \bR\}$ is the set of properly embedded 
		leaves of a foliation $\tilde{\mathcal{D}}_{f, r_{0}}$  of $R_{f, r_{0}}$ by disks
		meeting $S_{f, r}$ for each $r$, $r_{0}< r < 1,$ transversally. 
		\begin{itemize} 
			\item  $g^{{t_{0}}}(D_{f, r_0, t}) = D_{f, r_0, t+t_{0}}$. 
			\item $D_{f, r_0, t'} \cap D_{f, r_0, t} = \emp$ for $t, t', t\ne t'$. 
			\item $\clo(D_{f, r_0, t}) \cap \Ss_{+}$ is given as a geodesic ending at the parabolic fixed point of $g$. 
		\end{itemize} 
	\end{itemize} 
\end{theorem}

\begin{remark}\label{rem:Sr} 
	The quotient $R_{f, r_0}/\langle g \rangle$ is foliated by the foliation 
	$\mathcal{S}_{f, r_0}$ induced by $\tilde{\mathcal{S}}_{f, r_0}$ 
	and $\mathcal{D}_{f, r_0}$ induced by $\tilde{\mathcal{D}}_{f, r_0}$. 
	The leaves of $\mathcal{S}_{f, r_0}$ are annuli of the form 
	$S_{f, r}/\langle g \rangle $ and the leaves of $\mathcal{D}_{f,r_0}$ 
	are the embedded images of $D_{f, r_0, t}$ for $t \in \bR$. 
	The embedded image of $R_{f, r_0}/\langle g \rangle$ in 
	$\Lspace/\Gamma$ are foliated by induced foliations 
	to be denoted by the same names. 	
\end{remark} 



\section{Orbits of proper affine deformations and translation vectors} \label{sec:orbit}

We now come to the most important section of this paper. 
\purple{
In this section, we assume $\mathcal{L}(\Gamma) \subset \SO(2,1)^{o}$
and work with Criterion \ref{cr:positive} only without assuming the 
properness of the $\Gamma$-action. }
In Sections \ref{sub:mar} and \ref{sub:netralized}, 
we will {present the objects of our discussion}. In Section \ref{sub:Anosov}, 
we will discuss the Anosov properties of geodesic flows extended to a flat bundle 
$\mathbf{V}$. 
In Section \ref{sub:tranv}, we will put the translation cocycle into an integral form. 
In Section \ref{sub:transvector}, we will compute the translation parts of the holonomy 
representations. Theorem \ref{thm:unifest} is the main result where we will give an outline of the proof. 
We will prove the converse 
part of Theorem \ref{thm:equivalence} at the end of Section \ref{sub:transvector}. 
{In Section \ref{sub:accumulate},  
we obtain Corollary \ref{cor:conv1} which discusses all the accumulation points of $\Gamma$.}

\subsection{Convergence sequences} \label{sub:mar} 


Let $g \in \Gamma$. 
%
Let $\lambda_{1}(g)$ denote the largest eigenvalue of $\mathcal{L}(g)$, 
which has eigenvalues $\lambda_{1}(g), 1, 1/\lambda_{1}(g)$. 
Note the relation 
\begin{equation} \label{eqn:length}
\oldch{l_{\Ss_+}(g)} = \log\left(\frac{\lambda_1(g)}{1/\lambda_1(g)}\right) = 2 \log \lambda_1(g).
\end{equation}

Recall that $\Gamma$ acts as a convergence group of a circle $\partial \Ss_{+}$. 
That is, if $g_{i}$ is a sequence of mutually distinct elements of $\Gamma$, 
then there exists a subsequence $g_{j_{i}}$ and points $a, r$ in $\partial \Ss_{+}$ so that 
\begin{itemize} 
\item as $i \ra \infty$, $\{g_{j_{i}}|\partial \Ss_{+}\setminus \{r\}\}$ uniformly converges to a constant map with value $a$ on every compact subset, and 
\item as $i \ra \infty$ $\{g_{j_{i}}^{{-1}}| \partial \Ss_{+}\setminus \{a\}\}$ uniformly converges to a constant map with value $r$ on every compact subset. 
\end{itemize}
Call $a$ the \hypertarget{term-att}{{\em attractor}} of $\{g_{j_{i}}\}$ and $r$ the \hypertarget{term-rep}{{\em repeller}} of  $\{g_{j_{i}}\}$. 
Here, $a$ may or may not equal $r$. (See \cite{ABT04} for detail.)
We call the sequence ${g_{i}}$ satisfying the above properties 
the \hypertarget{term-cfg}{{\em convergence sequence}}.

For a point $x \in \Lspace$, let $\Gamma(x)$ denote the orbit of $x$. 
We define the {\em Lorentzian limit set} $\Lambda_{\Gamma}:= 
\bigcup_{x\in \Lspace}(\clo(\Gamma(x))\setminus \Gamma(x))$. 
By the properness of the action, we obviously have: 

\begin{lemma} \label{lem:limitset} 
Let $\Gamma$ be a \hyperlink{term-pad}{{proper affine free group}}  with rank $\geq 2$.  
Then $\Lambda_{\Gamma}$ is a subset of $\Ss$.
\end{lemma}

Recall $\Ss_{0} = \Ss\setminus \Ss_{+}\setminus \Ss_{-}$. 
For each point  $x$ of $\partial \Ss_{+}$, there exists an \hyperlink{term-gs}{accordant} great segment $\zeta_{x}$
(see Definition \ref{defn:semicircle}). 
We denote by $\Pi_{+}: \Ss_{0}\ra \partial \Ss_{+}$ the map given by sending 
every point of $\clo(\zeta_{x})$ to $x$. 
 This is a fibration by Section 3.4 of \cite{CG17}.

Let $\Lambda_{\Gamma, \Ss_{+}} \subset \clo(\Ss_{+})$ be 
the limit set of the discrete faithful Fuchsian group action on $\Ss_{+}$ by $\mathcal{L}(\Gamma)$. 
(See \cite{Beardon}.)

One of our main results of the section is Corollary \ref{cor:conv1} 
also giving us: 
\begin{theorem} \label{thm:Lambda} 
Let $\Gamma$ be a proper affine free group of rank $\geq 2$ with or without parabolics. Assume $\mathcal{L}(\Gamma) \subset \SO(2,1)^o$. 
Then 
$\Lambda_{\Gamma}\subset \Pi_{+}^{-1}(\Lambda_{\Gamma, \Ss_{+}}). $
\end{theorem}


\subsection{The bundles ${\mathbf E}$ over $\Uu \Sf$}\label{sub:netralized}
%
%
Let $\Uu \Ss_+$ denote the unit tangent bundle of $\Ss_+$, i.e., 
the space of direction vectors on $\Ss_+$. For any subset $A$ of $\Ss_+$, 
we let $\Uu A$ denote the inverse image of $A$ in $\Uu \Ss_+$ under 
the projection. 
%
%
The projection $\Pi_{\Sf}: \Uu \Sf \ra \Sf$ lifts to the projection
$\Pi_{\Ss_{+}}:\Uu \Ss_{+} \ra \Ss_{+}$.

Let $\Gamma := h_{u}(\pi_{1}(\Sf))$ be a \hyperlink{term-pad}{proper affine deformation free group} of rank $\geq 2$. 
We note that $\Gamma$ acts on 
$\Uu\Ss_+$ as a deck transformation group over $\Uu \Sf$. 
An element $\gamma \in \Gamma$ goes to the differential map $D\gamma: \Uu\Ss_+ \ra \Uu\Ss_+$
defined by 
\[ D\gamma(x, \mathbf{u}) = \left(\gamma(x), \left. \frac{d\gamma(\beta(t))}{d t}\right|_{t=0}\right),  x\in \Ss_+, \mathbf{u} \in \Uu_x\Ss_+ \]
where $\beta(t)$ is a unit speed geodesic with $\beta(0) = x$ and $\dot \beta(0) = \mathbf{u}$. 
Goldman-Labourie-Margulis in \cite{GLM09} 
constructed a flat affine bundle 
${\mathbf E}$ over the unit tangent bundle $\Uu \Sf$ of $\Sf$.  
They took the quotient of $ \Uu\Ss_+ \times \Lspace$  by the diagonal action given 
by \[\gamma(\vv, x) = (D\gamma(x), \gamma(\vv)), x \in \Uu\Ss_+, \vv \in \Lspace \] for a deck transformation 
$\gamma \in \Gamma$. 
The cover $\Uu \SI_+\times \Lspace$ of $\mathbf E$ is denoted by $\hat{\mathbf E}$ and is identical with $\Lspace \times \Uu\Ss_+$. We denote by 
\[{\Pi_{\Lspace}}: \hat{\mathbf E} = \Uu\Ss_+ \times \Lspace  \ra \Lspace\]
the projection.

\subsection{The Anosov property of the geodesic flow} \label{sub:Anosov} 

We denote the standard $3$-vectors by 
\[ \vi :=(1,0,0), \vj=(0,1,0), \vk=(0,0,1).   \]

 \begin{definition} \label{defn:compt}
	We say \newch{that} two positive-valued functions $f(t)$ and $g(t)$, $t \in \bR$, 
	{\em are compatible} or {\em satisfy} $f \cong g$  if there exists $C > 1$ such that 
	\[\frac{1}{C} \leq \frac{f(t)}{g(t)} \leq C \hbox{ for } t \in \bR.\] 
\end{definition}

Given $(\llrrparen{ \vx }, \bu) \in \Uu\Ss_{+}$, 
\begin{itemize}
	\item we denote by 
	$l(\llrrparen{ \vx }, \bu)\subset \Ss_{+}$ the oriented complete geodesic 
	passing through $\llrrparen{ \vx }$ in the direction of $\bu$, and
	\item we denote by $\vv_{+, (\llrrparen{ \vk }, \vj)}$ and 
	$\vv_{-, (\llrrparen{ \vk }, \vj)}$  the respective null vectors 
	$\frac{1}{\sqrt{2}} \vj + \frac{1}{\sqrt{2}}\vk$ and $\frac{-1}{\sqrt{2}}\vj + \frac{1}{\sqrt{2}}\vk$
	in the directions of the forward and backward
	endpoints of  the oriented complete geodesic $l(\llrrparen{ \vk }, \vj)\subset \Ss_{+}$.
	\item We define $\vv_{+, (\llrrparen{ \vx }, \bu)}$ and 
	$\vv_{-, (\llrrparen{ \vx }, \bu)}$ respectively as the images 
	of $\vv_{+, (\llrrparen{ \vk }, \vj)}$ and $\vv_{-, (\llrrparen{ \vk }, \vj)}$ under \oldch{ $g$ 
	for $g \in \SO(2,1)^o$  provided
	\[ g(\llrrparen{\vk}) = \llrrparen{\vx} \hbox{ and } 
	g(\vj) = \bu.\]} 
The well-definedness of these objects follows since 
	there is a \newch{one-to-one} correspondence of $\Uu \Ss_+$ with $\SO(2,1)^o$.  
\end{itemize} 

\begin{definition}\label{defn:V}
	We define $\mathbf V$ as the quotient space 
	of $ \tilde {\mathbf V} := \Uu\Ss_+  \times \bR^{2, 1} $  under the diagonal action defined by 
	\begin{equation} \label{eqn:bfV}
	\gamma(x, \vv) = ( D\gamma(x), {\mathcal L}(\gamma)(\vv)),  x \in \Uu\Ss_+, \vv \in \bR^{2, 1}, \gamma \in \Gamma.
	\end{equation} 
	
	We will also need to define $\widetilde{\mathscr{V}}:= \Ss_+ \times \bR^{2,1}$
	and the quotient bundle $\mathscr{V} := \widetilde{\mathscr{V}}/\Gamma$ 
	where the action is given by 
	\begin{equation}\label{eqn:scrV}
	\gamma(x, \vv) = ( \gamma(x), {\mathcal L}(\gamma)(\vv)),  x \in \Ss_+, \vv \in \bR^{2, 1}, \gamma \in \Gamma.
	\end{equation}
	
\end{definition}

The vector bundle $\mathbf{V}$ has a fiberwise Riemannian metric $\llrrV{\cdot}_\rf$ where $\Gamma$ acts as {an isometry group}. 
At $(\llrrparen{ \vx }, \bu) \in \Uu\Ss_{+}$ with 
$\vx \hbox{ satisfying } \Bs(\vx, \vx) = -1$, 
we give as a basis 
\begin{equation}\label{eqn:frame}
\left\{ \vv_{+ ,(\llrrparen{ \vx }, \bu) }, \vv_{-, (\llrrparen{ \vx }, \bu) }, \vv_{0, (\llrrparen{ \vx }, \bu) }:=  \frac{\vv_{-, (\llrrparen{ \vx }, \bu) }\times \vv_{+, (\llrrparen{ \vx }, \bu) }}{\llrrV{\vv_{-, (\llrrparen{ \vx }, \bu) }\times \vv_{+, (\llrrparen{ \vx }, \bu)} } }   \right\} 
\end{equation} 
for the fiber over  $\llrrparen{\vx}$
where $\times $ is the Lorentzian cross product. 
We choose the positive definite metric $\llrrV{\cdot}_\rf$ on $\tilde{\mathbf{V}}$ so that the above vector frame \newch{is} orthonormal at the 
fiber of $\tilde{\mathbf{V}}$ over $(\llrrparen{ \vx }, \bu)$. 
The metric is $\SO(2, 1)^{o}$-invariant on $\Uu\Ss_{+}$.
Thus, this induces a metric $\llrrV{\cdot}_\rf$ on $\mathbf{V}$ as well. 


Let $\tilde{\mathbf{V}}_{\omega}$ be the $1$-dimensional subbundle of $\Uu \Ss_+\times \bR^{2, 1}$ containing $\vv_{\omega, (\llrrparen{ \vx }, \bu) }$
for each $\omega$, $\omega = +, -, 0$. 
It is redundant to say 
that  $\vv_{\omega, (\llrrparen{ \vx }, \bu) }$ is a fiber over the point $\llrrparen{\vx}$ in $\Ss_+$ for each $\omega$. 

We define a so-called {\em neutral map}
\[\tilde\bnu: \Uu\Ss_+ \ra \Uu \Ss_+\times \bR^{2, 1}\]
given by  $(\llrrparen{ \vx }, \bu) \mapsto \vv_{0, (\llrrparen{ \vx }, \bu) }$.
Here, $\tilde\bnu$ is an $\SO(2,1)^o$-equivariant map. 
By action of the isometry group $\Gamma$, we obtain a
\hypertarget{term-ns}{{\em neutral section}}
\[ \bnu: \Uu \Sf \ra {\mathbf V}\]
by using the $\SOto^{o}$-equivariance of the map. 
Hence, $\tilde{\mathbf{V}}_{0}$ coincides with 
the subspace generated by the image of the neutral section $\tilde \bnu$.


For any smooth map $g: \Uu \Ss_+ \ra \Uu \Ss_+$ or $\Ss_+ \ra \Ss_+$, we denote by 
$\bbD g$ the induced automorphism $\Uu \Ss_+ \times \Lspace$
acting trivially on the $\Lspace$-factor. 

Recall from  Section 4.4 of \cite{GLM09} 
the geodesic flow 
$\Psi_t: \Uu\Ss_{+}   \ra  \Uu\Ss_{+}$ denote the geodesic flow 
on $\Uu\Ss_+$ defined by the hyperbolic metric. Let
\[\bbD\Psi_{t}: \Uu\Ss_{+} \times \bR^{2, 1}  \ra  \Uu\Ss_{+} \times \bR^{2, 1}\]
denote the Goldman-Labourie-Margulis flow. This acts trivially on the second factor and as the geodesic flow on $\Uu\Ss_{+}$. 
The bundle $\mathbf{V}$ splits into three $\Psi_{t}$-invariant line bundles
$\mathbf{V}_{+}$, $\mathbf{V}_{-}$ and $\mathbf{V}_{0}$, 
which are images of $\tilde{\mathbf{V}}_{+}$, $\tilde{\mathbf{V}}_{-}$ and $\tilde{\mathbf{V}}_{0}$.
Our choice of $\llrrV{\cdot}$ shows that $\bbD\Psi_{t}$ acts as uniform contractions in $\mathbf{V}_{+}$ as $t \ra \infty, -\infty$, 
i.e., 
\begin{align} \label{eqn:PhitC}
& \llrrV{\bbD\Psi_{t}(\vv_{+})}_\rf \cong \exp(- t)\llrrV{\vv_{+}}_\rf \hbox{ for } \vv_{+}\in \tilde{\mathbf{V}}_{+}, \nonumber \\ 
& \llrrV{\bbD\Psi_{t}(\vv_{-})}_\rf \cong \exp(t)\llrrV{\vv_{-}}_\rf \hbox{ for }  \vv_{-}\in \tilde{\mathbf{V}}_{-}, \hbox{ and } \nonumber \\
& \llrrV{\bbD\Psi_{t}(\vv_{0})}_\rf \cong \llrrV{\vv_{0}}_\rf \hbox{ for } \vv_{0}\in \tilde{\mathbf{V}}_{0}.
\end{align} 
Here, $k$ in \cite{GLM09} equals $1$ since we can explicitly compute $k$ from the framing above.
The signs are different from \cite{GLM09} because we have slightly different objects. 
The fiberwise metric on $\Uu \Ss_+$ 
is not dependent on the group $\Gamma$ itself.
See the last paragraph of Section 4.4 of \cite{GLM09}.

\begin{figure}[h]
	
			\begin{lpic}{Surface2(10.7cm,)}
		\lbl[tl]{18,219; $\vv_{-}$}
	\lbl[tl]{62, 180; $\vv_0$}

		\lbl[tl]{67,210; $\vv_{-}$}
		\lbl[tl]{100, 177; $\vv_0$}
		\lbl[tl]{102, 198; $\vv_+$}

	\lbl[tl]{120, 200; $\vv_-$}
		\lbl[tl]{142, 178; $\vv_0$}
		\lbl[tl]{150, 198; $\vv_+$}

		\lbl[tl]{185, 200; $\vv_-$}
		\lbl[tl]{201, 201; $\vv_+$}
	\lbl[tl]{208, 180; $\vv_0$}	

	\lbl[tl]{233, 199; $\vv_-$}
		\lbl[tl]{243, 180; $\vv_0$}
	\lbl[tl]{250, 208; $\vv_+$}
	
		\lbl[tl]{279, 197; $\vv_-$}
	\lbl[tl]{285, 183; $\vv_0$}		
\lbl[tl]{299, 220; $\vv_+$}

	\lbl[tl]{315, 183; $\vv_0$}
		\lbl[tl]{311, 197; $\vv_-$}
		\lbl[tl]{336, 220; $\vv_+$}

		\lbl[tl]{185, 175; $\mathscr{K}$}
		\lbl[tl]{158, 130; $p_{\Sf}(\mathscr{K})$}
	\end{lpic}
	
	\caption{The frames on $\Uu\Ss_+$ and on $\Uu \Sf$. The circles bound horodisks covering the cusp neighborhoods below. The compact set 
	$\mathscr{K}$ is some small compact set where the closed geodesics pass through. We drew only one closed geodesic.} 
	\label{fig:surface}
\end{figure}

\begin{remark}\label{rem:dualpict}
	The induced geodesic flow on $\Sf$ is denoted by $\Psi_t$ and the
	induced action on $\mathbf{V}$ by $\bbD\Psi_t$. 
	We may think of translating the picture of the flat bundle over $\Uu \Ss_{+}$ to the bundle over $\Uu \Sf$.
	As a bundle over $\Uu\Sf$, $\bbD\Psi_{t}$ contracts and expands uniformly for $\mathbf{V}_{\pm}$ with respect to $\llrrV{\cdot}_\rf$. 
	However, in the picture over $\Uu \Ss_+$, $\bbD\Psi_{t}$ is the identity between 
	fibers and objects lifted from $\mathbf{V}$ will uniformly increase or decrease exponentially with respect to any fixed Euclidean metric
	$\llrrV{\cdot}_E$ on $\tilde {\mathbf{V}}$. 
	(See Figure \ref{fig:surface}.)
	
\end{remark} 


Denote by
\[\tbV_{+}(\llrrparen{\vec x}, \bu), \tbV_{-}(\llrrparen{\vec x}, \bu), \tbV_{0}(\llrrparen{\vec x}, \bu)\] 
the fibers of $\tbV_{+}, \tbV_{-}, \tbV_{0}$ over 
$(\llrrparen{\vec x}, \bu) \in \Uu \Sf$ respectively.  
We denote by 
\begin{equation} \label{eqn:projV} 
\Pi_{\tbV_{+}}: \tbV \ra \tbV_{+}, 
\Pi_{\tbV_{-}}: \tbV \ra \tbV_{-}, \hbox{ and } \Pi_{\tbV_{0}}: \tbV \ra \tbV_{0}
\end{equation} 
the projections 
using the direct sum decomposition 
\[\tbV = \tbV_{+} \oplus \tbV_{-}\oplus \tbV_{0}.\]

\subsection{Computing translation vectors} \label{sub:tranv}



Here, we will write the cocycle in terms of an integral. 
Let $g$ be a hyperbolic element. 
Let $a_{g}$ denote the \hypertarget{term-afp}{attracting fixed point} of $g$ in $\partial \Ss_{+}$ and $r_{g}$ the \hypertarget{term-rfp}{repelling one}. 
 Let $\Sigma_{+}$ denote the surface
 \[((\Ss_{+} \cup \partial \Ss_{+})\setminus \Lambda_{\Gamma, \Ss_{+}})/\Gamma.\]
The surface $\Sf$ is the dense subset of $\Sigma_{+}$.
The $\mathscr{V}$-valued forms are differential forms 
with values in the fiber spaces of $\mathscr{V}$. (See Definition \ref{defn:V}.)
The $\widetilde{\mathscr{V}}$-valued forms on $\Ss_+$ are simply 
the $\bR^{2,1}$-valued forms on $\Ss_+$. 
However, the group $\Gamma$ acts by 
\begin{equation} \label{eqn:Action} 
\gamma^\ast(\vv\otimes dx ) =  \mathcal{L}(\gamma)^{-1}(\vv)\otimes d (x \circ \gamma) =
 \mathcal{L}(\gamma)^{-1}(\vv)\otimes \gamma^* dx, \gamma \in \Gamma. 
\end{equation}
(See Chapter 4 of Labourie \cite{LabourieB}.)

Let \hypertarget{term-em}{$\llrrV{\cdot}_{E}$} denote a Euclidean metric on $\Lspace$
by changing signs of the Lorentz metric which we fix from now on.
Let $g$ be a hyperbolic isometry. 
Let $x_g$ be a point of the geodesic $l_g$ in $\Ss_+$ {on which $g$ acts preserving an orientation direction $\vu_g$}. 
We define 
\[  \bnu_g := \vv_{0, (x_g, \vu_g)}= \tilde \bnu(x_{g}, \vu_{g}), \]
which is independent of the choice of $(x_g, \vu_g)$ on $l_g$ by 
\eqref{eqn:frame}.

\newch{Recall from Section \ref{sub:propaffine} the cocycle of $\Gamma = h_{\bfb}(\pi_1(\Sf))$ for 
the holonomy homomorphism $h_{\bfb}$:
$\bfb \in Z^{1}(\pi_{1}(\Sf), \bR^{2,1}_{h})$.}
\newch{We} write every element $g$ as $g(x) = A_{g}x + \vb_{g}$, $x \in \Lspace$.
Then the function $\bfb: \Gamma \ra \bR^{2, 1}$  given by 
\[ g \mapsto \vb_{g} \hbox{ for every } g\] 
is a cocycle representing an element of
\[H^{1}(\pi_{1}(\Sf), \bR^{2, 1}) = H^{1}(\Sf,\mathscr{V})\]
using the de Rham isomorphism. 
(See Theorem 4.2.3 of Labourie \cite{LabourieB}.)
Let $\eta$ denote the smooth $\mathscr{V}$-valued $1$-form on $\Sf$ representing 
the cocycle $\mathbf{b}$ in the de Rham sense. 

\newch{
Let $\tilde \eta: \Ss_+ \ra \bR^{2, 1}$ denote the lift of $\eta$ to $\Ss_+$. 
We can think of $\tilde \eta$, which is 
\newch{$h$}-equivariant,  
as the differential of a section $s_{\tilde \eta}: \Ss_+ \ra \Lspace$
which is $h_{\bfb}$-equivariant: 
\begin{equation} \label{eqn:seta} 
\tilde \eta = d s_{\tilde \eta} 
\end{equation} 
by  
Theorem 1.14 of \cite{GH84}
and lifting to the cover $\Ss_+\times \Lspace$. 
}



{
Recall from Section \ref{sub:thin}, the end neighborhood $E$ and its inverse image $\mathscr{H}\subset \Ss_+$. }
Let $\CH(\Lambda)$ denote the convex hull of a closed subset $\Lambda$ 
of $\partial \Ss_+$ in $\Ss_+$. The surface 
$\Sf_{C} := \CH(\Lambda_{\Gamma, \Ss_+})/\Gamma$ is a finite-volume connected  hyperbolic surface with geodesic boundary and cusp ends. The boundary of 
$\Sf_{C}$ is a union of finitely many closed geodesic boundary components, 
and each end of $\Sf_C$ is a cusp. Assume that each 
component of $E$ is a subset of $\Sf_C$ by choosing suitable 
cusp neighborhoods. 
We let $F$ to denote a compact fundamental domain of $\CH(\Lambda_{\Gamma, \Ss_+})\setminus{{\mathscr{H}}}$. 


Let $\Uu \Sf_{C}$ denote the space of unit vectors on $\Sf$ with base points at $\Sf_{C}$, 
and $\Uu \CH(\Lambda_{\Gamma, \Ss_+})$ denote one for $\CH(\Lambda_{\Gamma, \Ss_+})$. 
We can compute the cocycle $\mathbf{b}$ by the following way: 

Let $\mathscr{K}$ be a small fixed compact domain in 
$\CH(\Lambda_{\Gamma, \Ss_+})\setminus {\mathscr{H}}$ in $\Ss_+$. 
Let $\tilde \eta$ denote the lift of $\eta$ on $\Ss_+$. 
We may also assume that 
\begin{equation} \label{eqn:etaK}
\tilde \eta|\mathscr{K} \equiv 0 
\end{equation} 
by locally changing 
$\eta$ \newch{ by \eqref{eqn:seta}}. 
 We simply need to change the section to 
 a section that is a fixed parallel section on $p_{\Sf}(\mathscr{K})$.
 This can obviously be achieved by 
using a partition of unity while this does not change the cohomology class of $\eta$. 
(See Section 4 of \cite{GLM09}.)

{To simplify, we assume that $s_{\tilde \eta}$ 
at $\mathscr{K}$ takes the value of  the origin $O$.}

\begin{definition}\label{defn:GammaK}
	Let \hypertarget{term-GK}{$\Gamma_{\hat {K}}$} denote the set of hyperbolic elements $g\in\Gamma$ that
	acts on a geodesic $l_{g}$ in $\Ss_{+}$ passing a compact 
	subset $\mathscr{K} \subset \Ss_+ \setminus {\mathscr{H}}$. 
\end{definition}

We lift the discussion to $\Uu \Sf_C$ and its cover $\Uu\CH(\Lambda_{\Gamma, \Ss_+}) \subset \Uu \Ss_{+}$. 
Let $g$ be an element of $\Gamma_{\mathscr{K}}$ 
corresponding to a closed geodesic $c_{g}$. 
Let $l_g$ be the unit speed geodesic in $\Ss_{+}$ in  connecting $x_{g}\in \mathscr{K}$ to $g(x_{g})$
covering $c_{g}$ \newch{with the length $t_g$}. 
Let $\Pi_{\bR^{2,1}}: \Uu\Ss_+\times \bR^{2,1} \ra \bR^{2,1}$ denote the projection 
to the second factor. 
Then by the trivialization on $\mathscr{K}$
\begin{equation}\label{eqn:bg-pre} 
\vb_{g} =  \Pi_{\bR^{2,1}} \left(\int_{[0, t_0]} \tilde \eta\left(\frac{dl_g(t)}{dt}\right) dt
\right) 
\end{equation}
where $t_g$ is the time needed to go from $x_{g}$ to $g(x_{g})$.
(See Section 4.2.2 of Labourie \cite{LabourieB}.)
{However, we will consider the case when $x_g$ is anywhere in $\Ss_+$,} 
Since 
	\begin{multline} 
\Pi_{\bR^{2,1}} \left(\int_{[0, t_g]} \tilde \eta\left(\frac{dl_g(t)}{dt}\right) dt \right) =	g(\Pi_\Lspace \circ s_{\tilde \eta}(x_g))) -   \Pi_{\Lspace}\circ s_{\tilde \eta}(x_g)  \\
 	=  (\mathcal{L}(g)-\Idd)(\Pi_\Lspace\circ s_{\tilde \eta}(x_{g}))  + \oldch{\vb_g}, 
	\end{multline}   
we have 
\begin{equation}\label{eqn:bg} 
\vb_{g} = \Pi_{\bR^{2,1}} \left(\int_{[0, t_g]} \tilde \eta\left(\frac{dl_g(t)}{dt}\right) dt \right)
+ (\Idd - \mathcal{L}(g))({\Pi_{\Lspace} \circ} s_{\tilde \eta}(x_{g})).
\end{equation}
Thus, we obtain
\begin{multline} \label{eqn:Phi}
\vb_{g} = \Pi_{\bR^{2,1}} \Bigg(\int_{[0, t_{g}]} \bbD\Psi((x_{g}, \bu_g), t)^{-1}\left(\tilde \eta\left(\frac{d\Psi((x_{g}, \bu_g), t)}{dt}\right)\right) dt  \Bigg)\\
+ (\Idd - \mathcal{L}(g))({\Pi_{\Lspace} \circ} s_{\tilde \eta}(x_{g}))
\end{multline}
where the geodesic segment $\Psi((x_{g}, \bu_g), [0, t_{g}])$ for a unit vector $\bu_g$  at $x_g$, covers a closed curve representing $g$. 



{Using the origin $O$ of $\Lspace$, we can consider it as $\bV$ with a vector subspace $\bV_\omega$, $\omega = +, -, 0$.
Define $\Pi_{\omega, x_0}:= \Pi_{\bR^{2,1}} 
\circ \Pi_{\tbV_{\omega}, x_0}:\{x_0\} \times \Lspace \ra \bV_{\omega, x_0} \ra \bR^{2, 1}$ to denote the projection $\Pi_{\tbV_{\omega}}$ at the fiber $\Lspace$ over $x_0 \in \Uu \Ss_+$. 
Define 
\begin{equation}
\tilde \eta_{\omega}(x_0) = \Pi_{\tbV_{\omega, x_0}}(\tilde \eta(x_0)), 
\end{equation}
where $\omega = +, -, 0$.} 
Since $\Psi_{t}$ preserves the decomposition, 
$\bbD\Psi(x, t)$ commutes with these projections.

\begin{definition}\label{defn:bg-}
	Let $\mathscr{K}$ be the compact subset of $\Ss_+ \setminus {\mathscr{H}}$. 
Let $g \in \Gamma_{\mathscr{K}}$. 
We choose $x_{g}\in \mathscr{K}$ so that 
the arc $\Psi((x_{g}, \bu_g), [0, t_{g}])$ 
for a unit vector $\bu_g$ at $x_{g}$ covers a closed geodesic representing $g$
where $(g(x_{g}), Dg(\bu_g)) = \Psi((x_{g}, \bu_g), t_{g})$.
{The arc here is not necessarily in $\mathscr{K}$ of course.}  
We define invariants: 
\begin{multline}
\vb_{g, \omega} := \Pi_{\tbV_{\omega}, x_g}(\vb_{g}) = \\ 
\Pi_{\bR^{2,1}}\Bigg(\int_{[0, t_{g}]} \bbD\Psi((x_{g}, \bu_g), t)^{-1}\left(\tilde \eta_{\omega}\left(\frac{d\Psi((x_{g}, \bu_g), t)}{dt}\right)\right) dt\Bigg) \\
+ (\Idd - \mathcal{L}(g)) \left({\Pi_{\omega, x_g}}(s_{\tilde \eta}(x_{g}))\right), 
\label{eqn:b+}  
\end{multline}
where $\omega = +, -, 0$ respectively. 
The second equalities hold since $\bbD\Psi(x, t)$ {and $\mathcal{L}(g)$ 
commute} with projections $\Pi_{\tbV_{+}}, \Pi_{\tbV_{-}}$ and $\Pi_{\tbV_{0}}$. 
\end{definition} 

\begin{proposition}\label{prop:b0alpha}
	{For nonparabolic $g \in \Gamma - \{\Idd\}$,} we have 
	\begin{gather} \label{eqn:b0alpha} 
	\vb_{g, 0} = \alpha(g) \bnu_g,\\
	\llrrV{\vb_{g, 0}} = \alpha(g).
	\end{gather}
\end{proposition}
\begin{proof} 
	First, $\vb_{g, 0}$ is parallel to $\bnu_{g}$ by 
	\eqref{eqn:b+}. 
	Since $\bnu_g$ is Lorentz orthogonal to 
	the subspace spanned by $\vv_{+, (x_g, \vu_g)}$ and 
	$\vv_{-, (x_g, \vu_g)}$,
	the component $\vb_{g, 0}$ is the image 
	$\vb_g$ under the Lorentzian projection to $\bnu_g$.
	Since $\vb_{g} = g(O) - O$ for the origin $O$ by our choice of the $\mathbf{E}$-section
	near \eqref{eqn:etaK}, and $\llrrV{\bnu_g} =1$, 
	\eqref{eqn:alpha} and {Criterion \ref{cr:positive} imply} the result. 
	\end{proof}

\purple{
The norm of  a $1$-form with values in $\bV_0$ is given by the fiberwise 
norm of $\bV_0$ and the norm of hyperbolic metric for the tangent bundle of 
$\Sf$. Finally, we will need: 
\begin{definition} \label{defn:alphafactor} 
Let $K$ be a compact subset of $\Sf$, and let 
$\tilde K$ denote the inverse image of $K$ in $\Ss_+$. 
The {\em neutral factor of $\eta|K$} is given as 
the maximum norm of $\tilde \eta_{0}$ on $\Uu \tilde K$.
\end{definition}
} 



%

\subsection{Translation vectors have direction limits in $\Ss_0$} \label{sub:transvector} 

\purple{We aim to prove Theorem \ref{thm:unifest} from Section
\ref{subsub:cusp} to Section \ref{subsub:direction}. 
Section \ref{subsub:cusp} discusses the standard cusp $1$-forms
and how to integrate along geodesics to obtain the Margulis invariants. 
Important Lemma \ref{lem:largecusp} shows that long cusp geodesics can absorb many
possibly negative perturbations during the argument that we will present. 
Section \ref{subsub:summing} outlines the proof of Theorem \ref{thm:unifest}. 
In Section \ref{subsub:ratiobound}, \newch{we show
$\alpha(g_i) \ra \infty$ and $\alpha(g_i)/\llrrV{\vb_{g_i}} \ra \infty$ if 
$l_{\Ss_+}(g_i) \ra \infty$. 
We will use the fact that a sequence converges to $+\infty$ if 
we can show that a subsequence of any subsequence converges to $+\infty$. 
Hence, we will start with a subsequence and keep taking subsequences 
to obtain one that converges to $+\infty$.
}
In Section \ref{subsub:direction}, we finish the proof of the theorem on 
the limit of direction vectors. 
}


\subsubsection{Cusp forms} \label{subsub:cusp} 

A \hypertarget{term-shd}{{\em standard horodisk}} $D$ is an open disk bounded by a horocycle in $\Ss_{+}$ passing $\llrrparen{\vk}$ and ending 
at {the} unique point $\llrrparen{\vj + \vk}$. 
We denote by $\partial_h D$ the horocycle 
$\clo(D)\setminus (D\cup \{\llrrparen{\vj + \vk} \})$ for any horodisk $D$.

Let $D'$ be a horodisk in $\Ss_+$.
Let $\vp$ denote a null-vector in the direction of $p \in \clo(D') \cap \partial \Ss_{+}$.
Let us use an upper half-space model of the hyperbolic plane with the standard coordinates $x, y$
and $p$ corresponding to $\infty$. 
Then we may assume without loss of 
generality that $D'$ is given by $y > 1$.
%

	\begin{definition} \label{defn:propc}
\purple{ 
Let $g$ be an accordant parabolic transformation in $\Gamma$.
Using the parabolic coordinates, 
let $g$ be of \newch{the} form \eqref{eqn:Phit} for some $t >0$. 
Let $E'$ be a cusp neighborhood covered by $D'$ where 
$\langle g \rangle$ acts as the deck transformation group. 
On $D'$, we can find a $\mathscr{V}$-valued $1$-form
\newch{
\begin{equation} \label{eqn:cuspform} 
\mu(x^2/2,-  x, 1) d x
\end{equation} 
that is closed but not exact and is $g$-invariant by \eqref{eqn:Action}
with respect to a coordinate system adopted to $g$.
}
We call such a form on $D'$ and the induced one 
on $E'$ {\em standard cusp $1$-forms}, 
$\mu> 0$ is the {\em cusp coefficient} of $E'$.
}
(See \cite{StandardForm} to check the form and the invariance.)

	\end{definition}
Here $\mu > 0$ by Lemma \ref{lem:pex} since $t > 0$ under the 
assumption.

Let ${{\mathscr{H}}}_{j} \subset \Ss_{+}, j=1, 2, \dots $, denote the horodisks covering 
the components of $E$.  
Let $p_{j}$ denote the parabolic fixed point corresponding to ${{\mathscr{H}}}_{j}$.
Each ${{\mathscr{H}}}_{j}$ has {standard} coordinates $x_{j}, y_{j}$ from the upper half-space model of {$\Ss_+$}
where $p_{j}$ becomes $\infty$, and ${{\mathscr{H}}}_{j}$ is given by $y_{j} >1$. 


Since $\Sf$ has finitely many cusps, we can choose 
horocyclic end neighborhoods with mutually disjoint closures.
By taking even smaller ones, 
we may also assume that 
\begin{equation}\label{eqn:distH}
d_{\Ss_+}(g({{\mathscr{H}}}_i),  k({{\mathscr{H}}}_j)  ) > C^{\text{\eqref{eqn:distH}}}_E, C^{\text{\eqref{eqn:distH}}}_E \geq \newch{5/4}, 
g, k \in \Gamma, i,j=1, \dots, m_0
\end{equation} 
whenever $g({{\mathscr{H}}}_i) \ne  k({{\mathscr{H}}}_j)$
for some fixed constant $C^{\text{\eqref{eqn:distH}}}_E$ depending only on $E$.


There are only finitely many cusps in $\Sf_C$. Thus, we can choose finitely many 
cusps in each orbit class of cusps whose closures meet the fundamental domain $F$.
We may denote these by ${{\mathscr{H}}}_{1}, \dots, {{\mathscr{H}}}_{m_{0}}$ by
reordering if necessary.
We denote by $\vp_{1}, \dots, \vp_{m_{0}}$ the corresponding null vectors. 
\newch{We choose a parabolic coordinate system for each ${{\mathscr{H}}}_{j}$ 
in the $\Gamma$-equivariant manner.}

\newch{
Recall from Section \ref{sub:propaffine} the cocycle of $\Gamma = h_{\bfb}(\pi_1(\Sf))$ for 
the holonomy homomorphism $h_{\bfb}$:
$\bfb \in Z^{1}(\pi_{1}(\Sf), \bR^{2,1}_{h})$.
For each $\gamma \in \pi_1(\Sf)$, 
$\bfb(\gamma) = h_{\bfb}(\gamma)(x_0) - x_0$ for a basepoint $x_0$. 
For each peripheral element $\gamma$ in the boundary orientation, 
let $\hat \gamma$ denote the corresponding deck transformation. 
We choose an adopted parabolic coordinate system where
$h(\hat \gamma)$ is accordant.  Let $E_\gamma$ be a component of $E$ corresponding to $\gamma$. 
Let $\gamma'$ be the homotopy class in $E_\gamma$ 
of the simple closed curve $c_{\gamma'}$ bounding $E_\gamma$ with a basepoint $x_{0, \gamma}$. 
If we choose a basepoint to be the origin of the coordinate system, 
we obtain a class  $\bfu$ in 
$H^1(\langle \hat \gamma  \rangle,  \bR^{2,1}_{\langle h(\hat \gamma) \rangle})$.
Let $\tilde c_{\gamma'}$ denote the boundary horocycle 
corresponding to $\hat \gamma$. 
Using the partition of unity, 
we change the section $s_{\tilde \eta}$ associated with $\tilde \eta$ 
so that so that $s_{\tilde \eta}| \tilde c_{\gamma'}$ is the orbit of the origin of  
the one-parameter group of parabolic affine transformations containing  $h(\hat \gamma)$. 
By \eqref{eqn:seta}, new $\eta$ is obtained in $E_\gamma$. 
Since the de Rham class $[\eta^c_\mu] \in H^1(E', \mathcal{V})$ goes to 
$\bfu \in H^1(\langle \hat \gamma \rangle,  \bR^{2,1}_{\langle h(\hat \gamma) \rangle})$, 
we obtain by \newch{Propositions \ref{prop:replace} and \ref{prop:cohomology}}:
} 

\purple{  
\begin{corollary} \label{cor:replace}
Let $\Sf$, $\Gamma$, $P$, $E$, and $\gamma$ be as above. 
Then we may replace a closed $\mathscr{V}$-valued $1$-form $\eta$ on $\Sf$ with 
a cohomologous one $\eta'$ so that $\eta'|E'$ for each component $E'$ of $E$ is
a standard cusp $1$-form \newch{in a parabolic coordinate system adopted to
the accordant holonomy element following the boundary orientation}. 
\end{corollary} 
}

\purple{
We may choose the $1$-form $\eta$ representing the cohomology class so that 
$\tilde \eta$, its lift to $\Ss_+$, is a standard cusp $1$-form on ${{\mathscr{H}}}_j$.
} 
\purple{
Let $\mu_j$ denote the cusp coefficients for each $j$, $j=1, 2, \dots$, 
Since there are only finitely many cusps in $\Ss_+/\Gamma$, 
there are only finitely many values of the cusp coefficients. 
Let $\mu_{\text{min}}$ be the minimum of $\mu_1, \mu_2, \dots$, 
and let $\mu$ be the maximum of $\mu_1, \mu_2, \dots$.
}



Let \hypertarget{term-em}{$\llrrV{\cdot}_{E}$} denote a Euclidean metric on $\Lspace$ which we fix in this paper. 

\begin{lemma} \label{lem:metricbound}
	Let $\mathscr{K}$ be a compact subset of $\Ss_+ \setminus {\mathscr{H}}$.
Suppose $x\in \mathscr{K} $. 
Then 
the matrix $\mathscr{C}_{i}$ with columns 
\begin{equation}\label{eqn:vframe} 
\vv_{+, (x, \vu)},  \vv_{0, (x, \vu)}, \vv_{-, (x, \vu)} \hbox{ for every } 
\vu \in \Uu_x \Ss_+
\end{equation}
is in a compact subset of $\GL(3, \bR)$ depending only on $\mathscr{K}$. 
\end{lemma} 
\begin{proof} 
There is a uniformly bounded element of $\SO(2,1)^o$ sending 
a complete geodesic 
$\overline{\llrrparen{0,-1,1} \llrrparen{0,1,1}}$ to $l_{g_i}$
and $\llrrparen{1,0,0}$ to $\llrrparen{\bnu_{g_i}}$.
From this and the way we define the frames 
in  Section \ref{sub:Anosov}, the conclusion follows. 
\end{proof}

\newch{
	Let $g$ be a hyperbolic element. 
	We recall from \eqref{eqn:b+} and \eqref{eqn:b0alpha}, 
	\begin{multline} 
		\alpha(g) = \llrrV{\vb_{g, 0}}, 
\vb_{g, 0} = \Pi_{\tbV_{0}, x_{g}}(\vb_{g}) = \\ 
\Pi_{\bR^{2,1}}\left(\left(\int_{[0, t_{g}]} \Bs\left(\nu_{x_{g}, \bu_{g}},\tilde \eta\left(
		\frac{d\Psi((x_{g}, \bu_{g}), t)}{dt}\right)\right) dt \right) \nu_{x_{g}, \bu_{g}}\right) 
		\label{eqn:b0}  
	\end{multline} 	
	since  $(\Idd - \mathcal{L}(g)) \left({\Pi_{\tbV_0, x_{g}}}(s_{\tilde \eta}(x_{g}))\right) =0$.  
}

\newch{
	For any subinterval $\zeta$ in a cusp 
with the cusp coefficient $\mu$, 
we define 
	$\alpha(\zeta)$ to be the corresponding part of the above integral from 
	$t_{\zeta_0}$ and $t_{\zeta_1}$	 for the corresponding 
arc-length parametrizing interval  $[t_{\zeta_0}, t_{\zeta_1}]$. 
Define $R(\zeta)$ as the radius of $\zeta$ in the upper half-space model where the horocycle  is given by $y=1$. 
By Proposition \ref{prop:mintegrals}, and the compatibility \eqref{eqn:PhitC},
 we can use
	\begin{equation} \label{eqn:b0zeta} 
		\alpha(\zeta) = 
		\mu \left(\frac{\pm \sqrt{2}\sqrt{ R(\zeta)^2-1}}{R(\zeta)}  + 2  R(\zeta)\sqrt{ R(\zeta)^2-1})\right)   
	\end{equation}
}


\purple{
\begin{definition} \label{defn:r-value} 
	We define $r(\zeta) := \sqrt{R(\zeta)^2 -1}$, which equals $1/2$ times the absolute value of the difference of
	the $x$-coordinates of the endpoint of $\zeta$ in the upper half-space model where the horocycle is given by $y=1$. 
The {\em horospherical length} $h$ of a cusp neighborhood $E$ is 
the $d_{\Ss_+}$-length of $\partial E$. 
Note that if two maximal geodesics $\zeta$ and $\zeta'$ in a cusp $E$ 
have the same endpoints, then $r(\zeta)$ and $r(\zeta')$ differ by 
a half an integer times $h$. 
\end{definition}
} 

\purple{
One useful result is Theorem 4.6 of Heinze and Hof \cite{HH}, 
\begin{equation} \label{eqn:HH} 
r(\hat \zeta) = \sinh \frac{1}{2}(l_{\Ss_+}(\hat \zeta)). 
\end{equation} 
From this, we can show that the difference of $x$-coordinates of the end points of
an arc of length $l$ is $\leq 2\sinh(\frac{l}{2})$. 
} 

\newch{Heuristically, Lemma \ref{lem:largecusp} states that 
the homotopy classes of maximal geodesics in a cusp neighborhood will give
quadratic differnces in $\alpha$-values. In particular the item (ii) gives 
us the main estimations to absorb the negative contributions.} 


\purple{
\begin{lemma}[Large cusp radius]\label{lem:largecusp} 
Let $\zeta$ be a maximal geodesic in a cusp neighborhood 
$E'$ with  the standard cusp $1$-form and a cusp coefficient $\mu'$.
Let $h$ be the horospherical length of $E'$.  
There exists  a positive constant $R_{\text{c}}$, independent of $\mu'$ but dependent on 
$h$ and $C$, which is defined below 
so that for any $R_1 > R_{\text{c}}$ has the following properties\/{\em :} 
\begin{itemize} 
\item[{\em (i)}] For the set of maximal geodesics in $E'$, 
$r(\zeta') \mapsto \alpha(\zeta')$ for each $\zeta'$ in it
forms a strictly increasing positive function of $r(\zeta')$
for $r(\zeta') >  R_1$. 
\item[{\em (ii)}] 
\newch{Let $\zeta$ and $\zeta'$ be two maximal geodesics 
in $E$  with the same endpoint as $\zeta$ but 
in the different  homotopy classes with respect to endpoints.
For any constant $0 \leq \eta_0 < C $ with 
\[R - h/2< r(\zeta) <  R  < r(\zeta') \hbox{ for }R > R_1, \]
we have
\[
\alpha(\zeta') - \alpha(\zeta) -\mu' \eta_0 \geq  2C^{(4.6)}_{R_1,C} \mu' r(\zeta')^2
\] 
for a constant $C^{(4.6)}_{R_1, C}> 0$ depending only on $h, R_1$ and $C$.}  
\end{itemize} 
\end{lemma} 
}
\begin{proof} 
\purple{ 
We choose a horoball $\tilde E'$ covering $E'$.
Then we can compute $\alpha(\zeta)$ for a geodesic $\zeta$ 
by lifting $\zeta$ to $\tilde E'$.  
}
\purple{
(i) is straightforward. 
} 



\newch{
For (ii), the last term of \eqref{eqn:b0zeta} dominates the absolute values of
 other terms and 
$\mu \eta$ for sufficiently large $R_1$:
}
\newch{
Using \eqref{eqn:b0zeta}, the above term  divided by $\mu'$ is bounded below by 
\[ 
r(\zeta')^2 - r(\zeta)^2 - \eta_0 - 2\sqrt{2}.
\]
Since 
$(x-\newch{h/2})/x$ is an increasing function of $x$, 
the supremum on $x \in (R, R\newch{+}h/2)$  is $R/(R+h/2)$.
Hence, $r(\zeta) < C_R r(\zeta')$ for $C_R = R/(R+h/2)$
since the ratio $r(\zeta)/r(\zeta')$ 
is less than $C_R$ for $r(\zeta') \geq R+h/2$. 
}
\newch{
Then $\alpha(\zeta') - \alpha(\zeta) -\mu' \eta_0$ divided by $\mu'$   is bounded below by 
\begin{equation}
r(\zeta')^2( 1 - C_R^2 ) - C - 2\sqrt{2} \geq 
(1-C_R^2) \left(r(\zeta')^2  - \frac{C+2\sqrt{2}}{1-C_R^2}\right).
\end{equation}  
Let $f_R(x)$ denote the polynomial given by the right side
with $x$ replacing $r(\zeta')$.
The largest root of $f_R(x)$ is smaller than
}
\newch{
\[ \sqrt{\frac{(R+ h)(C+2\sqrt{2})}{h}}.\]
}
\newch{
Since the function $R \mapsto R$ dominates any function given by the square root of the 1st order polynomial of $R$,
there exists $R'> h$ so that for 
$R > R'$, we have 
\[
\newch{R >  \sqrt{\frac{(R+h)(C+2\sqrt{2})}{h}} }
\hbox{  which implies } f_{R}(x) > 0 \hbox{ for } x > R.
\]
Define $c:= \frac{f_{R'+1}(R'+1)}{(R'+1)^2} > 0$. 
Then 
\[f_{R'+1}(x) \geq c x^2 \hbox{ for } x \geq R'+1\] 
by
an easy calculus argument. 
We take $R_1 = R'+1$,  and 
$C^{(4.6)}_{R_1, C} = c/2$. 
We can make $R_1$ as large as we wish to since we only need $c> 0$.
}

\end{proof}

\subsubsection{Summing up the contributions} \label{subsub:summing} 


\newch{Let $\{g_i\}$ be a sequence of elements in $\Gamma_{\mathscr{K}}$.} 
We denote by $\hat l_{g_{i}}$ the lift of $l_{g_{i}}$ to $\Uu \Ss_{+}$ directed towards the attracting fixed point of $g_{i}$ in $\partial \Ss_{+}$. 


Recalling \eqref{eqn:b+}, we estimate $\vb_{g_{i}, -}(x)$. 
We give an outline of the rest of the long proof of Theorem \ref{thm:unifest}
starting from Section \ref{subsub:summing}:   
\begin{itemize}
\item[(I)] First, we estimate
the last term in the integral \eqref{eqn:b+} for $\omega =-$.   
\item[(II)] We estimate the contribution of $\eta| \Sf_C \setminus E$ of 
the integral \eqref{eqn:b+} for $\omega =-$.
\item[(III)] We estimate the contribution of the arcs in ${\mathscr{H}}$
\begin{itemize} 
	\item[(a)] We estimate the contribution of the arc when it is put into a standard position.
	\item[(b)] We obtain the relationship of the contributions to the arc in 
	the standard position and actual one by Lemma \ref{lem:etazeta}. 
	\item[(c)] We estimate the comparisons of sizes by length.
	\end{itemize}   
\item[(IV)]  \purple{ Then we sum these results to estimate the integral \eqref{eqn:b+} for $\omega =-$.}
\item[(V)] \purple{In Section \ref{subsub:ratiobound}, we show that $\alpha(g_i) \ra \infty$ and  $\alpha(g_i)/\llrrV{\vb_{g_i, -}} \ra \infty$ as $l_{\Ss_+}(g_i) \ra \infty$.}
\item[(VI)] \purple{Finally, we estimate the asymptotic direction as the last item
in Section \ref{subsub:direction}. }
\end{itemize}


Let $(x, \bu) \in \Uu \mathscr{K}$. 
The arc $\Psi((x, \bu), [0, t])$ is a geodesic passing $\Uu \mathscr{K}$. 
We choose $x_{i} \in \mathscr{K} \cap l_{g_{i}}$ for each $i$ 
and the unit vector $\bu_{i}$ at $x_{i}$ in the direction of $\hat l_{g_{i}}$. 
We let $t_{g_{i}} > 0$ be so that $\Psi((x_{i}, \bu_{i}), [0, t_{g_{i}}])\subset l_{g_{i}}$ corresponds
to the closed geodesic corresponding to $g_{i}$.  


Let $\Uu \Sf_C$ denote the unit tangent bundle over $\Sf_C$. 
\begin{itemize} 
\item We denote 
by ${{\mathscr{H}}}_{i, 1}, {{\mathscr{H}}}_{i, 2}, \dots$, the components  of 
${\mathscr{H}}$ meeting 
$\Pi_{\Ss_{+}}(\Psi(x_{i}, \vu_i), t))$ as $t$ increases. 
\item Let $\vp_{i,j} dx_{i, j}$ denote  
$\tilde \eta|{{\mathscr{H}}}_{i, j}$ where 
$\llrrparen{\vp_{i,j}}$ is the parabolic fixed point in
the boundary of ${{\mathscr{H}}}_{i, j}$.
\item Let $t_{i, j}$, $0 < t_{i, j}< t_{g_{i}},$ be the time the geodesic $\Psi((x_{i}, \bu_{i}), t)$ enters $\Uu {{\mathscr{H}}}_{i, j}$, 
and $\hat t_{i, j}$ the time it leaves $\Uu {{\mathscr{H}}}_{i, j}$ for the first time after $t_{i, j}$. 
\item We denote $I_{i, j} = [t_{i, j}, \hat t_{i, j}]$. 
\end{itemize} 

(I) We estimate 
$\llrrV{(\Idd - \mathcal{L}(g_i))({\Pi_{-, x_g}}\circ s_{\tilde \eta}(x_i))}_E$
for $g \in \Gamma_{\mathscr{K}}$ { from \eqref{eqn:b+}}\/\newch{:}
The matrix of $\mathcal{L}(g_i)$ with the basis 
$\vv_{+, (x_i, \vu_i)},  \vv_{0, (x_i, \vu_i)}, \vv_{-, (x_i, \vu_i)}$
is a diagonal matrix with 
entries 
\[\lambda_{1}(g_i), 1, 1/\lambda_{1}(g_i).\] 
Hence, the above is given by 
\begin{equation} \label{eqn:lastterm}
\llrrV{\left(1-\frac{1}{\lambda_{1}(g_i)}\right)({\Pi_{-, x_g}}\circ s_\eta(x_i))}_E < C_{\mathscr{K}}
\end{equation}
where we have a uniform constant $C_{\mathscr{K}}$ depending only on $\mathscr{K}$ 
by Lemma \ref{lem:metricbound} and \eqref{eqn:length}
since $\lambda_1({g_i}) > 1$ and $\llrrV{s_{\tilde \eta}}_E|\mathscr{K}$ is bounded 
by a constant depending only on $\mathscr{K}$. 


(II) Define 
\[N(\Sf_{C} \setminus E):= \max \{\llrrV{\eta(\bu)}_{\rf}| \bu \in \Uu\Sf_{C} \setminus E\}.\] 
We have 
\begin{equation}\label{eqn:intest0}
\bigg|\bigg| \int_{[0, t_{g_i}] \setminus \bigcup_{j} I_{i, j}} \bbD\Psi((x_{i}, \bu_{i}), t)^{-1}\left(\tilde \eta_{-}\left(\frac{d\Psi((x_{i}, \bu_{i}), t)}{dt}\right)\right) dt \bigg|\bigg|_{\rf} < C_{1}
\end{equation} 
for $C_1 < \infty$ 
by the second part of \eqref{eqn:PhitC} applied to 
$\bbD\Psi((x_{i}, \bu_{i}), t)^{-1}$ and the integrability of the exponential function. 
Here, $C_1= C_1(N(\Sf_{C} \setminus E))$ depends only on $N(\Sf_{C} \setminus E)$. 

Since these integrals have values in the fibers over $\mathscr{K}$, 
and $\llrrV{\cdot}_{\rf}$ and $\llrrV{\cdot}_E$ are 
uniformly compatible over $\mathscr{K}$, we have 
\begin{equation}\label{eqn:intest1}
\bigg|\bigg| \int_{[0, t_{g_i}]\setminus \bigcup_{j} I_{i, j}} \bbD\Psi((x_{i}, \bu_{i}), t)^{-1}\left(\tilde \eta_{-}\left(\frac{d\Psi((x_{i}, \bu_{i}), t)}{dt}\right)\right) dt \bigg|\bigg|_{E} < C_{2}
	\end{equation} 
	for $C_2 < \infty$. (See Remark \ref{rem:dualpict}.) 
	Hence, $C_2$ depends only on $\mathscr{K}$ and 
	$N(\Sf_{C} \setminus E)$.  We write $C_2 = C_2(\mathscr{K}, N(\Sf_{C} \setminus E))$. 

 (III)  For each $I_{i,j}$, we define for 
the maximal geodesic segment in $l_{g_{i}} \cap {{\mathscr{H}}}_{i, j}$ 
 \begin{align} \label{eqn-beta}
 \eta_{i,j} &:=  \Pi_{\Ss_{+}}\circ  \Psi((x_{i}, \bu_{i}), I_{i,j}) \subset {{\mathscr{H}}}_{i,j}\, \hbox{ and }  \nonumber \\
 \vb_{g_{i},- }(\eta_{i,j}) &:=  \int_{I_{i,j}} \bbD\Psi((x_{i}, \bu_{i}), t)^{-1}\left(\tilde \eta_{-}\left(\frac{d\Psi((x_{i}, \bu_{i}), t)}{dt}\right)\right) dt.
 \end{align} 
We now estimate $\vb_{g_i, -}$ contributed by $I_{i,j}$ by looking at the situation of  \eqref{eqn:cd}.
\begin{figure}[h]
	
	\begin{center}
		\includegraphics[height=7cm, trim={0.5cm 2.0cm 0.1cm 1.0cm}, clip]{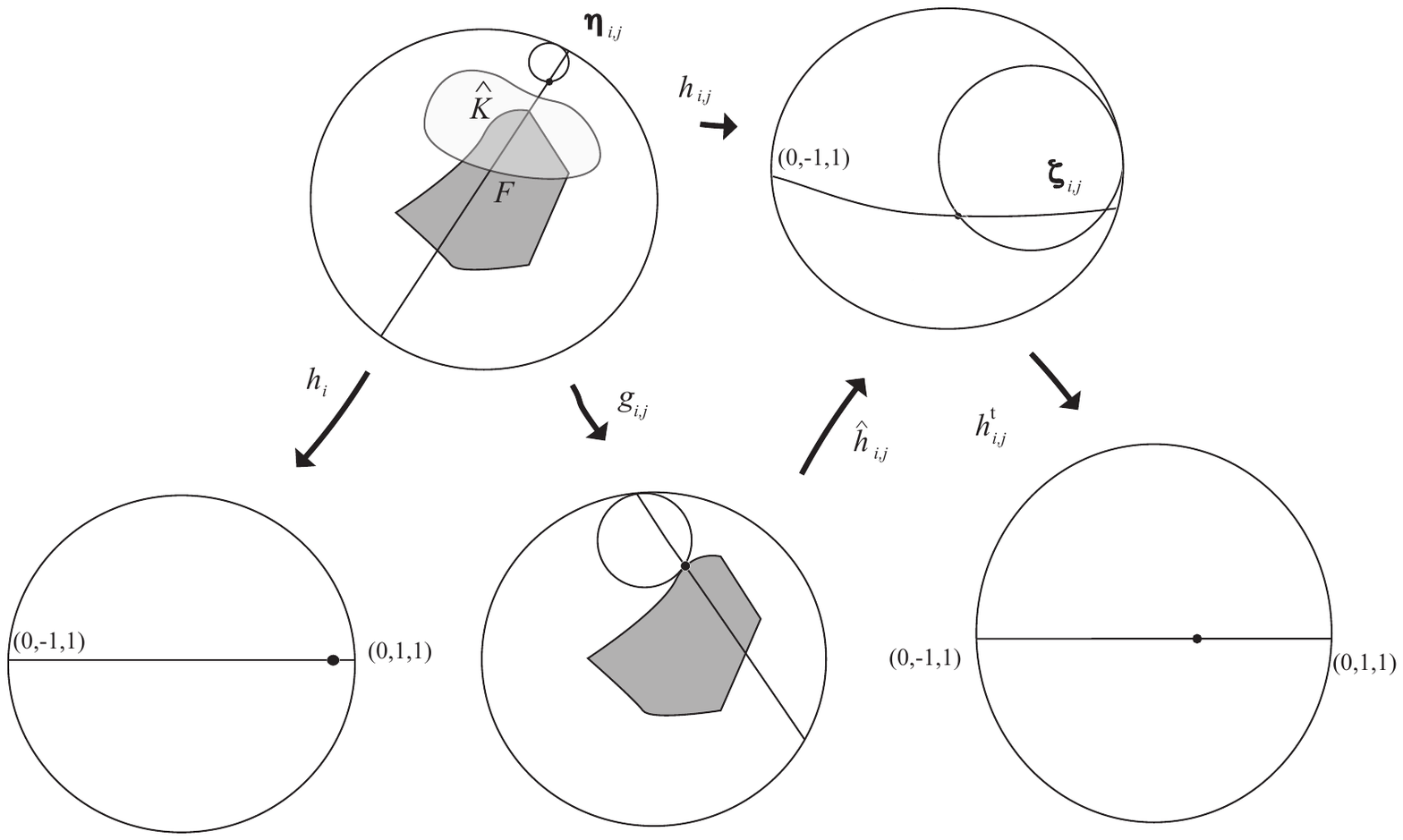}
	\end{center} 
	\caption{ $g_{i,j} \in \Gamma$ moves $q_{i, j}$ to a point of $F$. 
		$\hat h_{i, j}$ sends ${{\mathscr{H}}}_{i, j}$ to the standard horodisk $D$, 
		$h_i$ is a normalization map of $l_{g_i}$, 
		$h^t_{i, j}$ the normalization map for $h_{i, j}(l_{g_i})$, 
		{where} $h_{ij} = \hat h_{ij}\circ g_{ij}$. 
		See Definitions \ref{defn:threem} and \ref{defn:twom}. 
		The black dots indicate the images of $q_{i, j}$.}
	\label{fig:fived2}
\end{figure}

Recall  the fundamental domain $F$ of $\CH(\Lambda_{\Gamma, \Ss_+})\setminus {{\mathscr{H}}}$ covering $\Sf_{C}\setminus E$. 
Let $p_{i}$ denote the beginning point in $\partial \Ss_{+}$ of  $l_{g_{i}}$ in $\Ss_{+}$, and $p'_{i}$ denote the forward endpoint of $l_{g_{i}}$ in $\partial \Ss_{+}$.
Let $q_{i, j}$ denote the beginning point of $\eta_{i, j}$ itself
and $\vu_{i,j}$ the unit tangent vector to $l_{g_i}$ at the point $x_i$ in $\mathscr{K}$.

\begin{definition} \label{defn:threem} 
We define three maps and two others slightly later. 
\begin{description}
	\item[$g_{i,j}$:] There is an element $g_{i, j} \in \Gamma$ so that 
	$g_{i,j}(q_{i,j}) \in F$, and 
	\[g_{i, j}({{\mathscr{H}}}_{i, j}) = {{\mathscr{H}}}_{k} \hbox{ for } k = 1,\dots, m_{0}
	\hbox{ and } g_{i,j}(q_{i,j}) \in F \cap \clo({{\mathscr{H}}}_{k}).\]
	\item[$\hat h_{i,j}$:] 
	Since $\{ {{\mathscr{H}}}_{1}, \dots, {{\mathscr{H}}}_{m_{0}}\}$ is finite, we can put ${{\mathscr{H}}}_{k}$ to the standard horodisk $D$ by a uniformly bounded sequence 
	$h'_{i, j}$ of elements 
	of $\SO(2, 1)^o$. 
	Since $g_{i, j}(q_{i, j})$ is in a compact set
	$F \cap  \clo({{\mathscr{H}}}_{k})$, 
	$h'_{i, j}(g_{i, j}(q_{i, j}))$ is in a uniformly bounded subset of 
	$\Uu \partial_h D$. Hence, we can put
	$h'_{i, j}(g_{i,j}(p_{i}))$ to be $\llrrparen{0,-1,1}$ by a bounded sequence $h''_{i, j}$
	of parabolic elements fixing  $\llrrparen{0,1,1}$. 
	Let $\hat h_{i, j} = h''_{i, j}\circ h'_{i, j}$. Then 
	\[ \hat h_{i, j}({\mathscr{H}}_{i,j}) = D, 
	\hat h_{i, j}(g_{i, j}(p_i)) = \llrrparen{0,-1,1},\] and 
	$\hat h_{i,j}$ in a uniformly bounded set of elements of $\SO(2, 1)^{o}$ 
	not necessarily in $\Gamma$. 
	This is called a {\em normalization map.}
	(There is a bound on the size of $\hat h_{i, j}$ depending only on $F$.)
	\item[$h_{i,j}$:] Let $h_{i,j}= \hat h_{i,j}\circ g_{i,j}$. 
\end{description} 
\end{definition} 


The image \[\zeta_{i,j} = h_{i,j}(\eta_{i,j})\] 
satisfies the premise of Lemma \ref{lem:estimation}.
(See Figure \ref{fig:fived2}.)


(III)(a) We define
\begin{multline}\label{eqn-bzeta}
 \vb_{g_{i},- }(\zeta_{i,j}) :=  \\
 \int_{I_{i,j}} \bbD\Psi(h_{i,j}(q_{i,j}), t - t_{i,j})^{-1}
 \left(h_{i, j}^{-1\ast}\tilde \eta_{-}\left(\frac{d\Psi(h_{i,j}(q_{i,j}), t - t_{i,j})}{dt}\right)\right) dt.
 \end{multline} 
\purple{
Proposition \ref{prop:mintegrals} 
implies that 
\begin{equation}\label{eqn:zetai-pre}
\llrrV{\vb_{g_{i},- }(\zeta_{i,j}) }_{E} \leq \mu_k   r(\zeta_{i, j}). 
\end{equation}
Since there are only finitely many values of $\mu_k$s, 
\begin{equation}\label{eqn:zetai}
\llrrV{\vb_{g_{i},- }(\zeta_{i,j}) }_{E} \leq  \mu r(\zeta_{i, j}). 
\end{equation}
}



(III)(b) We compute the actual contribution for $\eta_i$.
We diagram the flow of the point $w \in \Uu \Ss_+$ and the action of 
the isometry	$g$ not necessarily in $\Gamma$:
\begin{equation} \label{eqn:cd}
\begin{CD} 
	w @>{-t}>> \Psi(w, -t) \\ 
	 @VV{g}V       @VV{g}V \\
	g(w) @>{-t}>> \Psi(g(w), -t).
\end{CD}
\end{equation}

\begin{lemma} \label{lem:etazeta} 
\begin{equation}\label{eqn-etazeta}
\Pi_{\bR^{2,1}}\left( \bbD h_{i, j}^{-1}( \vb_{g_{i},- }(\zeta_{i, j}) )\right) 
= \Pi_{\bR^{2,1}}\left(\vb_{g_{i},- }(\eta_{i, j})\right).
\end{equation}
\end{lemma} 
\begin{proof} 
 Since the flow commutes with isometry group action on $\Uu \Ss_{+}$, 
 we have by considering \eqref{eqn:cd} and the triviality of actions in the fibers
  \begin{multline} \label{eqn:commDg} 
  \bbD g (\bbD \Psi(w, -t)(\vv))= (\bbD g \circ \bbD\Psi(w, -t) \circ \bbD g^{-1})\circ \bbD g(\vv)= \\ \bbD\Psi(g(w), -t)\circ \bbD g(\vv) 
  \hbox{ for } w \in \Uu \mathscr{K}, \vv \in \bR^{2,1}, 
  g \in \SO(2, 1)^{o}.
  \end{multline} 
We apply $\bbD h^{-1}_{i, j}$ to \eqref{eqn-bzeta}.
 Since  $\Psi(x, t)^{-1} = \Psi(x, -t)$,
 we obtain by  \eqref{eqn:commDg} 
  \begin{multline} \label{eqn:res1} 
 \bbD h_{i, j}^{-1}\left(\bbD\Psi(h_{i, j}(q_{i, j}, \vu_{i,j}), t - t_{i, j})^{-1}
 \left(h_{i, j}^{-1\ast}\tilde \eta_{-}\left(\frac{d\Psi(h_{i,j}(q_{i, j}, \vu_{i, j}), t - t_{i, j})}{dt}\right)\right) \right)  \\
=	 \bbD\Psi((q_{i, j}, \vu_{i,j}), t - t_{i, j})^{-1} 
	 \bbD h_{i, j}^{-1} \left(h_{i, j}^{-1\ast}\tilde \eta_{-}\left(\frac{d\Psi(h_{i,j}(q_{i, j}, \vu_{i, j}), t - t_{i, j})}{dt}\right)\right).
	\end{multline} 
	The above \eqref{eqn:res1} equals by \eqref{eqn:Action}
	\begin{multline} \label{eqn:res2} 
\bbD\Psi((q_{i, j}, \vu_{i,j}), t - t_{i, j})^{-1} 
\bbD h_{i, j}^{-1}\left(h_{i, j}^{-1\ast}\tilde \eta_{-}\left(\frac{d\Psi(h_{i,j}(q_{i, j}, \vu_{i, j}), t - t_{i, j})}{dt}\right)\right) \\
=	\bbD\Psi((q_{i, j}, \vu_{i,j}), t - t_{i, j})^{-1} 
\left(\tilde \eta_{-}\left(Dh_{i, j}^{-1}\left(\frac{d\Psi(h_{i,j}(q_{i, j}, \vu_{i, j}), t - t_{i, j})}{dt}\right)\right)\right).
\end{multline}
By the definition of differentials and \eqref{eqn:cd}, we obtain
\begin{multline} \label{eqn:diff} 
Dh_{i, j}^{-1}\left(\frac{d\Psi(h_{i,j}(q_{i, j}, \vu_{i, j}), t - t_{i, j})}{dt}\right) \\
= \frac{d(h_{i,j}^{-1}\circ \Psi)((h_{i,j}(q_{i, j}, \vu_{i, j}), t - t_{i, j})}{dt} = \frac{d \Psi((q_{i, j}, \vu_{i, j}), t - t_{i, j})}{dt}. 
\end{multline}
Above \eqref{eqn:res2} equals by \eqref{eqn:diff} 
\begin{equation}\label{eqn:res3} 
	\bbD\Psi((q_{i, j}, \vu_{i,j}), t - t_{i, j})^{-1} 
\left(\tilde \eta_{-}\left(\frac{d\Psi((q_{i, j}, \vu_{i, j}), t - t_{i, j})}{dt}\right)\right).
\end{equation} 
	Since	$\Psi((x_{i}, \vu_{i}), t_{i,j}) = (q_{i,j}, \vu_{i, j})$,
	 \eqref{eqn:res3} equals 
\begin{multline} 	\label{eqn:res4} 
 \bbD\Psi(\Psi((x_i, \vu_i), t_{i, j}), t- t_{i, j})^{-1}\left(\tilde \eta_{-}\left(\frac{d\Psi(\Psi((x_i, \vu_i), t_{i, j}), t-t_{i, j})}{dt}\right)\right) \\
  = \bbD\Psi((x_i, \vu_i), t)^{-1}\left(\tilde \eta_{-}
 \left(\frac{d\Psi((x_i, \vu_i), t)}{dt}\right)\right) \\
 \hbox{ for every }  t \in [t_{i, j}, \hat t_{i, j}], (x_i, \vu_i) \in \Uu \mathscr{K}, 
 \end{multline}
where we multiplied by $\bbD\Psi((x_i, \vu_{i}), t_{i, j})^{-1}$ which 
is $\Idd$ on the fibers
to the left side. 
Integrating \eqref{eqn:res1} and the last line of \eqref{eqn:res4} for $[t_{i, j}, \hat t_{i, j}]$, we 
proved \eqref{eqn-etazeta}. 
\end{proof} 



(III)(c) Now, we compare the contributions of these arcs. 
Now, $h_{i, j}(q_{i, j}) \in \Uu \partial_h D$ is in a uniformly bounded subset 
$F', \Uu F \subset F'$, independent of $i, j$,
of $\Uu \Ss_{+}$ since $h_{i, j}(p_{i}) = \llrrparen{0, -1, 1}$ and the complete geodesic containing $h_{i, j}(\eta_{i, j})$ 
passes the standard horodisk $D$. 

Thus, $h_{i, j}(l_{g_{i}})$ is uniformly bounded from the line $\kappa$ 
in $\Ss_{+}$ connecting $\llrrparen{0, -1, 1}$ to $\llrrparen{\vj + \vk}$, 
oriented towards $\llrrparen{0,1,1}$. 
Let $\hat \kappa$ denote the lift of $\kappa$ to 
$\Uu \Ss_+$ taking the direction towards $\llrrparen{\vj + \vk}$. 

\begin{definition}\label{defn:twom} 
We define two additional normalization maps: 
\begin{description}
	\item[$h^{\dagger}_{i,j}$:] We take a uniformly bounded element $h^{\dagger}_{i,j}$ of  
	$\SO(2, 1)^{o}$ so that $h^{\dagger}_{i,j}(h_{i, j}(l_{g_{i}})) = \kappa$
	and $h^{\dagger}_{i,j}(h_{i, j}(q_{i, j})) = \llrrparen{0,0,1}$. 
	\item[$h_{i}$:] Since $l_{g_{i}}$ is a geodesic passing $\mathscr{K}$, we take a uniformly bounded element $h_{i}$ of $\SO(2, 1)^{o}$
	so that $h_{i}(l_{g_{i}}) = \kappa$ and $h_i(x_i) = \llrrparen{0,0,1}$ 
	without changing the orientation. 
	(The bound only depends on $\mathscr{K}$.)
\end{description} 
\end{definition}

Then 
\[h_{i} \circ h_{i, j}^{-1} \circ h^{\dagger,-1}_{i,j}(h^{\dagger}_{i,j}(\zeta_{i, j})) = h_{i}(\eta_{i, j})\] 
and $h_{i} \circ h_{i, j}^{-1}\circ h^{\dagger,-1}_{i,j}$ acts on $\kappa$.  
\begin{itemize}
\item Under $h_{i} \circ h_{i, j}^{-1} \circ h^{\dagger, -1}_{i,j}$, $h^{\dagger}_{i,j} \circ h_{i, j}(q_{i, j})$ goes to a point $h_{i}(q_{i, j})$. 
\item \begin{equation} \label{eqn:htij} 
d_{\Ss_+}(h^{\dagger}_{i,j} \circ h_{i, j}(q_{i, j}), h_{i}(q_{i, j})) = t_{i, j} 
\end{equation}
since	$h_i(x_i) = \llrrparen{0,0,1} = 
h^{\dagger}_{i,j} \circ h_{i, j}(q_{i, j})$ and 
the $d_{\Ss_+}$-length of the arc from 
$x_i$ to $q_{i, j}$ is $t_{i, j}$ 
which is also the $d_{\Ss_+}$-length of 
the arc from $h_i(x_i)$ to $h_i(q_{i, j})$. 
\end{itemize} 

By \eqref{eqn:length} and \eqref{eqn:htij}, the eigenvalue of $\mathcal{L}(h_{i} \circ h_{i, j}^{-1} \circ h^{\dagger, -1}_{i,j})$ at the eigenvector $(0, 1, -1)$ 
is $\exp(-t_{i, j}/2)$.
Since 
\[ \Pi_{\bR^{2,1}}(\vb_{g_{i},- }(\eta_{i, j}))  = \Pi_{\bR^{2,1}}(h_{i, j}^{-1 \ast}(\vb_{g_{i},- }(\zeta_{i, j}))),\] 
$\mathcal{L}(h_{i} \circ h_{i, j}^{-1} \circ h^{\dagger, -1}_{i,j})$ sends the $\bR^{2,1}$-vector 
\[\Pi_{\bR^{2,1}}(\mathcal{L}(h^{\dagger}_{i,j})(\vb_{g_{i},- }(\zeta_{i, j})))\in 
\langle (0,-1, 1) \rangle \hbox{ to } 
\Pi_{\bR^{2,1}}(\mathcal{L}(h_{i})(\vb_{g_{i},- }(\eta_{i, j})))\in \langle (0,-1, 1)\rangle\]
by multiplying by $\exp(-t_{i, j}/2)$.  
Since $h^{\dagger}_{i,j}$ and $h_{i}$ are uniformly bounded
depending only on $\mathscr{K}$ and $F$, 
we obtain
\begin{equation}\label{eqn:intest2}
\tilde C(F, \mathscr{K})\exp(-t_{i, j}/2)\llrrV{\vb_{g_{i},- }(\zeta_{i, j}) }_{E} \geq \llrrV{ \vb_{g_{i}, -}(\eta_{i, j})}_{E}.
\end{equation} 
for a constant $\tilde C(F, \mathscr{K}) > 0$ depending only on 
$\mathscr{K}$ and $F$. 


(IV) We sum up the contributions. 
Hence, $\frac{1}{R(\zeta_{i, j})} < 1$. 
By \eqref{eqn:lastterm}, \eqref{eqn:intest1}, \eqref{eqn:zetai}, 
 \eqref{eqn:intest2} 	and Proposition \ref{prop:mintegrals},
  we estimate the upper bound 
 depending only on $E, \mathscr{K}, \eta|\Sf_{C} \setminus E$:
\begin{multline}\label{eqn:intest3p}
\llrrV{\vb_{g_{i}, -} }_{E}\leq \tilde C(F, \mathscr{K})\sum^{m_i}_{j} \exp(-t_{i, j}/2) 
\purple{\left( \mu \frac{r(\zeta_{i,j})(1 + 4 R(\zeta_{i,j})^2)}{2 \sqrt{2} R(\zeta_{i,j})^2} \right) }\\ 
 +C_{2}(\mathscr{K}, N(\Sf_{C} \setminus E)) + C_{\mathscr{K}} \\
 \leq  \tilde C(F, \mathscr{K})\sum_{j}^{m_i} \exp(-t_{i, j}/2)
 \purple{\left( 4\mu r(\zeta_{i, j})\right)}\\ 
 +C_{2}(\mathscr{K}, N(\Sf_{C} \setminus E)) + C_{\mathscr{K}} 
\end{multline} 
\purple{since $R(\zeta_{ij}) \geq 1$.}

\subsubsection{\purple{$\alpha(g_i) \ra \infty$ and 
$\frac{\alpha(g_i)}{\llrrV{\vb_{g_i, -}}} \ra \infty$.}} \label{subsub:ratiobound} 

In Step (V), 
we will prove that $\alpha(g_i) \ra \infty$ and 
$\alpha(g_i)/\llrrV{\vb_{g_i, -}} \ra +\infty$ provided 
$l_{\Ss_+}(g_i) \ra \infty$ using the fact that we can absorb many negative 
uncertainties during perturbation into long edges in the cusps using Lemma 
\ref{lem:largecusp}. 

We can do this by showing that every subsequence has a subsequence converging to $+\infty$.  
\newch{
We give an outline of the step (V). 
\begin{itemize} 
\item[(i)] First, we will choose some constants such as $\eps, \delta, R_0$ 
sufficiently small or large. 
\item[(ii)] Let $g_i$ denote a closed geodesic.
We \newch{replace} the maximal segment $\zeta$ in a cusp neighborhood 
with $r(\zeta)> R_0+ h/2$ with one $\zeta'$ with the same endpoints but with 
$R_0  < r(\zeta') \leq R_0+h/2$.  We denote the result by $\tilde g_i$. 
\item[(iii)] Then we find a closed geodesic $\hat g_i$ freely homotopic to $\tilde g_i$. 
Then we estimate $|\alpha(\tilde g_i) - \alpha(\hat g_i)|$ in terms of the constant times 
the number of components of the above arcs in \eqref{eqn:difference}. 
This constant is bounded since $R_0 \delta = 2$ by our choice below. 
\item[(iv)] This is the final step. 
$\alpha(g_i)$ is bounded below by $\alpha(\hat g_i)$ plus constant times
the sum of $r(\zeta)^2$. Then we use the standard Schwartz inequalities.
\end{itemize} 
}

\begin{definition}\label{defn:arcs} 
Let $g_i$ also denote the arclength-parameterized 
closed geodesic in $\Sf$ whose lift $l_{g_i}$ passes
a fixed compact set $\mathcal{K}$ in $\Ss_+$. 
\newch{
Let $J$ be  \newch{the index set} of mutually disjoint subintervals $I_i \subset I$ and 
 $\alpha_i  := g_i| I_i$, 
By $g_i \setminus \bigcup_{i\in J}\alpha_i$,  we mean the map 
$g_i| I \setminus \bigcup_{i\in I} I_i$. }
\end{definition} 


We denote by $E_\eps$ the set obtained by decreasing $E$ inward by $\eps$ when 
$\eps > 0$ and the $(-\eps)$-neighborhood of $E$ when $\eps < 0$. 
We will assume that 
\newch{$E_{-1/2}$ is still a cusp-neighborhood, and}  
$\eta|E_{-1/2}$  is still standard cusp $1$-form for each component 
by taking sufficiently smaller $E$ if necessary. 

We denote by $\mu_{i, j}$ the cusp coefficient for the cusp neighborhood 
that $\zeta_{i, j}$ goes into. There are only finitely many values.
We  assume that the horospherical lengths of all cusp neighborhood components of $E$ equal $h$. 
Let $C_{\Sf \setminus E}$ denote the neutral factor of
the compact set $\Sf \setminus E$.  

\newch{We remark that following constants depend only on the two constants 
$h$ and $C_{\Sf\setminus E_1}$.
There is no obstruction for the following choices}. 

(i) The first step is to decide on constants to be used later: 
\blue{
\begin{itemize} 
\item Choose $\delta > 0$ so that $0< \delta < 1/40$ 
by Lemma \ref{lem:perturb} and let \newch{$\eps = 7\delta$}. 
\item We also require \newch{$\delta < \mu/(7C_{\Sf \setminus E_1})$}.
\item \newch{Also assume $6h \eps < 1 $, $\eps < 1/8$, and $R_0 > 10$. }
\item  We require $\delta$ to be given by 
$\delta := 2/R_0$ 
by taking $R_0$ sufficiently large and $\delta$ sufficiently small. 
By Lemma \ref{lem:vertical}, the angle that $\zeta_{i, j}$ with $r(\zeta_{i, j})\geq R_0$ 
makes with the vertical line is $< \delta$ in the upper half-space model. 
\item $R_0$ is a constant satisfying all conclusions for the variable $R_1$ in   
 Lemma \ref{lem:largecusp} for $C > 222 \frac{\mu}{\mu_{\min}}$. 
For simplicity, we assume $R_0 > 10$. 
\end{itemize} 
}


(ii) We will replace very long $\zeta_{i, j}$ in $g_i$ with ones that are outside some 
cusp neighborhood: 
We denote by $\zeta_{i, j}$ the sequence of maximal geodesics in $g_i$ 
 going into $E$. 
\newch{
We denote by $J_{i, t}$ the set of $\zeta_{i, j}$ with $r(\zeta_{i, j}) > t$ 
for $t \geq 0$.
}

For each $\zeta_{i, j}$ in $J_{i, R_0+ h/2}$, we take a maximal geodesic 
$\hat \zeta_{i, j}$ with the same endpoints but with 
$R_0  \leq r(\hat \zeta_{i, j}) < R_0+ h/2$ 
since we can decrease the $r(\zeta)$-values by $h/2$ times integers by wrapping 
a smaller number of times around the cusps.  
Since the geodesics are unique up to homotopy classes relative to endpoints, 
the homotopy class of $\hat \zeta_{i, j}$ is, of course, different from 
$\zeta_{i, j}$ relative to the endpoints. 
Thus, we obtain for $\zeta_{i, j} \in J_{i, R_0+ h/2}$, 
\begin{multline} \label{eqn:replacezeta} 
\alpha(\zeta_{i, j}) - \alpha(\hat \zeta_{i, j}) \geq \mu_{i, j}\delta(R_0+ h/2),  \\
 \alpha(\zeta_{i, j}) - \alpha(\hat \zeta_{i, j}) - \mu_{i, j}\eta_0  \geq C^{(4.6)}_{R_0+ h/2, C'} \mu_{i, j} r(\zeta_{i, j})^2, C^{(4,6)}_{R_0+h/2, C'} >0
\end{multline} 
by Lemma \ref{lem:largecusp} where $\eta_0 < C'$. 

(iii) The third step is to estimate the relationship between
$\alpha$-values for $g_i$ \newch{and the closed curves $\hat g_i$ and $\tilde g_i$
to be constructed}\/:
Let $\hat g_i$ denote  the closed curve obtained by $g_i$ removing $\zeta_{i, j}$ and adding $\hat \zeta_{i, j}$ for each 
$\zeta_{i, j} \in J_{i, R_0+ h/2}$. 
\newch{
By Lemma \ref{lem:vertical}, $\hat g_i$ has turning angles $< \delta= 2/R_0$
at each endpoint of maximal geodesic segments by Lemma 
\ref{lem:vertical}. }
We define 
\[\hat \alpha(\hat g_i) :=
\alpha\left(g_i \setminus \bigcup_{\zeta \in J_{i, R_0+ h/2}} \zeta\right) + 
\sum_{\zeta \in J_{i, R_0+ h/2}} \alpha(\hat \zeta).\]

There exists a closed geodesic $\tilde g_i$ homotopic to $\hat g_i$
which is in the $\eps$-neighborhood of $\tilde g_i$
 \newch{for $\eps = 7\delta$}  by Lemma \ref{lem:perturb}. 
Let $E_{R_0/2+ h/4 + \eps}$ denote the cusp neighborhood obtained by moving $E$ inside by $R_0 /2 + h/4 +\eps$. 
Then both $\hat g_i$ and $\tilde g_i$ are in $\Sf\setminus E_{R_0/2 + h/4+ \eps}$. 

Define $\hat J_{i, 0}$ the subset of $J_{i, 0}$ of  consisting of arcs $\zeta_{i, j}$ where where $d_{\Ss_+}$-lengths are strictly bigger than \newch{$5/4$}. 
For every arc in $J_{i, 0} \setminus \hat J_{i, 0}$, the arcs are in $\Sf \setminus E_{5/8}$. 
We will not remove these from $g_i$ in the following because of this. 
By skipping these,  we have
\begin{equation} \label{eqn:skip} 
|\hat t_{i, j} - t_{i, j}|\geq 5/4, |t_{i, j+1} - \hat t_{i, j}| \geq C^{\text{\eqref{eqn:distH}}}_E\geq  5/4  \hbox{ for every } 
\zeta_{i, j} \in \hat J_{i, 0}
\end{equation}  
where $C^{\text{\eqref{eqn:distH}}}_E $ is from \eqref{eqn:distH}. 

Each maximal geodesic $\zeta_{i, j} \in \hat J_{i, 0}\setminus J_{i, R_0+ h/2}$ in $E$ of $g_i$ 
goes to a geodesic $\tilde \zeta_{i, j}$ in 
\newch{$E_{-1/8}$ of $\tilde g_i$
by the perpendicular projection which moves points by distances $< \eps <1/8$. }
We obtain two distances
\[d_{i, j, \pm} := d_{\Ss_+}(\partial_\pm \zeta_{i, j}, \partial_{\pm} \tilde \zeta_{i, j}) .\]
\newch{These are less than $\eps$ by Lemma \ref{lem:divergence}
since each endpoint of $\zeta_{i, j}\in J_{i, R_0+ h/2}$ moves less than $\eps$.}
The corresponding endpoints are at most distance 
$d_{i, j, \pm}$ apart, which are values of the divergence functions 
\newch{corresponding to} 
$\hat t_{i, j}$ and $t_{i, j}$ respectively.
Hence, their \newch{$x$-coordinate values} differ by less than $\newch{1.1}d_{i, j, \pm}$ respectively
using  \eqref{eqn:HH} as $0< \eps < 1/8$. 
\newch{By last parts of \cite{NewNormalComp} and \cite{NewNormalComp2} of 
the differences in the $\alpha$-values,}
we can estimate
for $\zeta_{i, j} \in \hat J_{i, 0} \setminus J_{i, R_0+ h/2}$, 
\begin{equation} \label{eqn:dij+-} 
|\alpha(\tilde \zeta_{i, j}) -\alpha(\zeta_{i, j})| \leq 5\mu_{i, j} (\newch{R_0+ h/2})(d_{i, j, +}/2+ d_{i, j,-}/2)
\end{equation} 
since we can put in the new $x$-coordinates and take differences in $E_{-1/2}$ 
where $\eta$ has the form of the standard cusp $1$-form.  
\newch{Here, we need to use the fact that $r > 10, \eps < 1/8, \eps < r/80$, 
$r \mapsto \sqrt{r^2+1}, r > 0,$ is distance decreasing,  and
estimates of differences of the inverses of radii of arcs using calculus.} 

We claim that 
the sum of $d_{i, j, +} + d_{i, j, -}$ for $\zeta_{i, j} \in \hat J_{i, 0}$ in 
$\tilde g_i \setminus \bigcup_{\zeta_{i, j} \in J_{i, R_0+ h/2}}  \tilde \zeta_{i, j}$ 
is less than $2$ times the sum of $d_{i, j, +}$ and $d_{i, j, -}$ over all
$\zeta_{i, j} \in J_{i, R_0+ h/2} $ which is less than $4\eps|J_{i, R_0+ h/2}|$:  
We move $g_i \setminus \bigcup_{\zeta_{i, j} \in J_{i, R_0+ h/2}}  \zeta_{i, j}$
to $\tilde g_i  \setminus \bigcup_{\zeta_{i, j} \in J_{i, R_0+ h/2}}  \tilde \zeta_{i, j}$
by perpendicular projections, and hence, the endpoints of 
$\zeta_{i, j}$ for $\zeta_{i, j}\in J_{i, R_0+ h/2}$ moving to 
$\tilde \zeta_{i, j}$ gives us the divergence functions. 
The sum of the values of the divergence functions at $t_{i, j}, \hat t_{i, j}$ for the endpoints of 
$\zeta_{i, j} \in \hat J_{i, 0}\setminus J_{i, R_0+ h/2}$ in a component of $g_i \setminus \bigcup_{\zeta \in J_{i, R_0+ h/2} }\zeta$, 
is less than \newch{$2$ times the sum of the values of} its endpoints by \eqref{eqn:skip} and Lemma \ref{lem:divergence}.
%

Since each endpoint of $\zeta_{i, j}\in J_{i, R_0+ h/2}$ moves less than $\eps$, 
we have by \eqref{eqn:dij+-}  
\begin{equation} \label{eqn:mid} 
\sum_{\zeta_{i, j}  \in \hat J_{i, 0} \setminus J_{i, R_0+ h/2} } |\alpha(\tilde \zeta_{i, j}) -\alpha(\zeta_{i, j})| 
\leq \newch{10}\mu (\newch{R_0+ h/2}) \eps|J_{i, R_0+ h/2}|.
\end{equation}


\newch{As in the third paragraph above}, 
for arcs $\hat \zeta_{i, j} \in J_{i, R_0+ h/2}$, we have 
\[ 
|\alpha(\tilde \zeta_{i, j}) - \alpha(\hat \zeta_{i, j})| \leq 5 \mu(\newch{R_0 + h/2})\eps.
\]
Hence, 
\begin{equation} \label{eqn:long} 
\sum_{\zeta_{i, j}  \in J_{i, R_0+ h/2} } |\alpha(\tilde \zeta_{i, j}) - \alpha(\hat \zeta_{i, j})| \leq 5 \mu(\newch{R_0 + h/2})\eps |J_{i, R_0+ h/2}|.
\end{equation}

For $\alpha$-values outside these, we integrate $\eta$ projected to the neutral bundle over 
\[g_i \setminus \bigcup_{\zeta_{i, j} \in \hat J_{i, 0}}  \zeta_{i, j},  
\hbox{ and } \tilde g_i \setminus  \bigcup_{\zeta_{i, j} \in \hat J_{i, 0}}  \tilde \zeta_{i, j},\] 
they all happen inside $\Sf \setminus E_{\newch{5/8}+\eps}$. 
\newch{
By Lemmas \ref{lem:divergence} and \ref{lem:perturb}, the absolute value of
the $\alpha$-value difference is 
bounded above by the neutral factor $C_{\Sf\setminus E_{\newch{5/8}+\eps}}$ times 
$2$ times the sum of 
perpendicular distances
at the end points of the corresponding arcs.
These  values are from endpoints of arcs in  $\hat J_{i,0}$ considered by 
a pararagraph above \eqref{eqn:dij+-}  or endpoints of 
arcs in $J_{R_0 + h/2}$.  
Hence the absolute value of 
the $\alpha$-value difference is bounded above by $4\eps C_{\Sf\setminus E_{5/8+\eps}}|J_{i, R_0 + h}|$.
}

\newch{
Hence, we obtain by 
\eqref{eqn:mid} and \eqref{eqn:long} and the assumptions in (i). 
\begin{multline} \label{eqn:difference} 
|\alpha(\tilde g_i)- \hat\alpha(\hat g_i)| \leq |J_{i, R_0+ h/2}|( 4C_{\Sf\setminus E_{5/8+\eps}} \eps 
+ 15\mu (R_0 + h/2 ) \eps), \\ 
15(R_0\eps + h\eps/2 ) = 15\times 2 \times 7  + 15 h \eps/2  < 218
\end{multline} 
}


\newch{
(iv) Lastly, we apply the above to complete the convergences to $\infty$. 
At (i),  
we chose above a 
sufficiently small $\eps$ so that $C_{\Sf\setminus E_1}\eps < \mu$. 
Since $C_{\Sf \setminus E_{1}} \geq C_{\Sf\setminus E_{5/8 + \eps}}$, 
we obtain
\begin{multline} \label{eqn:est1} 
\alpha(g_i) = \hat \alpha(\hat g_i) + \sum_{\zeta \in J_{i, R_0+ h/2}} (\alpha(\zeta) - \alpha(\hat \zeta)) \\ 
\geq \alpha(\tilde g_i) + \sum_{\zeta \in J_{i, R_0+ h/2}} \left( \alpha(\zeta) - \alpha(\hat \zeta) - 
(4C_{\Sf\setminus E_1}\eps + 218\mu )\right)\\
\geq \alpha(\tilde g_i) +\sum_{\zeta \in J_{i, R_0+ h/2}}  \mu_{\text{min}} C^{(4.6)}_{R_0+ h/2, C'} r(\zeta)^2 \hbox{ for } C':= 222 \mu/\mu_{\text{min}},   
\end{multline} 
by Lemma \ref{lem:largecusp} and \eqref{eqn:difference}. 
}

Now we can show that $\alpha(g_i) \ra \infty$ provided $l_{\Ss_+}(g_i) \ra \infty$: 
Suppose that $l_{\Ss_+}(g_i) \ra \infty$. 
If $\alpha(\tilde g_i) \ra \infty$, then $\alpha(g_i) \ra \infty$ by \eqref{eqn:est1},
and we are done.  
Suppose that $\alpha(\tilde g_i)$ is bounded above. 
Then $l_{\Ss_+}(\tilde g_i)$ is also bounded above by Lemma \ref{lem:outsidecusp}.
\newch{Since $r(\hat \zeta_{i, j}) \geq R_0$,
we have 
$l_{\Ss_+}(\tilde \zeta_{i, j}) \geq
l_{\Ss_+}(\hat \zeta_{i, j}) - 2\eps =
2 \text{arcsinh}(R_0 ) - 2\eps$ by \eqref{eqn:HH}.   }
\newch{Since $R_0 > 10, \text{arcsinh}(10) > 2.99,  1/8> \eps$ by assumptions in (i), 
it follows that $|J_{i, R_0+ h/2}|$ is bounded above. }
Only possibility is $r(\zeta_{i, j}) \ra \infty$ for some 
members $\zeta_{i, j}$ of $J_{i, R_0+ h/2}$ in order that $l_{\Ss_+}(g_i) \ra \infty$. 
This also implies 
$\alpha(g_i) \ra \infty$ by \eqref{eqn:est1}.

\newch{ 
Now we go to the ratio limit.
Notice that 
\begin{multline} \label{eqn:pasum} 
\sum_{\zeta_{i, j} \in J_{i, 0}} (\exp(-t_{i, j}/2) r(\zeta_{i, j}))  
\leq \sum_{\zeta_{i, j} \in J_{i, R_0+ h/2}} (\exp(-t_{i, j}/2) 
r(\zeta_{i, j}) + \\
\sum_{\zeta_{i, j} \in J_{i, 0}\setminus J_{i, R_0+ h/2}} (\exp(-t_{i, j}/2) r(\zeta_{i, j}).
\end{multline}
The second term is bounded above by a constant since 
each term is bounded above. This term can be absorbed into 
$C_{\mathscr{K}}$ in \eqref{eqn:av3}. 
}

We obtain by \eqref{eqn:intest3p},
\newch{ \eqref{eqn:est1}, and \eqref{eqn:pasum}} that 
\begin{equation} \label{eqn:av3} 
\frac{\alpha(g_i)}{\vb_{g_i, -}} 
\geq \frac{ \alpha(\tilde g_i) +\sum_{\zeta \in J_{i, R_0 \newch{+ h/2}}} \mu_{\text{min}} C^{(4.6)}_{R_0+ h/2, C'} r(\zeta)^2        }{  \tilde C(F, \mathscr{K}) \sum_{\zeta_{i, j} \in J_{i, R_0+ h/2}} (\exp(-t_{i, j}/2) 
r(\zeta_{i, j}))  +C_{2}(\mathscr{K}, N(\Sf_{C} \setminus E)) + C_{\mathscr{K}}}.
\end{equation} 
If $J_{i, R_0+ h/2} = \emp$ for infinitely many $i$, 
then $\alpha(g_i) \ra \infty$ up to a choice of a subsequence  
by Lemma \ref{lem:outsidecusp}. 
\newch{Since the nominator is a sum of bounded constants, we are done.}
Suppose not and that we have a sequence such that 
$\sum_{\zeta_{i, j} \in J_{i, R_0+ h/2}} \exp(-t_{i, j}/2) r(\zeta_{i, j})) \ra 0$ as $i\ra \infty$.
Define $\hat t_i$ to be the first $t_{i, j}$ where $r(\zeta_{i, j}) > R_0+h/2$. 
This means that $\hat t_i \ra \infty$ and 
$l_{\Ss_+}(\tilde g_i) \ra \infty$ and $\tilde g_i \subset \Sf \setminus E_{R_0+ h/2 +\eps}$.
By Lemma \ref{lem:outsidecusp},  $\alpha(\tilde g_i) \ra \infty$, and we are done
for the purpose of Section \ref{subsub:ratiobound}. 

Since we need to show the result for subsequences only, 
we may assume that 
\[  C_0 \tilde C(F, \mathscr{K}) \sum_{\zeta_{i, j} \in J_{i, R_0+ h/2}} \exp(-t_{i, j}/2) r(\zeta_{i, j})) \geq 
(C_{2}(\mathscr{K}, N(\Sf_{C} \setminus E)) + C_{\mathscr{K}})\]
for a constant  $C_0 > 0$. 
Hence, we obtain
\begin{multline} \label{eqn:av4}  
\frac{\alpha(g_i)}{\llrrV{\vb_{g_i, -}}}
\geq \frac{\alpha(\tilde g_i)}{  \tilde C(F, \mathscr{K})(1+C_0) 
\sum_{\zeta_{i, j} \in J_{i, R_0+ h/2}} (\exp(-t_{i, j}/2)r(\zeta_{i, j})) }  \\ 
+ \frac{ \mu_{\text{min}} C^{(4.6)}_{R_0+ h/2, b, C'} \sum_{\zeta \in J_{i, R_0+ h/2}}  r(\zeta)^2   }{ 
 \tilde C(F, \mathscr{K}) (1+C_0) \sum_{\zeta_{i, j} \in J_{i, R_0+ h/2}} (\exp(-t_{i, j}/2)r(\zeta_{i, j})) } .
\end{multline} 

We define 
\begin{multline} 
\newch{\vec{e}_i} := (\exp(t_{i, j}/2))_{\zeta_{i, j} \in J_{i, R_0+ h/2}},  \\
\newch{\vec{r}_i} := (r(\zeta_{i, j}))_{\zeta_{i, j} \in J_{i, R_0+ h/2}} \in \bR^{|J_{i, R_0+ h/2}|}, 
\hbox{ and } \\
\llrrV{\vec{v}}_{i, R_0 + h/2} := \sqrt{ \vec{v} \cdot \vec{v} }, \vec{v} \in \bR^{|J_{i, R_0+ h/2}|}. 
\end{multline} 
Using the Schwartz inequality 
\[ \left|\sum_{\zeta_{i, j} \in J_{i, R_0+ h/2}} (\exp(-t_{i, j}/2)r(\zeta_{i, j}))  \right| 
\leq \llrrV{\newch{\vec{e}_i}}_{i, R_0+ h/2}  \llrrV{\newch{\vec r_i}}_{i, R_0+ h/2},\] 
we obtain that \eqref{eqn:av3} is bigger than equal to  $\frac{1}{\tilde C(F, \mathscr{K})}$ times 
\[ 
\frac{\alpha(\tilde g_i)}{(1+C_0) \llrrV{\vec{e}_i}_{i, R_0+ h/2}  \llrrV{\vec r_i}_{i, R_0+ h/2}} 
+ \frac{\mu_{\text{min}}  C^{(4.6)}_{R_0+ h/2, b, C'} \llrrV{\vec r_i}_{i, R_0+ h/2} }{(1+C_0)  \llrrV{\vec{e}_i}_{i, R_0+ h/2} }.
\] 
For each arc $\zeta$ in $J_{i, R_0+ h/2}$, there is a corresponding maximal geodesic 
$\zeta_f$ in $\tilde g_i$ given by the perpendicular projection and extending 
to a maximal geodesic arc in $E$. 

Using the  perpendicular projection paths at the end points and the triangle inequalities, we obtain
\begin{multline} 
l_{\Ss_+}(\tilde g_i) \geq l_{\Sf}\left(\tilde g_i
\setminus \bigcup_{\zeta \in J_{i, R_0+ h/2}} \tilde \zeta\right)
 + \sum_{\zeta \in J_{i, R_0+ h/2}} l_{\Ss_+}(\tilde \zeta)  \\
\geq  l_{\Sf\setminus E}(g_i)  + \sum_{\zeta \in J_{i, R_0+ h/2}} 
(l_{\Ss_+}(\hat \zeta) - 2\eps_+(\zeta) - 2\eps_{-}(\zeta))
\end{multline} 
where
\newch{
 $\eps_{+}(\zeta)$ amd $\eps_{-}(\zeta)$ 
respectively are the vertical projection path lengths from 
the forward and backward endponts of $\zeta \in J_{i, R_0+ h/2}$ to the corresponding ones of the arc $\tilde \zeta$ in $\tilde g_i$. 
}
\newch{
We obtain $l_{\Ss_+}(\hat \zeta) \geq  2\text{arcsinh}(R_0) > 5.8$
by \eqref{eqn:HH} as $R_0 >10$ by the assumption in (i). 
}
\newch{
Since $\eps_{\pm}(\zeta) \leq \eps < 1/8$, 
the positivity of the later terms follows. 
}



By Lemma \ref{lem:outsidecusp}, 
$\alpha(\tilde g_i) \geq c^{(1.4)}_{\Sf\setminus E_{R_0/2+ h/2 + \eps}} l_{\Ss_+}(\tilde g_i)$. 
We obtain \eqref{eqn:av3} is bigger than equal to $\frac{1}{\tilde C(F, \mathscr{K})}$ times 
\[ 
\frac{ c^{(1.4)}_{\Sf\setminus E_{R_0/2+ h/2 + \eps}} l_{\Sf\setminus E}(g_i)}{(1+C_0) \llrrV{\vec{e}_i}_{i, R_0+ h/2}  \llrrV{\vec r_i}_{i, R_0+ h/2}} 
+ \frac{\mu_{\text{min}}  C^{(4.6)}_{R_0+ h/2, b, C'} \llrrV{\vec r_i}_{i, R_0+ h/2} }{(1+C_0)  \llrrV{\vec{e}_i}_{i, R_0+ h/2} }.
\] 
Now, this is a function converging to $\infty$ 
as 
 \[\max\{l_{\Sf\setminus E}(g_i), \llrrV{\vec r_i}_{i, R_0+ h/2}\}\ra \infty.\] 

\newch{
Suppose that $l_{\Ss_+}(g_i) \ra \infty$. 
Then we claim that 
$\max\{l_{\Sf\setminus E}(g_i),  \llrrV{\vec r_i}_{i, R_0+ h/2}\}\ra \infty$: }

\newch{
Suppose that $l_{\Sf\setminus E}(g_i)$ is bounded. Then the number of 
maximal geodesic arcs of $g_i$ going into $\Sf \setminus E$ is finite 
by \eqref{eqn:distH}.
Then $r(\zeta_{i, j}) \ra \infty$ for some index $(i, j)$ as $i \ra \infty$ 
since otherwise we will have $l_{\Ss_+}(g_i)$ bounded by \eqref{eqn:HH}.
Hence, $\llrrV{\vec r_i}_{i, R_0+ h/2} \ra \infty$.  
}

\newch{
Conversely, suppose that
$\{ \llrrV{\vec r_i}_{i, R_0+ h/2} \}$ is bounded above. 
If $|J_{i, R_0+h/2} | \ra \infty$, then 
 $l_{\Sf\setminus E}(g_i) \ra \infty$ by \eqref{eqn:skip}.  
Otherwise, 
if $|J_{i, R_0+h/2} | $ is bounded, there is an upper bound to the absolute values of
the coordinates of $\vec r_i$ and 
$l_{\Ss_+}(\zeta)$ for $\zeta \in J_{i, R_0+h/2}$, 
implying the absurdity that $l_{\Ss_+}(g_i)$ is bounded above. 
}

 We are done proving the main aim of Section \ref{subsub:ratiobound}.

\subsubsection{The direction result} \label{subsub:direction} 

\purple{(VI)  We come to the last step.} 


\begin{theorem}  \label{thm:unifest} 
Assume Criterion \ref{cr:positive} and $\mathcal{L}(\Gamma) \subset \SO(2,1)^{o}$. 
%
Let $\eta$ be a $\mathcal{V}$-valued $1$-form corresponding to 
the boundary cocycle for $\Gamma$.  
Let $\mathscr{K}$ be a compact subset of 
\newch{$\CH(\Lambda_{\Gamma, \Ss_+})\setminus {\mathscr{H}}$.}
For every sequence $\{g_{i}\}$ with \oldch{$\{l_{\Ss_+}(g_{i})\} \ra \infty$} of elements of 
\hyperlink{term-GK}{$\Gamma_{\mathscr{K}}$}, 
the following hold\,{\rm :} 
\begin{itemize}
\item  
$\{\llrrV{ \vb_{g_{i}}}_{E}\} \ra \infty $.
\item \purple{$\alpha(g_i) \ra \infty$ and $\alpha(g_i)/\llrrV{\vb_{g_{i},-}}_E \ra \infty$.}
\item 
$\{\hyperlink{term-bdd}{\bdd}(\llrrparen{ \vb_{g_{i}} }, \clo(\zeta_{a_{g_{i}}}))\} \ra 0$. 
\end{itemize}
\end{theorem}
\begin{proof}
\purple{
The first item follows since otherwise $g_{i}(O)$ is in a bounded set contradicting the properness of the $\Gamma$-action.
}
\purple{
By Lemma \ref{lem:metricbound},
we may also assume that 
\begin{equation} \label{eqn:vecx+} 
\{\vv_{+, (x_i, \vu_i)}\} \ra \vv_{+}, \bnu_{g_i} \ra \bnu,
\hbox{ and } \{\vv_{-, (x_i, \vu_i)}\} \ra \vv_-
\end{equation} 
for an independent set of 
vectors $\vv_+, \bnu, \vv_- $ by choosing subsequences if necessary.
These are all positively oriented in $\Lspace$. 
Let $\mathscr{C}_{\infty}$ denote the matrix with columns $\vv_+,  \bnu,$ 
and $\vv_-$. 
}

\purple{
We showed  in Section \ref{subsub:ratiobound} that 
\begin{equation} \label{eqn:alphabv-} 
\alpha(g_i) \ra \infty \hbox{ and } 
\frac{\alpha(g_i )}{ \llrrV{\vb_{g_i, -} }}  \ra \infty \hbox{ as } l_{\Ss_+}(g_i) \ra \infty. 
\end{equation} 
Hence, 
 for $i \ra \infty$
\begin{equation}\label{eqn:ratio} 
\{\llrrparen{ \llrrV{\vb_{g_{i},+}}_{E}: \llrrV{\vb_{g_{i},0}}_{E}: \llrrV{\vb_{g_{i},-}}_{E}}\}\,\,   \ra  
\llrrparen{ \pm 1:0: 0 }\,\,  \hbox{ or } 
\llrrparen{ \ast_{1}: \ast_{2}: 0 }\,\,
\end{equation} 
where $\ast_{2} \geq 0$ since $\alpha(g_{i}) > 0$ {by Criterion \ref{cr:positive}}. 
\eqref{eqn:vecx+} implies
\begin{equation}
\{\hyperlink{term-bdd}{\bdd}\left( \llrrparen{ \vb_{g_{i}} },  \clo(\zeta_{a_{g_{i}}})\right)\}\ra 0 \hbox{ as } l_{\Ss_+}(g_{i})\ra \infty   
\end{equation} 
by the above conclusion. 
}

%
\end{proof}

\subsection{Accumulation points of $\Gamma$-orbits} \label{sub:accumulate} 



Recall $N_{\bdd, \eps}(\cdot)$ from Section \ref{sub:Hlimit}. 
We again use $\bdd_H$ in $\SI^3$. 
We say that $\gamma_i(K)$ for a compact set  $K$ and a sequence 
$\gamma_i$ \hypertarget{term-acc}{{\em accumulates only to}} a set $A$ if 
$\gamma_i(z_i), z_i\in K$ has accumulation points only in $A$. 
Of course, the same definition extends to the case when $K$ is a point. 
It is easy to see that 
this condition is equivalent to the condition that 
\[
\hbox{ for every } \eps > 0, \hbox{ there is } I \hbox{ so that }  
\gamma_i(K)\subset N_{\bdd, \eps}(A) \hbox{ for } i > I.
\]  
(For the point case, we need to change the symbol $\subset$ to the symbol $\in$.)


\begin{corollary} \label{cor:conv1}
	Assume Criterion \ref{cr:positive} and $\mathcal{L}(\Gamma) \subset \SO(2,1)^{o}$. 
%
	Let $K \subset \Lspace$ be a compact subset. 
	Let $y \in \Ss_+$, and 
	let $\gamma_i \in \Gamma$ be a sequence such that 
	$\{\gamma_i(y)\}\ra y_\infty$ for $y_\infty \in \partial \Ss_+$. 
	Then 
	for every $\eps > 0$, there exists $I_0$ such that 
	\begin{equation}\label{eqn:zetainfty}  
	\gamma_i(K) \subset N_{\bdd, \eps}(\clo(\zeta_{y_\infty})) \hbox{ for } i > I_0.
	\end{equation} 
	Equivalently, any sequence $\{\gamma_i(z_i)| z_i \in K\}$ accumulates 
	only to $\clo(\zeta_{y_\infty})$.
	\end{corollary} 
\begin{proof} 
%
\newch{It is enough to prove for subsequences of every subsequence 
that the conclusion holds}.
To obtain all limit points of $\{\gamma_{i}(K)\}$, we will use the fact that $\Gamma$ acts as a convergence group on $\partial \Ss_+$ from Section \ref{sub:mar}.
Up to choosing subsequences, we assume that $\{\gamma_{i}\}$ is a \hyperlink{term-cfg}{convergence sequence} with 
the attracting point $a$ and the repelling point $r$. 

%

We first consider the case $a \ne r$. 
Then $\gamma_{i}$ acts on a geodesic $l_{i}$ in $\Ss_{+}$ passing a compact set $\mathscr{K}$ for sufficiently large $i$. 
Let $x_i \in \mathscr{K} \cap l_i$ where 
$l_i$ is given the direction $\vu_i$ so that $\gamma_i$ acts in the forward direction.   
Using the notation of the proof of Theorem \ref{thm:unifest}, 
we have \[a(\gamma_i) = \llrrparen{\vv_{+, (x_i, \vu_i)}}, 
r(\gamma_i) = \llrrparen{\vv_{-, (x_i, \vu_i)}}.\] 
{We only need to consider subsequences $\{\gamma_i\}, \gamma_i \in \Gamma$, where the sequence  
$a(\gamma_i) \in \partial \Ss_+$ of attracting fixed points and the sequence $r(\gamma_i) \in\partial \Ss_+$ of repelling fixed points are both 
convergent. 
Here, \[ \{a(\gamma_{i})\} \ra a \hbox{ and } \{r(\gamma_{i})\} \ra r \,  {\hbox{ in } \partial \Ss_+}.\]
}

Since $\{\gamma_{i}(y)\} \ra y_{\infty} \in \partial \Ss_{+}$, 
$\gamma_{i}$ is unbounded in $\Gamma$ and hence \oldch{$\{l_{\Ss_+}(\gamma_{i})\} \ra \infty$}, and 
$\{\lambda(\gamma_{i})\} \ra \infty$ for the largest eigenvalue $\lambda(\gamma_{i})$ of $\gamma_{i}$.

The convergences are uniform on the compact set $K\subset \Lspace$. 
To explain, we recall \eqref{eqn:vecx+}. We introduce 
the $(x^{(i)}, y^{(i)}, z^{(i)})$-coordinate system where 
\[\vv_{+, (x_i, \vu_i)}, \bnu_{\gamma_i}, \hbox{ and } \vv_{-,(x_i, \vu_i)}\] 
\newch{form} a coordinate basis 
parallel to the $x^{(i)}$-, $y^{(i)}$-, and $z^{(i)}$-axes respectively. 
We let $x, y, z$ denote the coordinate functions, where
$(\vv_+, \nu, \vv_-)$ forms a coordinate basis. 

$K$ is in a region $R_i$ given by 
\[[-C_1, C_1] \times [-C_2, C_2] \times [-C_3, C_3]\]
in the  $(x^{(i)}, y^{(i)}, z^{(i)})$-coordinate system. 
We may assume $C_1, C_2, C_3$ are independent of $i$ since 
the coordinate functions $x^{(i)}$, $y^{(i)}$, and $z^{(i)}$ 
converge respectively to coordinate functions $x, y$, and $z$ 
on $\Lspace$. 
We write $\gamma_{i}(x) = A_{\gamma_{i}} x + \vb_{\gamma_{i}}$. 
Since the sequence of largest 
eigenvalues of the linear parts of $\gamma_i$ goes to $+\infty$, 
$A_{\gamma_i}(R_i)$ is given under the $(x^{(i)}, y^{(i)}, z^{(i)})$-coordinate
system
by
\[[-D_i, D_i] \times [-E_i, E_i] \times [-F_i, F_i]\] 
where $\{D_i \}\ra \infty$, $E_i =C_2$, $\{F_i\} \ra 0$ for $F_i> 0$. 
By Definition \ref{defn:bg-}, $\gamma_i(R_i)$ is in
\begin{multline} \label{eqn:gammaC} 
S_{i} := [-\infty, \infty] \times [-E_i+ \alpha(\gamma_i), E_i+\alpha(\gamma_i)] \times  \\
\left[-F_i - \frac{\llrrV{\vb_{\gamma_{i},-}}_E}{\llrrV{\vv_{-, (x_i, \vu_i)}}_E}, F_i+ \frac{\llrrV{\vb_{\gamma_{i},-}}_E}{\llrrV{\vv_{-, (x_i, \vu_i)}}_E}\right]
\end{multline}
in the  $(x^{(i)}, y^{(i)}, z^{(i)})$-coordinate system. 

Recall coordinate change maps $\mathscr{C}_i$ and $\mathscr{C}_\infty$ 
near \eqref{eqn:vecx+}.  
For sufficiently large $i$, we deduce that $\gamma_i(R_i)$ is
 a subset of $N_{\bdd, \eps}({\clo(\zeta_a)})$ as follows:
There is a sequence of coordinate change maps 
$h_i:\Lspace \ra \Lspace$ with a uniformly bounded 
matrix $\mathscr{C}_{\infty}\mathscr{C}_{i}^{-1}$ such that 
\[x\circ h_i = x^{(i)}, y\circ h_i = y^{(i)}, z\circ h_i = z^{(i)}.\] 
Since \[x^{(i)} \ra x, y^{(i)} \ra y, z^{(i)} \ra z\] by \eqref{eqn:vecx+}, 
we obtain $h_i \ra \Idd_{\SI^3}$ as $i \ra \infty$.
What $h_i$ does is to send a box in the $(x^{(i)}, y^{(i)}, z^{(i)})$-coordinate system to 
the box of the same coordinates in the $(x, y, z)$-coordinate system. 

\purple{ 
Since $\alpha(\gamma_i) \ra \infty$ and 
$\alpha(\gamma_i)/\llrrV{\vb_{\gamma_{i},-}}_E \ra \infty$ 
by Theorem \ref{thm:unifest},} 
\eqref{eqn:gammaC} implies that 
$h_i(S_i) \ra \clo(\zeta_a)$ geometrically. 
Since $h_i \ra \Idd_{\SI^3}$, we deduce that 
$S_i \ra  \clo(\zeta_a)$ by Corollary \ref{cor:geoc}. 
Hence, for every $\epsilon> 0$, 
we have \[\gamma_{i}(R_{i}) \subset S_i \subset N_{\bdd, \eps}(\clo(\zeta_a))\] for sufficiently large $i$, \oldch{ and \eqref{eqn:zetainfty}}  holds. 

Finally, suppose that $a = r$. 
We choose $\gamma$ so that 
\[a(\lim_{i\ra \infty}\gamma\gamma_i)= 
\gamma(a) 
\ne r = \lim_{i \ra \infty }r(\gamma_{i})\]
and use the sequence $\gamma\gamma_{i}$ as our convergence sequence. 
Then $\{\gamma \gamma_i(K)\}$ accumulates only on $\clo(\zeta_{\gamma(a)}) = \gamma(\clo(\zeta_a))$. 
Therefore, $\{\gamma_i(K)\}$ accumulates only to $\clo(\zeta_a)$. 

\end{proof}

%
	%

We end with the \newch{following}: 
\begin{proof}[Converse part of Theorem \ref{thm:equivalence}] 
Suppose that $\Gamma \subset \SO(2,1)^o$. 
To show the proper action of $\Gamma$, we  
show that for any sequence $\{g_i\}$ of infinite elements, 
$g_i(K) \cap K \newch{\neq} \emp$ for only finitely many elements. 
Suppose not. Then 
by taking a subsequence, we may assume that $\gamma_i(y) \ra y_\infty$ for 
$y_\infty\in \partial \Ss_+$.  
By Corollary \ref{cor:conv1}, we showed that this cannot happen. 

If $\Gamma$ is not in $\SO(2, 1)^o$, 
then we use the index $2$ subgroup $\Gamma' \subset \SO(2, 1)^o$ and it acts properly on 
$\Lspace$ and so does $\Gamma$. 
\end{proof}

\section{The topology of Margulis space-times with parabolics} \label{sec:MP}


We first give an outline of this long section. 
We discuss the classical theory of Scott and Tucker \cite{ST89} on open $3$-manifolds homotopy equivalent to compact ones. 
Next, we will construct \hyperlink{term-pbr}{parabolic regions} in $\tilde M$. 

%




In Section \ref{sub:fund}, we will find a fundamental region for $\Gamma$ in $\Lspace$ using the work of Epstein-Petronio \cite{EP94}.
{By Proposition \ref{prop:exhaust},} we obtain an exhausting sequence 
\[M_{(1)}\subset M_{(2)} \subset M_{(3)} \subset \cdots \subset \Lspace/\Gamma.\]
In Section \ref{subsub:bounded}, we discuss some boundedness properties of {the inverse image $\tilde M_{(J)}$ of $M_{(J)}$} for some $J$ 
meeting with disks and topological polytopes. 
First, we construct the candidate disks to bound a candidate fundamental domain.
The key step is Proposition \ref{prop:boundedness} 
that the universal cover $\tilde M_{(J)}$ of 
an element $M_{(J)}$ of the exhausting sequence meets the candidate disks and 
parabolic regions in bounded sets.
This implies 
Corollary \ref{cor:FmMi}  that $\tilde M_{(J)}$ meets a candidate
fundamental domain $\mathbf{F}$ in a compact submanifold
and hence $\mathbf{F}\setminus \tilde M_{(J)}$ is a compact finite-sided topological polytope. 
In Section \ref{subsub:chooseF}, we choose our candidate disks $\mathcal{D}_{j}$, $j=1, \dots, {\mathbf{g}}$, 
and the candidate fundamental domain $\mathbf{F}$. 
Then we divide $\tilde M$ into $\tilde M_{(J)}$ and $\tilde M\setminus \tilde M_{(J)}$.
We show $\mathbf{F}\setminus \tilde M_{(J)}$ for sufficiently large $J$ is the fundamental 
domain of $\tilde M\setminus \tilde M_{(J)}$ 
using Proposition \ref{prop:Poincare} (the Poincar\'{e} fundamental domain theorem).   
Candidate disks in $\tilde M$ are replaced by ones mapping to embedded disks in $M$ by 
replacing the parts in $\tilde M_{(J)}$ by {Theorem \ref{thm:Dehn}, i.e., Dehn's lemma}. 
We obtain the fundamental domain of $\tilde M_{(J)}$, 
proving the tameness of $M$, and the first part of Theorem \ref{thm:main}. 

In Section \ref{sub:parab}, we will show that for a choice of parabolic regions sufficiently far from $\tilde M_{(J)}$, their images under $\Gamma$ are mutually disjoint.
To show this, we use the tessellations by the images of a fundamental domain, and
we explain how they intersect with the parabolic regions. Then we can account for 
every image by its relationship with the images of the fundamental domain. 

In Section \ref{sub:compactification}, we will discuss the relative compactification of $\tilde M$. We will prove the final part of 
Theorem \ref{thm:main} and Corollary \ref{cor:main2}.  (See Marden \cite{Marden74} \cite{Marden07}, and \cite{Marden16} 
for many aspects of ideas in this section.)




\subsection{Handlebody exhaustion of the Margulis space-times} \label{sub:handle}

The ends of $\Ss_+/\mathcal{L}(\Gamma)$ are finitely many, and some of these are cusps.  
A \hypertarget{term-peri}{{\em peripheral}} element of $\Gamma$ is an element corresponding to 
a closed loop in the complete hyperbolic surface freely homotopic to 
one in an end neighborhood homeomorphic to an annulus. 
Let $\mathcal{I}'$ denote the collection of the maximal peripheral cyclic subgroups of $\Gamma$,
and let $\mathcal{I}$ denote the ones with hyperbolic holonomy. 
Each peripheral element of $\Gamma$ acts on 
a point of $\partial \Ss_+$ as a parabolic element 
or on a connected arc $a_i \subset \partial \Ss_+, i\in {\mathcal I}$ with the hyperbolic cyclic group
$\langle \vartheta_i \rangle$ acting on it. 
Here 
\[\Sigma_{+}:= (\Ss_+ \cup \bigcup_{i\in {\mathcal I}} a_i)/\Gamma\] is a finite-type surface with  finitely many punctures and 
boundary components covered by arcs of the form $a_i$. 

We define $A_i := \bigcup_{x \in a_i} \zeta_{x}$, $i\in {\mathcal I}$, 
an open domain
where $\zeta_{x}$ is the \hyperlink{term-gs}{accordant} great segment for $x$. 
We define 
\begin{equation} \label{eqn:tSigma} 
\tilde \Sigma := \Ss_+ \cup \Ss_- \cup \bigcup_{i \in {\mathcal I}} (A_i \cup a_i \cup \mathcal{A}(a_i)).
\end{equation} 
Then $\Gamma$ acts properly  on $\tilde \Sigma$, and 
$\Sigma:= \tilde \Sigma/\Gamma$ is a \hyperlink{term-rps}{real projective surface}. 
This follows by the same proof as Theorem 5.3 of \cite{CG17} without change. 
Again, $\Sigma$ has twice the number of punctures as $\Sigma_{+}$
and $\chi(\Sigma) = 2\chi(\Sigma_{+})$. 

We define $\tilde N:= \Lspace \cup \tilde \Sigma$. 
{\em We will show below that $\Gamma$ acts properly 
 on $\tilde N$ 
to give us a manifold quotient $\tilde N/\Gamma$}\,: 


Let $N$ be a manifold.
A sequence $N_i$ of submanifolds of $N$ is {\em exhausting} if 
$N_i \subset N_{i+1}$ for all $i$ 
and every compact subset of $N$ is a subset of $N_i$ for some $i$.
We obtain $N = \bigcup_{i=1}^\infty N_i$ necessarily. 

The following is essentially due to Scott and Tucker \cite{ST89},
{which we learned from some talks by Ohshika \cite{Ohshika}} in this form
(See also page 5 of Canary and Minsky \cite{CM96}).

\begin{proposition} \label{prop:exhaust}
Let $\Lspace/\Gamma$ be a Margulis space-time with parabolics. 
Then $\Lspace/\Gamma$ has a sequence of handlebodies 
\[M_{(1)} \subset M_{(2)} \subset \dots \subset M_{(i)} \subset M_{(i+1)} \subset \dots\]
so that $M_0 = \bigcup_{i=1}^\infty M_{(i)}$. They have the following properties\,{\rm :} 
\begin{itemize}
\item $\pi_{1}(M_{(1)}) \ra \pi_{1}(M)$ is an isomorphism. 
\item The inverse image $\tilde M_{(i)}$ of $M_{(i)}$ in $\tilde M$ is connected. 
\item $\pi_{1}(M_{(i)}) \ra \pi_{1}(M)$ is surjective. 
\item For each compact subset $K \subset \Lspace/\Gamma$, there exists an integer $I$ 
so that for $i > I$, $K \subset M_{(i)}$. 
\end{itemize} 
\end{proposition} 
\begin{proof}
	The existence of exhaustion is clear. 
	We choose $M_{(1)}$ by using the $1$-complex homotopy equivalent to $M$. 
	$\pi_1(M_{(i)}) \ra \pi_1(M)$ is surjective since 
	$\pi_1(M_{(1)}) \ra \pi_1(M)$ factors into this map and 
	$\pi_1(M_{(1)}) \ra \pi_1(M_{(i)})$. 
	Choose a base point $x_0$ of $M_{(1)}$. 
	Any closed loop in $M$ with a basepoint in $M_1$ 
	is homotopic to a closed loop in $M_{(i)}$.  
	Hence, any two points of the inverse image  $x_0$ in
	$\tilde M$ is connected by a path in $\tilde M_{(i)}$ by the homotopy path-lifting 
	theorem of Poincar\'e. 	Thus, $\tilde M_{(i)}$ is connected. 
	\end{proof}

\subsubsection{Parabolic solid-torus regions} \label{subsub:pararegion} 
Let $\Sf := \Ss_+/\mathcal{L}(\Gamma)$. It has finitely many ends. 
Some of these are cusp ends, and some are hyperbolic ends. 
$\Gamma$ has parabolics \[\mathbf{g}_1, \dots, \mathbf{g}_{m_0},\] 
each of which represents a generator of the fundamental group of  
a cusp neighborhood of $\Sf$.
We let each of \[\mathbf{g}_{m_0+1}, \dots, \mathbf{g}_{m_0+h_0}\] 
represent the generator of each of the fundamental groups 
of the hyperbolic end neighborhoods of $\Sf$. 
We choose the generators along the boundary orientation of $\Sf$. 

{Recall the notations from Section \ref{sub:thin}.}
We take components ${\mathscr{H}}_i \subset \Ss_+$, 
$i \in \mathcal{I}'\setminus \mathcal{I}$ in $\Ss_+$ of ${\mathscr{H}}$. 
A parabolic primitive element $\mathbf{g}_i$ conjugate to $\mathbf{g}_{j}$ for some $j$ acts on ${\mathscr{H}}_i$.  
We also note for every $g \in \Gamma$, 
\begin{itemize} 
\item either $g({{\mathscr{H}}}_i) = {{\mathscr{H}}}_i$ and $g = \mathbf{g}_i^n$ for $n \in \bZ$, or else 
\item $g({{\mathscr{H}}}_i) \cap {{\mathscr{H}}}_i = \emp$. 
\end{itemize} 
We define ${\mathscr{H}}_{i, -} =\mathcal{A}({\mathscr{H}}_{i}).$ 
There is a fixed point $p_i$ of $\mathbf{g}_i$ 
in $\Bd_{\Ss} {{\mathscr{H}}}_i \cap \partial \Ss_+$ for each $i\in \mathcal{I}' \setminus \mathcal{I}$. 
For each $i \in 1, \dots, m_{0}$, 
Theorem \ref{thm:Sr}  gives us a properly embedded ruled surface
$S_{i}:= S_{f_i, r_{0}} \subset \Lspace$ for some fixed function 
$f_i: (0, 1) \ra \bR$ {and}     
\[\clo(S_{i})\setminus S_i = \clo(\zeta_{a(\mathbf{g}_{i})})\cup 
\partial_h {{\mathscr{H}}}_{i} \cup \partial_h {{\mathscr{H}}}_{i, -}  \]
where $a(\mathbf{g}_{i})$ 
is the parabolic fixed point of $\mathbf{g}_{i}$ in $\partial \Ss_{+}$.  
$S_{i}$ is called a {\em parabolic ruled surface}. 
The component of $\Lspace \setminus S_{i}$ whose closure 
contains ${{\mathscr{H}}}_{i}^o$ is called 
a \hyperlink{term-pbr}{{{\em parabolic region}}}, denoted by ${\mathcal{P}}_i$,
{which is homeomorphic to a $3$-cell by Theorem \ref{thm:ruled}.} 
These are distinct from \hyperlink{term-pc}{parabolic cylinders}. 
Here, $f_i$ is fixed for each conjugacy class of parabolic elements. 
{(See Section \ref{sub:parabolic} for detail.)}

For each $i \in \mathcal{I}'\setminus\mathcal{I}$, we define 
$S_{i} = \gamma(S_{j})$ and ${\mathcal{P}}_{i} = \gamma({\mathcal{P}}_{j}) $ for any $j, j=1, \dots, m_0$, and $\gamma$ 
so that $\gamma({\mathscr{H}}_{j}) = {\mathscr{H}}_{i}$. This surface $S_i$ 
is well-defined since 
any element acting on ${\mathscr{H}}_{j}$ acts on $S_{i}$ and ${\mathcal{P}}_{i}$. 
We have the \hypertarget{term-eqc}{{\em {$\Gamma$-}equivariant choice}} of parabolic ruled surfaces and parabolic regions.

Theorem \ref{thm:Sr} gives us a foliation $\mathcal{S}_{f_i, r_i}$ {with leaves that are parabolic ruled surfaces}
and a transversal foliation $\mathcal{D}_{f_i, r_i}$ 
for each ${\mathcal{P}}_i$ for each $i=1, \dots, m_0$. For other
${\mathcal{P}}_i$, we use the induced ones from ${\mathcal{P}}_j$ such that 
${{\mathcal{P}}_i=}\gamma({\mathcal{P}}_j)$ for $j=1, \dots, m_0$. 

Finally, we will make these $S_{i}$ and ${\mathcal{P}}_{i}$ \hyperlink{term-far}{sufficiently far} whenever it is necessary to do so {in} this paper. 
(See Definition \ref{defn:para}.) We may do so without acknowledging. 







\subsection{Finding the fundamental domain} \label{sub:fund}

A {\em topological polytope} in $\Lspace$ is a $3$-manifold closed as a subset 
of $\Lspace$ and whose closure in
$\clo(\Lspace)$ is a compact manifold with boundary 
that is a union of finitely many smoothly and properly embedded compact submanifold.
In \cite{CG17}, we defined a \hypertarget{term-crc}{{\em crooked circle}} to 
be a simple closed curve in $\Ss$ of {the} form 
\[ d \cup \mathcal{A}(d) \cup \bigcup_{x\in \partial d} \clo(\zeta_{x})\]
for a complete geodesic $d$ in $\Ss_{+}$ with boundary in a parabolic fixed point 
or in a boundary component of $\tilde \Sigma_{+}$. 
We may refer to them as being
\newch{{\em positively oriented}} since the definition depends on the orientations of $\Lspace$. 

{Recall parabolic regions from Section \ref{subsub:pararegion}.} 
\begin{definition}\label{defn:cbd} 
A \hypertarget{term-cbd}{{\em crooked-circle disk}} $D$ is a properly embedded open disk in $\Lspace$ whose 
boundary $\partial D$ is a crooked circle
satisfying the condition: 
If $x$ is a parabolic fixed point in $\partial d$ and
${\mathcal{P}}_i$ is a sufficiently far away 
parabolic region for $x$, 
${\mathcal{P}}_{i} \cap D$ is a ruled surface in a  leaf of the transversal foliation 
$\mathcal{D}_{f_i, r_i}$
obtained as in Theorem \ref{thm:Sr}. 
\end{definition}

A disk $D$ in $\Lspace$ is \hypertarget{term-sep}{{\em separating}} if it is properly embedded 
and $\Lspace\setminus D$ has two components. 
{Crooked-circle} disks and parabolic ruled surfaces are separating.


\subsubsection{The simple case of the properly acting parabolic cyclic group}
Theorem \ref{thm:FP} is a much easier version of that of Theorem \ref{thm:main}
presented {analogously}. 

\begin{theorem}[Small tameness]\label{thm:FP}  
	{Assume as in Theorem \ref{thm:main}.}
Suppose that $D$ is a crooked-circle disk in $\Lspace$ with a point $p \in \partial D$ fixed by a parabolic element 
$\gamma$ with a positive Charette-Drumm invariant. Then 
we can modify $D$ inside a compact set in $\Lspace$ so that $D \cap \gamma(D) =\emp$. 
If we denote $F_{P}$ to be the connected domain in $\Lspace$ bounded by $D$ and $\gamma(D)$, 
then $F_{P}$ is a fundamental domain of $\langle \gamma \rangle$ in $\Lspace$. 
Furthermore, $\Lspace/\langle \gamma \rangle$ is homeomorphic to a solid torus. 
\end{theorem}
\begin{proof} 
We take an arbitrary compact set $K$ in $\Lspace$. 
Then there exists a sufficiently far away parabolic region $R'_{p}$ where $\gamma$ acts
so that $K \cap R'_{p} = \emp$. We have 
\begin{equation}\label{eqn:Rp} 
\bigcup_{n\in \bZ} \gamma^{n}(K) \cap R'_{p} = \emp
\end{equation}
since $R'_p$ is $\gamma$-invariant. 

By taking sufficiently large $K$, we may assume that $\tilde T := \bigcup_{n\in \bZ} \gamma^{n}(K)$ is connected. 
By the proper discontinuity of the action of $\langle \gamma \rangle$, 
$K$ meets only finitely many $\gamma^{n}(K)$. 
Choose $K$ as a generic $3$-ball so that $T:= \tilde T/\langle \gamma \rangle$ is a compact manifold. 

We take a sequence of generic compact $3$-balls $K_{i}$ exhausting $\Lspace$. 
Then the corresponding $T_{i}$, $i=1, 2, \dots$, form 
an exhausting sequence of compact $3$-manifolds of $\Lspace/\langle \gamma \rangle$. 
We denote $\tilde T_{i} := \bigcup_{n\in \bZ} \gamma^{n}(K_i) $. 

(I) We first show that $\tilde T_i$ meet with $D$ in a compact set 
and find a candidate fundamental domain $F$ bounded by two disks in a compact set. 

By \purple{Theorem \ref{thm:equivalence} and }
Corollary \ref{cor:conv1}, $\gamma^{n}(K_i)$ as $n \ra \pm \infty$ 
can have accumulation points only in $\clo(\zeta_{p})$. 
$D \cap R'_{p} \cap \tilde T_{i }=\emp$ by \eqref{eqn:Rp}
for sufficiently far choice of $R'_p$.
Since $\{\gamma^{n}(K_i)|n\in \bZ\}$ is a locally finite collection of sets in $\Lspace$
\hyperlink{term-acc}{accumulating only} to $\clo(\zeta_p)$  by Corollary \ref{cor:conv1}, 
and $D\setminus R^{\prime}_{p}$ is $\bdd$-bounded 
away from $\clo(\zeta_{p})$, it follows that $(D\setminus R^{\prime}_{p}) \cap \tilde T_{i}$ is compact. 
Hence, $D \cap \tilde T_{i}$ is compact for each $i$. 
Similarly, so is $\gamma(D) \cap \tilde T_{i}$. 

By construction in Definition \ref{defn:cbd} and Theorem \ref{thm:Sr}, 
\[ D \cap \gamma(D) \cap R'_{p} = \emp
\hbox{ and } (\partial D \cap \gamma(\partial D))\setminus \clo(R'_{p}) = \emp.\] 
Since $D \cap R'_{p}$ is a ruled disk so that $\gamma(D \cap R'_{p}) \cap D \cap R'_{p} = \emp$, 
we can find a thin tubular neighborhood $T''$  in $\clo(D\setminus R'_p)$ 
of $\partial \clo(D\setminus R'_p)$ so that $T'' \cap \gamma(T'') = \emp$. 
We add the disk $D\cap R'_{p}$ to $T''$ to obtain $T'$. 
Hence, $\gamma(T') \cap T' =\emp$. 

We modify the disk $\gamma(D)\setminus \gamma(T')$ to another disk $D_{1}$
to be disjoint from $D$.
Then $D$ and $D_{1}$ bound a topological polytope  $F$ closed in $\Lspace$. 

Choose a sufficiently large $i$ so that 
\[D\setminus T', D_1 \setminus \gamma(T'), D \cap D_1  \subset \tilde T_{i},\]
and we choose sufficiently far $R'_p$ so that $R'_p \cap \gamma^n(K_i) = \emp$
for every $n\in \bZ$. 
We obtain that 
\[T'\setminus \tilde T_{i} = D \setminus \tilde T_i \hbox{ and } \gamma(T')\setminus \tilde T_{i}= \gamma(D)\setminus \tilde T_i\]
is a matching set under $\{\gamma, \gamma^{-1}\}$. 
They are also in $\Bd F \cap \Lspace$. 

Also, $(F \setminus R'_{p}) \cap \tilde T_{i}$ is again compact in 
$\Lspace$ 
since $F\setminus R'_{p}$ is $\bdd$-bounded away from $\clo(\zeta_{p})$ and 
$\gamma^n(K_i)$ accumulates only to $\clo(\zeta_p)$ by Corollary \ref{cor:conv1}. 
Since $R'_{p}\cap \tilde T_{i} = \emp$,  $F \cap \tilde T_{i}$ is compact, 
$F\setminus \tilde T_{i}^{o}$ is a topological polytope. 

(II) We find a fundamental domain that is a topological polytope. 

Since $\{T_{j}\setminus T_{i}^{o}|j > i\}$ is an exhausting sequence of $\Lspace/\langle \gamma \rangle\setminus T_{i}^{o}$, 
Proposition \ref{prop:Poincare} implies that 
\[T' \cap ( F\setminus \tilde T_{i}^{o} ) \hbox{ and } \gamma(T') \cap ( F\setminus \tilde T_{i}^{o} )\] 
bound a topological polytope $F\setminus \tilde T_{i}^{o}$ that is 
a fundamental domain of $\Lspace\setminus \tilde T_{i}^{o}$ under $\langle \gamma\rangle$. 

We choose {a generic set denoted by} $T_{i}$ so that $D \cap \tilde T_{i}$ is a union of simple closed curves.
The image in $\Lspace/\langle \gamma \rangle $ of the bounded component of $D\setminus \tilde T_{i}^{o}$ is 
embedded since $F\setminus \tilde T_{i}^{o}$ is a fundamental domain of $\Lspace\setminus \tilde T_{i}^{o}$
under $\langle \gamma \rangle $. We take mutually disjoint tubular neighborhoods of the images of these bounded components
in $\Lspace/\langle \gamma \rangle\setminus T_{i}^{o}$ whose lifts in $\Lspace\setminus \tilde T_{i}^{o}$ are disjoint. 
We add these tubular neighborhoods 
to $\tilde T_{i}$, and now each component of $D \cap \tilde T_{i}$ is a disk. 

{
\begin{theorem}[Dehn's lemma. See Hempel \cite{Hempel04}]\label{thm:Dehn} 
	Let $M'$ be a $3$-manifold  $M$ and $f: B \ra M$ be  a map from a disk $B$ such that for some neighborhood of $A$ of 
	the boundary $\partial B$ in $B$. 
	If $f|A$ is an embedding and $f^{-1}(f(A)) = A$, then $f|\partial B$ extends to an embedding $g:B \ra M$. 
\end{theorem}
}
 
{By Theorem \ref{thm:Dehn},} we replace the images in $T_{i}$ of disk components of $D \cap \tilde T_{i}$
by embedded disks in $T_{i}$. We lift these disks to $\tilde T_{i}$ and attach the adjacent ones to $D\setminus\tilde T_{i}^{o}$. 
We obtain a disk $D''$, and it is clear that $D'' \cap \gamma(D'')=\emp$. 

We rename $D''$ by $D$. 
Let $F_{P}$ denote the region in $\Lspace$ bounded by $D$ and $\gamma(D)$. 
Since $\clo(F_{P})\setminus \clo(R'_{p})$ is bounded away from $\clo(\zeta_{p})$ under $\bdd$, 
and $\{\gamma^{n}(K)|n\in \bZ\}$ is a locally finite collection of sets in $\Lspace$
\hyperlink{term-acc}{accumulating only} to $\clo(\zeta_p)$  by \purple{Theorem \ref{thm:equivalence} and }  
Corollary \ref{cor:conv1}, we obtain that 
$(F_{P}\setminus R'_{p}) \cap \tilde T_{i}$ is a compact set. 
Since $\tilde T_{i} \cap R'_{p} = \emp$, $F_{P} \cap \tilde T_{i}$ is {also compact}.

Also, $F_{P}\cap \tilde T_{i}$ is compact for each $i$. 
By Proposition \ref{prop:Poincare},
$F_{P}$ is a fundamental domain in $\Lspace$ of $\langle \gamma \rangle$. 
The existence of the fundamental domain tells us that $\Lspace/\langle \gamma \rangle $ is tame and
hence is homeomorphic to a solid torus. 
\end{proof} 

\begin{remark}\label{rem:compact} 
Using the closure of the fundamental domain $F_P$ and identifying 
$D$ and $\gamma(D)$, we deduce that 
\[\left(\Lspace \cup \Ss_{+}\cup \Ss_{-} \cup \bigcup_{x\in \partial \Ss_{+}\setminus\{p\}} \clo(\zeta_{x})\right)\Bigg/\langle \gamma \rangle \]
is homeomorphic to $A \times [0, 1)$ for a compact annulus 
$A = \bigcup_{x\in \partial \Ss_{+}\setminus\{p\}} \clo(\zeta_{x})/\langle \gamma \rangle$, forming a relative compactification. 

As an alternative proof of Theorem  \ref{thm:FP}, 
we may use a $\gamma$-invariant foliation of $\Lspace$ by crooked planes from
the results of Charette-Kim \cite{CK14}
to prove the relative compactification. A fairly simple computation 
shows that there exists such a foliation in $\Lspace/\langle \gamma \rangle$. 
\end{remark} 

\subsubsection{The boundedness of $\tilde M_{(J)}\cap F$ for  some polytope $F$}\label{subsub:bounded} 

{  We choose an exhausting sequence $M_{(J)}, J =1, 2, 3, \dots,$  by Proposition \ref{prop:exhaust}. 
We aim to prove Corollary \ref{cor:FmMi} showing that the $\tilde M_{(J)}$ meets a ``candidate'' fundamental domain in a bounded set. } 

\begin{lemma}\label{lem:finiteR}
	Let $R$ be a conical region in $\Ss_{+}$ {that is} a fundamental domain of a parabolic element $\gamma$ with $p$ as the fixed point in $\partial \Ss_{+}$, 
	and let $F_{P}$ be a fundamental domain in $\Lspace$ of $\gamma$ 
	{bounded by two embedded disjoint crooked-circle disks $D_{1}$ and $\gamma(D_{1})$ in $\Lspace$ where
	\[\clo(R)\cap \Ss_{+} = \clo(F_{P}) \cap \Ss_{+} \hbox{ and } \clo(D_{1}) \cap \gamma(\clo(D_{1})) = \clo(\zeta_{p}).\]}    
	Let $L$ be a fundamental domain of $\tilde M_{(J)}$. 
	{Suppose that
	\begin{itemize} 
	\item	 the sequence $\{\eta_j\}$, $\eta_j \in \Gamma$, 
		takes infinitely many values, and 
	\item $\{\eta_{j}(y)\}$, $y\in \Ss_{+}$, accumulates only to 
	\[\clo(R) \cap \partial \Ss_{+}\setminus \{p\}.\] 
	\end{itemize}} 
	Then 
	\[\bigcup_{j=1}^{\infty} \eta_{j}(L) 
	\subset \bigcup_{i=-m_{0}}^{m_{0}}\gamma^{i}(F_{P}) \hbox{ for some finite } m_{0}.\]
\end{lemma} 
\begin{proof} 
	%
	Since $\clo(F_{P}) \cap \Ss$ is bounded by two crooked circles  $\clo(D_{1}) \cap \Ss$ and  $\clo(D_{2}) \cap \Ss$, we obtain 
	\[ (\clo(F_P)\setminus \clo(\zeta_p))\cap \Ss_{0}= \bigcup_{z \in \clo(R) \cap \partial \Ss_{+}\setminus \{p\} } \clo(\zeta_{z}). \] 
	Since $\eta_j(y)$ accumulates only to $\clo(R) \cap \partial \Ss_{+}\setminus \{p\}$, it follows that 
	$\eta_j(L)$ \hyperlink{term-acc}{accumulates only} to
	\begin{equation}\label{eqn:accum} 
	(\clo(F_P)\setminus \clo(\zeta_p))\cap \Ss_{0}= \bigcup_{z \in \clo(R) \cap \partial \Ss_{+}\setminus \{p\} } \clo(\zeta_{z})  
	\end{equation}
	by  \purple{Theorem \ref{thm:equivalence} and} Corollary \ref{cor:conv1}.
	The relative boundary $\Bd_{\mathcal{H}} \clo(F_P)$ in 
	the $3$-hemisphere $\mathcal{H}$
	is a union of two disks $\clo(D_{1})$ and $\gamma(\clo(D_{1}))$ with boundary in $\Ss$.
	
	{Since $\Bd_{\mathcal{H}} F_P$ has two components $\clo(D_1)$ and $\gamma(\clo(D_1)$ which coincide with a component of 
	$\Bd_{\mathcal{H}}\gamma^{-1}(F_P)$ and one of $\Bd_{\mathcal{H}} \gamma(F_P)$ respectively, 
	 \[F'':= \clo(F_P) \cup \gamma(\clo(F_P)) \cup \gamma^{-1}(\clo(F_P)) \] has the boundary set
	\[\Bd_{\mathcal{H}} F''= \gamma^{2}(\clo(D_{1}))\cup \gamma^{-1}(\clo(D_{1})) \subset \mathcal{H},\]} 
	and it follows that $F''-  \clo(\zeta_p)$ contains a neighborhood 
	of $(\clo(F_P)\setminus \clo(\zeta_p)) \cap \Ss_0$ in $\mathcal{H}$.
	Hence, we obtain by \eqref{eqn:accum}  that except for finitely many $\eta_j(L)$, 
	\[\eta_j(L) \subset (\clo(F_P) \cup \gamma(\clo(F_P)) \cup \gamma^{-1}(\clo(F_P)))\setminus 
	\clo(\zeta_p).\]
	Since $F_{P}$ is a fundamental domain of $\gamma$,
	we obtain $\Lspace \subset \bigcup_{i\in \bZ} \gamma^{i}(F_{P})$. 
	By the paragraph above, we obtain
	\[\bigcup_{j=1}^{\infty} \eta_{j}(L) 
	\subset \bigcup_{i=-m_{0}}^{m_{0}}\gamma^{i}(F_{P}) \hbox{ for some finite } m_{0}.\]

	\end{proof}

The following is a crucial step in this paper: 

\begin{proposition}[Boundedness of {$\tilde M_{(J)}$} in disks] \label{prop:boundedness} 
	Let $J$ be an arbitrary positive integer. 
	For any crooked-circle disk $D$, 
	$D  \cap \tilde M_{(J)}$ is compact, i.e., bounded, and has only finitely many components. 
\end{proposition} 
\begin{proof} 
	Suppose not. Then we can find a compact fundamental domain $L$ of $\tilde M_{(J)}$ and an unbounded sequence $g_{j}\in \Gamma$, 
	$g_{j}(L) \cap D \ne \emp $ for infinitely many $j$. 
	Again, we may assume without loss of generality that $g_{j}$ is a \hyperlink{term-cfg}{convergence sequence} acting on 
	$\partial \Ss_{+}$ with $a$ as an attractor and $r$ as a repeller.
	(See Section \ref{sub:mar}.)
	Hence, we can find a sequence $x_{j} \in L$ with $g_{j}(x_{j}) \in D$,
	and $\{g_{j}(x_{j})\}$ accumulates to a point  $x$ of 
	\[\Ss \cap \partial D.\]
	
	If $x \in \Ss_{+}\cup \Ss_{-}$, then 
\purple{Theorem \ref{thm:equivalence} and Corollary \ref{cor:conv1} contradict this}. 
	In fact, we have $x \in \clo(\zeta_y)$ for some $y \in \Lambda_{\Gamma, \Ss_{+}}$. 
	
	
	If $\clo(D)$ is disjoint from $\Lambda_{\Gamma, \Ss_{+}}$, 
	then $D \cap \tilde M_{(J)}$ is compact by the above paragraph. We are finished in this case. 
	
	Now assume  
	$D \cap \clo(\Ss_{+})\cap \Lambda_{\Gamma, \Ss_{+}}$ is a finite set of parabolic fixed points or is empty. 
	Suppose that there exists a sequence 
	\begin{equation} \label{eqn:zetap}
	\{ g_j(x_j)\in D, x_j \in L\} \ra x \in \clo(\zeta_{p})
	\end{equation}
	for a fixed point $p$, $p\in \partial D$, 
	of a parabolic element $\gamma \in \Gamma$
	(see Definition \ref{defn:semicircle}). 
	Let $y$ be a point of $\Ss_+$.
	If $\{g_{j}(y)\}$ converges to $q \ne p$, then 
	\[x \in \clo(\zeta_q) \ne \clo(\zeta_p) 
	\hbox{ with } \clo(\zeta_q) \cap \clo(\zeta_p) = \emp\]
	by \purple{Theorem \ref{thm:equivalence} and} Corollary \ref{cor:conv1}.
	Since this is a contradiction, we obtain $g_{j}(y) \ra p$.

	We obtain $g_{j}(y) \ra p$ for a point $y \in \Ss_{+}$. 
	We can choose a sequence $\gamma^{k(j)} \in \Gamma, k(j)\in \bZ$, so that 
	$\gamma^{k(j)} g_{j}(y)$ is in a conical region $R$ closed in $\Ss_{+}$ 
	bounded by two complete geodesics $l, \gamma(l)$ with the common endpoint $p$ in $\partial \Ss_{+}$. 
	
	Since $p$ is not a conical limit point by Tukia \cite{Tukia98}, 
	$\gamma^{k(j)} g_{j}(y)$ is bounded away from $p$ in $R$. 
	Therefore, $\eta_j:= \gamma^{k(j)}g_{j}$ is a sequence
	so that $\eta_j(y) = \gamma^{k(j)}g_j(y)$ has accumulation points only in  
	$(\clo(R) \cap \partial \Ss_+)\setminus\{p\}$. 


	Here, $k(j)$ is an unbounded sequence since  
	$\gamma^{k(j)} g_j(y)$ still converges to $p$ otherwise. 
	By choosing a subsequence and the choice of $\gamma$, we may assume without loss of generality that $k(j) \ra \infty$. 
	
	
	
	
	We now modify the disk $D$ in a compact set in $\Lspace$ by Theorem \ref{thm:FP}. 
	Hence,  the new  disk $D$ does not violate the existence of  a sequence as in \eqref{eqn:zetap}.

	By Theorem \ref{thm:FP}, we find a fundamental domain $F_{P}$ closed in $\Lspace$ of $\langle \gamma \rangle$
	bounded by a crooked-circle disk $D$ and its image $\gamma(D)$ disjoint from $D$. 
	Here, $\clo(F_{P}) \cap \Ss_{+} = R$. 
	
	
	
		Suppose that $\eta_j$ takes infinitely many values. 
	Since $\partial D$ is a crooked circle, $\clo(D)$ and 
	$\gamma(\clo(D))$ meet only in $\clo(\zeta_{p})$, 
	Lemma \ref{lem:finiteR} shows that
	\[\bigcup_{j=1}^{\infty} \eta_{j}(L) 
	\subset \bigcup_{i=-m_{0}}^{m_{0}}\gamma^{i}(F_{P}) \hbox{ for some finite } m_{0}.\]
	When $\eta_j$ takes only finitely many values, 
	this is also obvious. 

	Since $k(j) \ra \infty$, the finiteness of $m_1$ and 
	the nature of the parabolic action of $\gamma^{-k(j)}$ show
	\[g_{j}(x_{j}) = \gamma^{-k(j)}\eta_{j}(x_{j}), x_{j} \in L,\] 
	cannot lie on  
	the fixed disk $D$ containing $\zeta_p$.

\end{proof}

\begin{proposition} \label{prop:parabound} 
	Let $\eta\in \Gamma$ be a parabolic element acting on 
	a parabolic region $R_\eta$. 
	Let $p_\eta$ denote the parabolic fixed point of $\eta$ in 
	$\partial \Ss_+$. 
Let $\hat R_{\eta}$ denote the closure in 
$R_{\eta}$  of a component of 
$R_{\eta} - D_1 - D_2$ for two crooked-circle disks $D_{1}$ and $D_{2}$ 
whose closures contain $\clo(\zeta_{p_{\eta}})$.
Assume that $D_{i}\cap R_{\eta}$$, i=1, 2$, is a ruled disk of the form of Theorem \ref{thm:Sr}. 
Suppose that $D_{1} \cap R_{\eta}$ and $\eta^{\delta}(D_{1})\cap R_{\eta}$ for $\delta = 1$ or $-1$ 
bound a region in $R_{\eta}$ 
containing $\hat R_{\eta}$. 
Then $\hat R_{\eta}\cap \tilde M_{(J)}$ is compact for each {$J$}. 
Furthermore, 
we may assume that 
\[\tilde M_{(J)} \cap R_{\eta} =\emp \hbox{ for } j=1, \dots, m_0,\] 
by choosing $R_{\eta}$ sufficiently far away. 
{\em (}See Definition \ref{defn:para}{\em .)} 
\end{proposition} 
\begin{proof} 
Suppose that $\hat R_{\eta}\cap \tilde M_{(J)}$ is not compact. 
Then again, we can find a compact fundamental domain $L$ of $M_{(J)}$ so that 
$g_{k}(L)$ meets $\hat R_{\eta}$ for infinitely many $k$. 
Then $\{g_{k}(L)\}$ has limit points in 
$\clo(\zeta_x)$ for $x \in \Lambda_{\Gamma, \Ss_{+}}$
by \purple{Theorem \ref{thm:equivalence} and} Corollary \ref{cor:conv1}. 
Since we have a sequence
\begin{gather} 
\{x_k\}, x_k \in \clo(\hat R_{\eta})\cap g_{k}(L) \hbox{, and } \nonumber \\
\clo(\hat R_{\eta}) \cap \Ss_0 \subset \clo(\zeta_{p_{\eta}}) \nonumber
\end{gather} 
for the parabolic  fixed point $p_{\eta}$ on $\partial \Ss_{+}$ fixed by $\eta$, 
it follows that $\{g_{k}(L)\}$ has limit points in 
$\clo(\zeta_{p_{\eta}})$ 

Let $y\in \Ss_+$. 
We again write $\eta_{i} = \gamma^{k(i)} g_{i}$ so that $\eta_{i}(y)$ is in a conical region $R$ as in the proof of Proposition \ref{prop:boundedness}. 
The sequence
$\{\eta_i(y)\}$ accumulates only to $(\clo(R)\cap \partial \Ss_+) \setminus \{p_{\eta}\}$.
Again $k(i) \ra \pm \infty$ since $g_{i}(y) \ra p_{\eta}$. 
Now, 
\[g_{i}(x_{i}) = \gamma^{-k(i)}\eta_{i}(x_{i}), x_{i} \in L,\] 
cannot lie on  $\hat R_{\eta}$ by Lemma \ref{lem:finiteR} since $k(i) \ra \pm \infty$.

For the final item, we can choose a new parabolic region
$R_{\eta}'$ sufficiently far away 
so that $R_{\eta}' \cap \hat R_{\eta} \cap \tilde M_{(J)} =\emp$. 
Then $R_{\eta}' \cap \tilde M_{(J)} = \emp$
by the parabolic action of $\langle \eta \rangle$. 
\end{proof}  


Recall $\tilde \Sigma$ from \eqref{eqn:tSigma}. 
\begin{corollary}[Finiteness]\label{cor:FmMi} 
Let $F$ be a topological polytope in $\Lspace$ bounded by finitely many crooked-circle disks.	
Suppose that every pair of these disks the closures of which
contain $\zeta_p$ for a parabolic fixed point $p$ satisfy the properties of 
$D_{1}$ and $D_{2}$ in Proposition \ref{prop:parabound}. 
Assume 
\[\clo(F) \cap \Ss \subset \clo(F) \cap \left(\tilde \Sigma \cup \bigcup_{k\in \mathcal{I}_F}\clo(\zeta_{p_{k}})\right)\]
for a finite subset $\mathcal{I}_F \subset {\mathcal I}'\setminus {\mathcal I}$. 
Then the subspaces $F \cap \tilde M_{(J)}$ and
	 $\clo(F)\setminus \tilde M_{(J)}^o$ are both compact 
	 topological polytopes for each $J$.
\end{corollary}
\begin{proof} 
	The premise says that 
	$F$ is disjoint from $\Lambda_{\Gamma}$ except at $\bigcup_{k\in \mathcal{I}_F}\clo(\zeta_{p_{k}})$. 
	Propositions \ref{prop:boundedness} and \ref{prop:parabound} imply
	that $\bigcup_{\gamma \in \Gamma} \gamma(L) \cap F = \tilde M_{(J)} \cap F$ can have accumulation points outside 
	itself only
	in the compact surface 
	\[(\tilde \Sigma \cap \clo(F))\setminus \bigcup_{k \in {\mathcal{I}}_F}	{{\mathscr{H}}}_{k} \subset \clo(F) \cap \left( \Ss_+\cup \Ss_- \cup \bigcup_{i \in {\mathcal I}} (A_i \cup a_i \cup \mathcal{A}(a_i))\right)\] by \eqref{eqn:tSigma}.
	This set is disjoint from 
	$\bigcup_{x\in \Lambda_{\Gamma, \Ss_+}} \clo(\zeta_x)$. 
	The existence of the accumulation points in here 
	contradicts  \purple{Theorem \ref{thm:equivalence} and}  Corollary \ref{cor:conv1}. 
	Hence, $\tilde M_{(J)} \cap F$ is a bounded subset of $F$. 
	
	Also, $\clo(F)\setminus M_{(J)}^o$ is bounded by a union of finitely many smooth finite-type surfaces. 
	Hence, it is a compact topological polytope. 
\end{proof}


\begin{figure}[h]

\begin{center}
\adjincludegraphics[height=7cm, trim={0 {.4\width} 0 {.14\width}},clip]{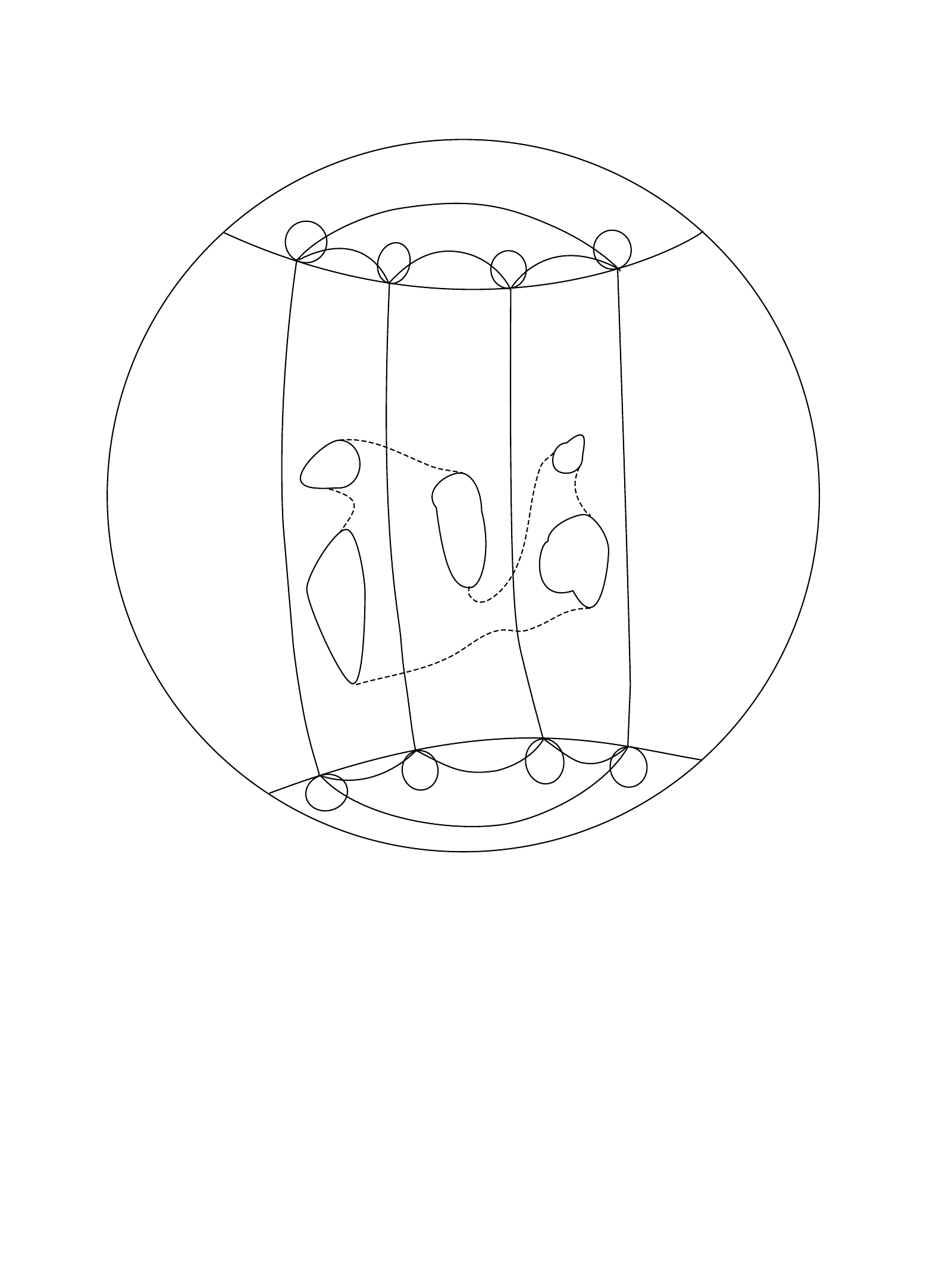}
\end{center} 
\caption{The fundamental domain bounded by disks $\mathcal{D}_{j}$, $j=1, \dots, 2{\mathbf{g}}$, and some horodisks 
drawn topologically.} 
\label{fig:decompM}
\end{figure}




\subsubsection{Choosing the candidate fundamental domain and side-pairing disks} \label{subsub:chooseF} 


\begin{figure}[h]

\begin{center}
\includegraphics[height=6cm, trim={0.5cm 1.5cm 0.5cm 1.5cm}]{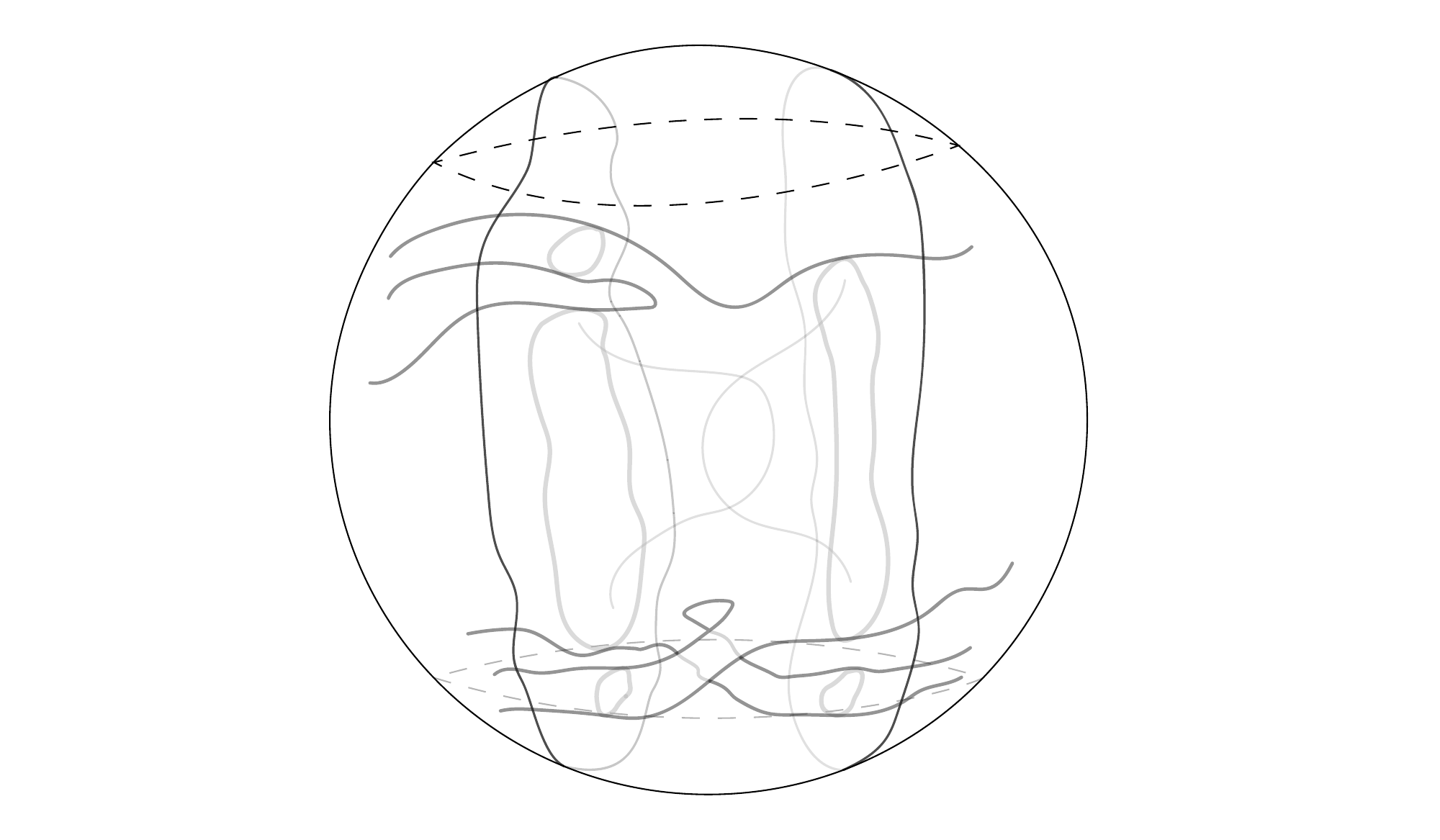}
\end{center} 
\caption{ $\tilde M_{(J)}$ meeting with disks.} 
\label{fig:FigDT}
\end{figure}





Suppose that $\mathbf{g}$ is the rank of $\Gamma$. 
We recall from Section 7 of \cite{CG17}. 
Now $\Sigma_{+}$ denote the surface
\[((\Ss_{+} \cup \partial \Ss_{+})\setminus \Lambda_{\Gamma, \Ss_{+}})/\Gamma,\]
where $\Sf$ is a dense subset of $\Sigma_{+}$ with $\chi(\Sf) = 1 - \mathbf{g}$. 
Also, $\Sigma_{+}= (\Ss_{+} \cup \bigcup_{i \in {\mathcal I}} a_i)/\Gamma $. 
We add to $(\Ss_+\cup \partial \Ss_{+})\setminus \Lambda_{\Gamma, \Ss_{+}}$ 
the set of ideal parabolic fixed points $a_i, i\in \mathcal{I}'\setminus \mathcal{I}$. 
The topology is given by a basis consisting of horodisks with fixed points added
or the open disks in $(\Ss_+\cup \partial \Ss_{+})\setminus \Lambda_{\Gamma, \Ss_{+}}$. 
We obtain a new surface 
$\hat \Sigma_{+} :=  \Ss_{+} \cup \bigcup_{i \in {\mathcal I}'} a_i/\Gamma $. 

We choose a collection $\{ \hat d_{j}| j=1, \dots, m_{0}\}$ of disjoint geodesics ending at one of 
the ideal vertices or the boundary arc of $\hat \Sigma_{+}$
so that the complement of their union is the union of mutually disjoint open regions, each of which is {homeomorphic to one 
of the following:} 
\begin{itemize} 
\item a hexagon where three alternate edges are arcs in $\partial \hat \Sigma_{+}$, 
\item a pentagon with one ideal vertex (collapsed from a boundary component) and two alternate edges in
$\partial \hat \Sigma_{+}$,
\item a quadrilateral with two ideal vertices (collapsed from two boundary components) 
and one edge in $\partial \hat \Sigma_{+}$, or 
\item a triangle with three ideal vertices.
\end{itemize} 
We may choose a set $\hat d_{i_{j}}$, $j=1, \dots, 2{\mathbf{g}}$, where
the complement of their union is a connected cell. 
We relabel these to be $\hat d_{1}, \dots, \hat d_{2{\mathbf{g}}}$. 

The lifts of the geodesics are geodesics in $\Ss_+$
ending at points of $\bigcup_{i \in {\mathcal I}} a_i$. 





\begin{lemma} \label{lem:sidetube} 
We can choose the mutually disjoint collection ${\mathcal D}_{j} \subset \Lspace$ of properly embedded open disks
and a tubular neighborhood $T_{j} \subset \clo({\mathcal D}_{j})$ of $\partial {\mathcal D}_{j}$ for each $j$, 
$j=1, \dots, 2{\mathbf g}$, 
that form a matching set $\{T_{j}|j=1, \dots, 2{\mathbf{g}}\}$ for a collection 
$\mathcal{S}_{0}$ of generators of $\Gamma$. Finally, 
$\partial \mathcal{D}_{j} = d_{j} \cup \mathcal{A}(d_{j}) \cup \bigcup_{x\in \partial d_{j}} \clo(\zeta_{x})$
for a lift $d_{j}$ of $\hat d_{j}$. 
\end{lemma} 
\begin{proof} 
We choose lifts $d_{1}, \dots, d_{2{\mathbf{g}}}$ of $\hat d_{1}, \dots, \hat d_{\mathbf{g}}$ 
bounding a connected fundamental domain in $\clo(\Ss_{+})$. 
Since a component $L$ of 
\[\Ss_+\setminus \bigcup_{g \in \Gamma}\bigcup_{i=1}^{2{\mathbf{g}}} g(d_j)\] 
is the fundamental domain of the $\Gamma$-action on $\Ss_+$, 
we obtain $\gamma_1, \dots, \gamma_{\mathbf{g}}$ generating $\Gamma$ forming 
a matching collection $\mathcal{S}_0$
by adding $\gamma_1^{-1}, \dots, \gamma_{\mathbf{g}}^{-1}$. 
Label $\gamma_j^{-1} = \gamma_{{\mathbf{g}}+j}$ for $j=1, \dots, {\mathbf{g}}$. 
Hence, we may assume that
$\gamma_j(d_j) = d_{{\mathbf{g}}+j}$ for $j=1, \dots, 2\mathbf{g}, \hbox{ mod } 2{\mathbf{g}}$
and $\{d_1, \dots, d_{2{\mathbf{g}}} \}$ is a matching set for $\mathcal{S}_0$. 

Let $p_{1}, \dots, p_{m_1}$ denote the set of parabolic fixed points on any of $d_{j}$. 
By choosing the parabolic regions $R_{p_{1}}, \dots, R_{p_{m_1}}$  sufficiently far away, we may assume that these are mutually disjoint. 
(See Section \ref{subsec:parabolic}.  {Temporarily, we are not using the terminology of Section \ref{subsub:pararegion}.} )
	
We remove the interior of $R_{p_{j}}, j=1, \dots, m_1$ from $\Lspace$. 
Let $S_{p_{j}}$ denote the ruled surface boundary in $\Lspace$ of $R_{p_{j}}$
where $\mathbf{g}_{j}^{t}, t \in \bR$ acts on.
$R_{p_{j}}$ meets $\Ss_{+}$ in a closed horodisk ${\mathscr{H}}_{p_{j}}$. 
Then we define
\begin{equation}\label{eqn:newsurf}  
\tilde \Sigma^{\ast} := 
\left(\tilde \Sigma \cup \bigcup_{j=1}^{m_1} S_{p_{j}}\right) 
\setminus \bigcup_{j=1}^{m_1} {\mathscr{H}}^{o}_{p_{j}}  \setminus \bigcup_{j=1}^{m_1} {\mathscr{H}}^{o}_{p_{j},-}. 
\end{equation} 
We assume that $d_{j}$ to be disjoint from ${\mathscr{H}}_{p_{k}}$ if $p_{k}$ is not an endpoint of $d_{j}$
by taking the cusp neighborhood sufficiently small. 

\begin{itemize} 
\item For each geodesic segment $d_{j}$ passing ${\mathscr{H}}_{p_{k}}$ for some $k$, 
we let $d'_{j}:= d_{j}\cap \tilde \Sigma^{\ast}$. 
For each point of $x\in \partial d'_{j} \cap \Ss_+$, we obtain a line $L_{x}$ in the ruled surface $S_{p_{j}}$. (See Appendix \ref{app:A}.)
We denote it by $\zeta_{x}$.
\item For the endpoint of $d'_j$ in $\partial \Ss_+$, i.e., 
in $\bigcup_{i\in \mathcal{I}} a_i$,  
we already defined $\zeta_x$ in Definition \ref{defn:semicircle}. 
\end{itemize} 

We define \[\tilde d_{j} = d'_{j} \cup \mathcal{A}(d'_{j}) \cup 
\bigcup_{x \in \partial d'_{j}} \clo(\zeta_{x}).\] 
Then $\tilde d_{j}\cap \tilde d_{k} = \emp$ for $j\ne k$, $j, k =1, \dots, 2{\mathbf{g}}$ since
$\{d'_{j}|j = 1, \dots, 2{\mathbf{g}}\}$ is a mutually disjoint collection of simple closed curves. 
Since $d_{j}\cap {\mathscr{H}}_{p_{k}}$ is a geodesic ending in $p_{k}$ or is empty for all $j, k$ by our choice, 
$d_{j}\cap \partial_h {\mathscr{H}}_{p_{k}}$ is {the} unique point or is empty. 
Also, \[ d_{j}\cap \partial_h {\mathscr{H}}_{p_{k}}, j=1, \dots, 2{\mathbf{g}},\] are distinct for a fixed $k$
as $d_{j}$ are mutually disjoint. Thus, $\{\clo(\zeta_{x}), x \in \partial d'_{j}\}$ is a mutually disjoint collection.

Furthermore, $\tilde d_{1}, \dots, \tilde d_{2{\mathbf{g}}}$ form a matching set for $\mathcal{S}_{0}$.

For each $x \in \partial d'_{j}$, $j=1, \dots, {\mathbf{g}}$, we take
\begin{itemize} 
\item for $x \in \partial \Ss_+$, a disk $Z_x$ of the form \[\bigcup_{y \in b} \clo(\zeta_{y}) \subset A_{k} = \bigcup_{y\in a_k} \clo(\zeta_{y})\] 
where $b$ is a small open interval in $a_k$ 
and $x \in a_k \cap \partial d'_{j}$, and
\item for $x \in \partial {\mathscr{H}}_{p_j}$, a ruled tubular open neighborhood $Z_{x}$ of $\clo(\zeta_{x})$ in the ruled surface $\clo(S_{p_{j}})$. 
\end{itemize} 
Here, each $Z_{x}$, $x \in \partial d'_j$, $j=1, \dots, \mathbf{g}$, 
is chosen sufficiently thin so that under elements of $\mathcal{S}_{0}$, the collection of $Z_{x}$ 
and their images
is a collection of mutually disjoint  sets. 
We take a union of all of these disks with  
\[(\Ss_{+} \cup \Ss_{-})\setminus \bigcup_{j=1}^{m} {{\mathscr{H}}}_{p_{j}}\setminus \bigcup_{j=1}^{m} {{\mathscr{H}}}_{p_{j},-} \]
 to $\Lspace\setminus \bigcup_{j=1}^{2{\mathbf{g}}} R_{p_{j}}$ to obtain a $3$-manifold with boundary. 
Then we can apply {Theorem \ref{thm:Dehn}}  for the simple closed curve $\tilde d_{j}$ to obtain open disks $\mathcal{D}'_{j}$ 
so that $\tilde d_{j} = \partial \mathcal{D}'_{j}$ 
for each $j=1, \dots, {\mathbf{g}}$. These are chosen to be mutually disjoint by the same theorem. 

Then we obtain $\mathcal{D}'_{j+{\mathbf{g}}}$ as the image 
$\gamma_{j}(\mathcal{D}'_{j} )$ for $\gamma_{j}\in \mathcal{S}_{0}$. 
Since the boundary components of $\mathcal{D'}_{j}$, $j=1, \dots, 2{\mathbf{g}}$, are mutually 
disjoint, using {Theorem \ref{thm:Dehn}} again, we may do disk exchanges to obtain 
mutually disjoint disks $\mathcal{D}''_{j}$, $j=1, \dots, 2{\mathbf{g}}$.

Let $p_{k}$ be a parabolic fixed point where $d_{j}$ ends. 
Now for each $\zeta_{x}, x\in \partial d'_{j}$, is in the boundary of a leaf of the foliation
$\mathcal{D}_{f, r}$ in $R_{p_{k}}$ obtained by Theorem \ref{thm:Sr}.
For each such $p_{k}$ and $d_{j}$, 
we add the disk to $\mathcal{D}'_{j}$ by joining them at each $\zeta_{x}, x\in \partial d'_{j}$. 
We call the results $\mathcal{D}_{j}$, $j=1, \dots, 2{\mathbf{g}}$. These are mutually disjoint. 

Now, 
\[\clo(\mathcal{D}_j)\cap \clo(R_{p_{k}}) \cap \Ss_{+} = {\mathscr{H}}_{p_{k}} \cap d_{j} = {\mathscr{H}}_{p_{k}}\cap \tilde d_{j}.\] 
Hence, by adding these arcs back, we obtain 
\[\partial \mathcal{D}_{j} = d_{j} \cup \mathcal{A}(d_{j}) \cup \bigcup_{x\in \partial d_{j}} \clo(\zeta_{x}).\] 

Since we do not change the sufficiently thin tubular neighborhoods of $\partial \mathcal{D}_{j}$ 
under the above disk exchanges, there exists 
a matching collection of tubular neighborhoods $\{ T_{j}| j=1, \dots, 2{\mathbf{g}} \}$ of $\{\partial \mathcal{D}_{j}|j=1, \dots, 2{\mathbf{g}}\}$  
under $\mathcal{S}_{0}$.

\end{proof} 




Here, of course, the disk collection is not yet a matching set
under $\mathcal{S}_0$. 
By Lemma \ref{lem:sidetube},  
the collection $\{\clo({\mathcal D}_{j})\}$ are mutually disjoint 
in $\clo(\Lspace)$. 
The collection  ${\mathcal D}_{j}$, $j=1, 2, \dots, 2{\mathbf{g}}$, 
bound a region $\mathbf{F}$ closed in $\Lspace$ 
with a compact closure in $\clo(\Lspace)$, 
a finite-sided polytope in the topological sense.  

Now we consider  
$K_0$ {to} be the set 
\[ \bigcup_{j=1}^{2{\mathbf{g}}} ({\mathcal D}_{j}\setminus T_{j}^{o}) 
\cup \bigcup_{1\leq j < k \leq 2{\mathbf{g}} } ({\mathcal D}_{j} \cap {\mathcal D}_k).\]
By Proposition \ref{prop:exhaust},
we choose $M_{(J)}$ in our exhaustion sequence of $M$ so that 
\begin{equation}\label{eqn:MJchoice}
\tilde M_{(J)} \supset N_{\bdd, \eps}\left(K_{0}\right)
\end{equation} 
for an $\eps$-neighborhood, $\eps > 0$. 




\subsubsection{Outside tameness} \label{subsub:outt}

The following is enough to prove tameness.  
\begin{proposition}[Outside Tameness] \label{prop:outfund} 
Let $M$ denote a Margulis space-time $\Lspace/\Gamma$ where
$\Gamma$ is an isometry group with $\mathcal{L}(\Gamma) \subset \SO(2, 1)^{o}$. 
	Let $\mathbf{F}$ be the domain bounded by $\bigcup_{i=1}^{2{\mathbf{g}}}{\mathcal D}_{i}$.   Suppose that $M_{(J)}$ satisfies \eqref{eqn:MJchoice}. 
	Then $\mathbf{F}\setminus \tilde M_{(J)}$ is a fundamental domain of 
	$M\setminus M_{(J)}$, and $M$ is tame. 
	Furthermore, $\bigsqcup_{i=1}^{2{\mathbf{g}}}{\mathcal D}_{i}\setminus \tilde M_{(J)}$ 
	embeds to a union of mutually disjoint properly embedded surfaces in $M$. 
	\end{proposition} 
\begin{proof} 

By Corollary \ref{cor:FmMi},   
$\mathbf{F}\setminus \tilde M_{(J)}^o$ is a tame $3$-manifold. 
Let $X$ denote  $\mathbf{F}\setminus \tilde M_{(J)}^o$, a tame $3$-manifold bounded 
by a union of finitely many compact surfaces. 
\[M_{(J+1)}\setminus M_{(J)}^{o} \subset M_{(J+2)}\setminus M_{(J)}^{o} \subset \cdots \] is an exhausting 
sequence of compact submanifolds in $M\setminus M_{(J)}^{o}$. 
	Since ${\mathcal D}_{i}\setminus \tilde M_{(J)}^o \subset T_i$, 
	\[\{{\mathcal D}_{i}\setminus \tilde M_{(J)}^o|i=1, \dots, 2{\mathbf{g}}\}
	\subset \Bd \mathbf{F}\cap \Lspace \] 
	is a matching collection under $\mathcal{S}_{0}$
	by Lemma \ref{lem:sidetube}. 
	Also, $\mathbf{F} \cap (\tilde M_{(J+n)}\setminus \tilde M_{(J)}^{o})$ for each $n$ 
	is a compact topological polytope by Corollary \ref{cor:FmMi}. 
	By Proposition \ref{prop:Poincare}, $X$ is the fundamental domain of 
	$\Lspace\setminus \tilde M_{(J)}^o$. 
	Hence, $M\setminus M_{(J)}^o$ is tame. 
%
	
	The tameness of $M$ follows since $M\setminus M_{(J)}^o$ and $M_{(J)}$  
	are tame. 
 The last statement follows since $\bigsqcup_{i=1}^{2{\mathbf{g}}}{\mathcal D}_{i}\setminus \tilde M_{(J)}^{o}$ 
 is the boundary of a fundamental domain in $\Lspace\setminus \tilde M_{(J)}^o$. 
	\end{proof}

\subsubsection{Considering the whole disks ${\mathcal D}_{i} \cap \tilde M_{(J)}$} \label{subsub:fund} 

We consider bounded components of ${\mathcal D}_{i}\setminus \tilde M_{(J)}^{o}$ for $i=1, \dots, 2{\mathbf{g}}$. 
By Proposition \ref{prop:outfund}, the union of these planar surfaces embeds to 
the union of disjoint ones in $M$. We take the mutually disjoint
 thin tubular neighborhoods of the images of compact planar components 
 of ${\mathcal D}_{i}\setminus \tilde M_{(J)}^{o}$ and take the inverse image to $\Lspace$. 
 We add these to $\tilde M_{(J)}$. 
Let us call the result $\tilde M_{(J)} \subset \tilde M$ again. 
Since $\Gamma$ acts on $\tilde M_{(J)}$, 
we obtain a compact submanifold $M_{(J)}$ in $M$. 



Thus, by {Theorem \ref{thm:Dehn}} applied to $M_{(J)}$, each component of ${\mathcal D}_{i} \cap \partial \tilde M_{(J)}$
bounds a disk mapping to a mutually disjoint collection of embedded disks in $M_{(J)}$. 
We modify $\mathcal{D}_{i}$ by replacing each component of  ${\mathcal D}_{i} \cap \tilde M_{(J)}$ with
lifts of these disks.
(See \cite{GL84} and \cite{Hempel04} for some details.)

The results are still embedded in $\tilde M$ since we modify only inside $\tilde M_{(J)}$
where the disks are {also disjoint}.

Hence, we conclude 
\begin{align} \label{eqn:matchD} 
 \gamma_{j}({\mathcal D}_{j}) = {\mathcal D}_{j+{\mathbf{g}}}  
&\hbox{ for } \gamma_{j}\in \mathcal{S}, \hbox{ and } \nonumber\\
  \gamma_{j}({\mathcal D}_{l}) \cap {\mathcal D}_{m}  =\emp
& \hbox{ for } (j, l, m) \ne (j, j, j+{\mathbf{g}}) \mod 2{\mathbf{g}}.
 \end{align}

We summarize 
\begin{proposition} \label{prop:FundD} 
	{Let $M$ denote a Margulis space-time $\Lspace/\Gamma$ where
	$\Gamma$ is an isometry group with $\mathcal{L}(\Gamma) \subset \SO(2, 1)^{o}$.} 
Then there exists a fundamental domain $\mathcal{R}$ closed in $\Lspace$ bounded by finitely many crooked-circle disks
$\mathcal{D}_{j}$, $j=1, \dots, 2{\mathbf{g}}$. 
Moreover, $\clo(\mathcal{R}) \cap (\Lspace \cup \tilde \Sigma)$ is the fundamental domain of 
a manifold $(\Lspace \cup \tilde \Sigma)/\Gamma$ with boundary $\Sigma$.
Here, {$\mathcal{R}^o$} and $\clo(\mathcal R)$ are $3$-cells, and 
$\Lspace/\Gamma$ is homeomorphic to the interior of a handlebody of genus $\mathbf{g}$.
\end{proposition} 
\begin{proof} 
	

		Let $\mathcal{R}$ be the region in $\Lspace$ with boundary equal to 
	$\bigcup_{j=1}^{2{\mathbf{g}}}\mathcal{D}_{j}$. 
	\oldch{Since $\mathcal{D}_{j}$ is a properly embedded separating {disk} in a cell, 
		repeated applications of Lemma 1.12 of \cite{Marden74} imply that $\mathcal{R}^o$ is a cell. 
Since $\clo(\mathcal{R})$ is a polyhedral manifold whose interior is a $3$-cell, 
it is a $3$-cell. }
	
Since by \eqref{eqn:matchD}, 
\[\{\mathcal{D}_{j}| j = 1, \dots, 2{\mathbf{g}}\}\] is a matched set under $\mathcal{S}_{0}$, 
$\mathcal{R}$ is the fundamental domain by Proposition \ref{prop:Poincare}. 
The quotient space is homeomorphic to the interior of a handlebody
since we can find a homeomorphism of $\mathcal R$ to the standard $3$-ball 
where $\mathcal{D}_{j}$, $i=1, \dots, 2{\mathbf{g}}$, correspond to disjoint open disks with piecewise smooth boundary. 

Also, 
\[\{\clo(\mathcal{D}_{j}) \cap (\Lspace \cup \tilde \Sigma)| j = 1, \dots, 2{\mathbf{g}}\}\]
is a matched set under $\mathcal{S}_{0}$. 
Also, every point in $\tilde N:=\Lspace\cup \tilde\Sigma$
is equivalent to a point of $\mathcal{R}':= \clo(\mathcal{R}) \cap \tilde N$
by the action of $\Gamma$. 
Hence, $\mathcal{R}'$ is a fundamental domain
 giving us the properness of the action of $\Gamma$ on $\tilde N:= \Lspace\cup \tilde\Sigma$. 
Thus, $\tilde N/\Gamma$ is a manifold with boundary $\Sigma$. 

\end{proof}


This proves the first part of Theorem \ref{thm:main}. 
{The remaining part of Theorem \ref{thm:main}} 
will be completed in Section \ref{sub:compactification}.


\subsection{Parabolic regions and the intersection properties} \label{sub:parab}

{We will now choose the parabolic regions so that their images under the deck transformation groups are mutually disjoint. 
We will need this in Proposition \ref{prop:Pdisj}.} 

The basic idea used is that disks are separating 
$\Lspace$ into two components. 
We choose the parabolic regions for each parabolic point in the closure of the fundamental domain so that 
they meet the fundamental domain $\mathcal{R}$ is in a nice manner.
Using this, we can show that each image of a parabolic region meets an image of the fundamental domain in finitely many 
manners. This will essentially give us the needed intersection properties. 

\KFAR{
To explain the ideas of our proof, we go to the $2$-dimensional case: 
we can show without using the thick and thin decomposition theory that the horodisks of parabolics of a Fuchsian groups can be made 
mutually disjoint if  we choose them sufficiently small. 
We use the tessellations by an ideal fundamental polygonal disk $D$ 
in the hyperbolic plane with a Fuchsian group action. 
$D$ can intersect a horodisk in ``sectors'' of 
the horodisk.
We control the horodisks meeting $D$ so that they only meet in
a sector and they are disjoint. Any image of these horodisks 
must meet with a corresponding image of $D$ in a sector. For a horodisk $B$,
we cover $B$ by a connected union $U_1$ finitely many images of $D$ 
and the images of $U_1$ by
the cyclic group action of parabolic elements $\langle \gamma_B\rangle$ 
acting on $B$.  
We take the union $U_B$. 
Then the boundary of the union $U_B$ is separating other images of the horodisks with 
parabolic fixed points outside the closure of $U_B$ with $B$. 
Thus, we can control the intersections of $B$ with the images of $B$ 
by considering what happens inside $U_B$
and hence inside $U_1$ up to the  $\langle \gamma_B \rangle$-action.
}

The above fundamental domain 
$\mathcal{R}$ is bounded by a union of disks
$\mathcal{D}_{i}, i=1, \dots, 2{\mathbf{g}}$, in the boundary of $\mathcal{R}$. 
We call $\mathcal{D}_{i}$ the \hypertarget{term-fd}{{\em facial disks}} of $\mathcal{R}$. 
By construction, the closure of $\mathcal{D}_{i}$ is disjoint from $\Lambda_{\Gamma, \Ss_+} \subset \partial \Ss_{+}$
except for parabolic points. The set of parabolic points meeting at least one $\clo(\mathcal{D}_{i})$ is 
a finite set $\{p_{1}, \dots, p_{m_{1}}\}$.  (Possibly two or more of $p_{i}$s may be in the same orbit of $\Gamma$.) 



{We use the notations of Section \ref{subsub:pararegion}.} 
For each $p_{i}$, $i=1, \dots, m_1$,  we have a \hyperlink{term-pbr}{parabolic region} 
of the form $\gamma(\mathcal{P}_{j})$ for some $j$, $j=1, \dots, m_0$, and $\gamma \in \Gamma$
whose closure contains $p_i$. {For each $i$, we} denote by $\mathcal{P}_{i}$ this region $\gamma(\mathcal{P}_{j})$.
We denote by $\eta_i$ the parabolic element fixing $p_i$, following the boundary orientation if we remove $E$,  
and hence $\eta_i = \gamma \eta_j \gamma^{-1}$ for some $j=1, \dots, m_0$.


%

Now, we choose $\mathcal{P}_j$, $j=1, \dots, m_0$, so that 
$\clo(\mathcal{P}_i) \cap \Ss_+ $ for each $i$
equals a component of the inverse image of $E$.
{These $R_j$, $j=1, \dots, m_0$, form a mutually disjoint  collection of} closed cusp neighborhoods of $\Sf$ as given in
Section \ref{sub:thin}. 
$\clo(\mathcal{P}_i) \cap \Ss_+ $ for each $i$
equals a component of the inverse image of $E$. 
Also, by our construction in Theorem \ref{thm:ruled}, we have
$\clo(\mathcal{P}_i) \cap \Ss_-  = \mathcal{A}(\clo(\mathcal{P}_i) \cap \Ss_+).$


\begin{definition} \label{defn:nice} 
Let $p_{i}$ and $\mathcal{P}_{i}$ be as above for $i \in {\mathcal I}'\setminus {\mathcal I}$, and 
let $\eta_i$ be the parabolic primitive element fixing $p_i$. 
{Let $D_{f_i, r_i,t}$  denote the {\em canonically defined} properly embedded disks by Theorem \ref{thm:Sr}.}
We say that an image $\gamma(\mathcal{R})$, $\gamma \in \Gamma$, of the fundamental domain $\mathcal{R}$ {bounded by crooked-circle disks} 
meets {\em nicely} with $\eta(\mathcal{P}_{i})$, $\eta \in \Gamma$, if 
\begin{equation} \label{eqn:far} 
\eta(\mathcal{P}_{i}) \cap \gamma(\mathcal{R}) = \bigcup_{t\in [t_{1}, t_{2}]} D_{f_i, r_i, t} \hbox{ for some } t_{1}, t_{2}\in \bR, t_{1}< t_{2} 
\end{equation}
and
\begin{align} \label{eqn-etap} 
\clo(\zeta_{\eta(p_{i})}) & \subset \clo(\gamma(\mathcal{R})), 
\hbox{ and }\\
\eta(\mathcal{P}_{i}) & \subset  \Paren{\bigcup_{k \in \bZ} 
 \gamma \eta_{i}^{k}\left(\bigcup_{j=1}^{k_{0}}\kappa_{j}(\mathcal{R} )\right)}^{o}
\end{align} 
for a finite collection of $\{\kappa_{j} \in \Gamma\}$ 
where $\clo(\zeta_{\eta(p_{i})}) \subset \clo(\gamma(\kappa_{j}(\mathcal{R})))$. 
\end{definition}
Of course, by the definition $\gamma \circ \eta_i^k \circ \kappa_j(\mathcal{R})$, {$k\in \bZ$,}
meets with $\eta(\mathcal{P}_j)$ nicely as well. 


\begin{lemma} \label{lem:paraf} 
Let $\Gamma$ satisfy Criterion \ref{cr:positive}. 
{Let $\mathcal{R}$ be the fundamental domain of $\Lspace/\Gamma$ bounded by crooked-cricle disks
as constructed by Proposition \ref{prop:FundD}.}
Let $q$ be a parabolic fixed point in $\clo(\mathcal{R})$. Then 
$q = p_{i}$ for some $i$, $i =1, \dots, m_1$. 
 Moreover, the following hold\/{\rm :} 
\begin{itemize}
\item $\clo(\zeta_{p_{i}}) \subset \clo(\mathcal{R})$,  
\item $\clo(\zeta_{p_{i}}) $ is a subset of  the closures of exactly two 
facial disks $\mathcal{D}_{l}$ and $\mathcal{D}_{m}$ among the facial disks of
$\mathcal{R}$, and
\item 
the corresponding parabolic region $\mathcal{P}_{i}$ meets nicely with $\mathcal{R}$ provided we choose $\mathcal{P}_{i}$, $i=1, \dots, m_0$, sufficiently \hyperlink{term-far}{far away}.
\end{itemize} 
\end{lemma} 
\begin{proof} 
Since $q \in \clo(\mathcal{R})$, $q = p_i$ by the construction 
in Lemma \ref{lem:sidetube}  of 
$\partial \mathcal{D}_{j}$, $j=1, \dots, 2{\mathbf{g}}$.
Since the closure of $\mathcal{D}_{j}$ is compact,
either $\mathcal{D}_{j}$ contains $\zeta_{p_{i}}$ in its boundary, or 
there is an $\eps$-$\bdd$-neighborhood of {$\clo(\zeta_{p_{i}})$ disjoint} from it for some $\eps > 0$. 
We can choose the boundary ruled surface of 
 $\mathcal{P}_{i}$ sufficiently far  
so that {\eqref{eqn:far} holds and hence} only facial disks of $\mathcal{R}$ that meet $\mathcal{P}_{i}$ are the two 
\hyperlink{term-fd}{facial disks} whose closures contain $\zeta_{p_{i}}$.  (See Definition \ref{defn:para}.)
Let us call these $\mathcal{D}_{j}$ and $\mathcal{D}_{k}$. 



Now, $\partial \mathcal{P}_{i}\cap \Lspace = S_{i}$  is an open disk \hyperlink{term-sep}{separating} $\mathcal{P}_{i}$ from $\Lspace\setminus \clo(\mathcal{P}_{i})$. 
$\partial \mathcal{P}_{i} \cap \mathcal{R}$ has  {the boundary formed by} two lines respectively in $\mathcal{D}_{j}$ and $\mathcal{D}_{k}$.
We take { finitely many} images $\kappa_{l}({\mathcal{R}}), l=1, \dots, k_{0}$ of $\mathcal{R}$ with $\kappa_1 = \Idd$ so that 
$\kappa_{l}(\mathcal{R}) \cap \kappa_{l+1}(\mathcal{R})$ is a copy of $\mathcal{D}_{i_{0}}$ for some $i_{0}$ whose closure contains 
$\zeta_{p_{i}}$.  
Since the collection $\{\gamma(\mathcal{R}), \gamma \in \Gamma\}$ tessellates $\Lspace$, 
we can choose enough of $\kappa_{j}$ so that 
$\kappa_{k_{0}+1}(\mathcal{R}) = \eta_{i}^{\pm 1}(\kappa_{1}(\mathcal{R}))$ for 
either $+$ or $-$ sign. 

Except for the closures of facial disks of $\{\kappa_{j}(\mathcal{R})|j=1, \dots, k_{0}\}$ containing $\clo(\zeta_{p_{i}})$, 
the closures of other facial disks contained in the boundary of 
\[\{\kappa_{j}(\mathcal{R})|j=1, \dots, k_{0}\}\]
  are disjoint from $\clo(\zeta_{p_{i}})$. Let $\hat K$ denote the union of the closures of these images of facial disks 
  of $\{\kappa_{j}(\mathcal{R})|j=1, \dots, k_{0}\}$ disjoint from  $\clo(\zeta_{p_{i}})$. 
Since these are separating disks in $\mathcal H$, we may choose $\mathcal{P}_{i}$ sufficiently far so that 
$\clo(\mathcal{P}_{i}) \cap \hat K = \emp$. 

Since $\clo(\mathcal{P}_{i})$ is $\eta_{i}$-invariant, 
$\clo(\mathcal{P}_{i})$ is disjoint from $\bigcup_{m\in \bZ}\eta_{i}^{m}(\hat K)$. 
Now, $\bigcup_{m\in \bZ}\eta_{i}^{m}(\hat K)$ is a separating set in $\Lspace$.
A component of $\Lspace\setminus \bigcup_{m\in \bZ}\eta_{i}^{m}(\hat K)$ equals 
\[\Paren{\bigcup_{k \in \bZ}  \eta_{i}^{k}\left(\bigcup_{j=1}^{k_{0}}\kappa_{j}(\mathcal{R} )\right)}^{o}.\]
Since $\mathcal{P}_{i}$ is a connected set, we obtain
\begin{equation} \label{eqn:kj} 
\mathcal{P}_{i} \subset \Paren{\bigcup_{k \in \bZ}  \eta_{i}^{k}\left(\bigcup_{j=1}^{k_{0}}\kappa_{j}(\mathcal{R} )\right)}^{o}.
\end{equation} 
\end{proof} 
{Notice that \eqref{eqn:kj} gives us the conditions of Definition \ref{defn:nice}.}


\begin{proposition}\label{prop:intnice} 
Let $\mathcal{P}_{i}$ be a parabolic region for 
the parabolic fixed point 
$p_{i}$, $i=1, \dots, m_{0}$, 
as we chose at the beginning of Section \ref{sub:parab}.
We can always choose $\mathcal{P}_{i}, i=1, \dots, m_{0}$, so that for every pair $\eta, \gamma \in \Gamma$, 
so that exclusively one of the following holds\,{\rm :}
\begin{itemize}
\item $\eta(p_i) \not\in \gamma(\clo(\mathcal{R}))$, 
or  else
\item 
 $\gamma(\mathcal{R})$ meets $\eta(\mathcal{P}_{i})$ nicely, 
 and $\gamma(\mathcal{P}_j) = \eta(\mathcal{P}_i)$ for some 
 $j= 1, \dots, m_1$. 
\end{itemize} 
\end{proposition}
\begin{proof}
We may assume $\gamma=\Idd$ since we can change $\eta$ to $\gamma^{-1}\eta$. 
Then the result follows by Lemma \ref{lem:paraf}. 
\end{proof}

We choose $\mathcal{P}_{j}$ \hyperlink{term-far}{far away} for $j=1, \dots, m_{0}$ so that the conclusions of Proposition \ref{prop:intnice} 
are satisfied. 

Let $P = \bigcup_{\gamma \in \Gamma} \bigcup_{i=1,\dots, m_{0}} \gamma(\mathcal{P}_i)$, and 
let $\mathcal{P}_{\mathcal{R}} := (\mathcal{P}_{1}\cup \cdots \cup \mathcal{P}_{m_{1}}) \cap \mathcal{R}$. 


\begin{proposition} \label{prop:Pdisj}
We can choose the sufficiently \hyperlink{term-far}{far away}  parabolic regions 
\[ \mathcal{P}_{1}, \dots, \mathcal{P}_{m_{0}}\]
meeting $\mathcal{R}$ nicely 
so that they are disjoint in $\Lspace$. 
Then the following hold\,{\rm :}
\begin{itemize}
\item The following are equivalent\/{\rm :}  
\begin{enumerate}
	\item[(1)] $\gamma(\mathcal{P}_{i})$ meets $\mathcal{R}$ nicely. 
	\item[(2)] $\gamma(\mathcal{P}_{i}) = \mathcal{P}_{j}$ for some $j$, $j=1, \dots, m_1$. 
	\item[(3)] $\gamma(\mathcal{P}_{i}) \cap \mathcal{R} \ne \emp$.
\end{enumerate}
\item $\mathcal{R}$ meets only $\mathcal{P}_{1}, \dots, \mathcal{P}_{m_{1}}$
among all images $\gamma(\mathcal{P}_r)$ for $\gamma \in \Gamma, r=1, \dots, m_0$. 
\item Moreover,  for every pair $\gamma, \eta \in \Gamma$,
\[\gamma(\mathcal{P}_{j}) \cap \eta(\mathcal{P}_{k}) = \emp \hbox{ or } 
\gamma(\mathcal{P}_{j}) = \eta(\mathcal{P}_{k}),\, j, k =1, \dots, m_{0}. \]
\end{itemize}
\end{proposition}
\begin{proof}
We first choose $\mathcal{P}_{i}, i=1, \dots, m_{0}$, sufficiently far so that
 $\mathcal{P}_{i} \cap \mathcal{P}_{j} \cap \mathcal{R} = \emp$ for $i\ne j$, 
 $i, j=1, \dots, m_1$, 
 and every $\mathcal{P}_{j}$, $j=1, \dots, m_1$, meets $\mathcal{R}$ nicely by Proposition \ref{prop:intnice}.  


{Obviously (2) implies (1).} 
For the first item, we show that (1) implies (2): 
Suppose $\gamma(\mathcal{P}_{j})$ meets $\mathcal{R}$ nicely. 
Then $\gamma(p_{j})$ is in $\clo(\mathcal{R})$.
Since $\gamma(p_{j})$ is a parabolic fixed point,  
it equals $p_{l}$ for some  $l=1, \dots, m_{1}$.
Only elements of $\Gamma$ fixing $p_{l}$ are 
of the form $\eta_{l}^{m}$ for some integer $m\in \bZ$. 
\[\gamma(\clo(\mathcal{P}_j)) \cap \Ss_+ = \gamma({\mathscr{H}}_k) \cup \gamma(\partial_h {{\mathscr{H}}_k})\] for 
a horodisk ${\mathscr{H}}_k$. Now, $\gamma({\mathscr{H}}_k) = {\mathscr{H}}_l$. 
Hence, the parabolic group acting on 
$\gamma(\mathcal{P}_j)$ is the same one acting on 
$\mathcal{P}_l$. 
By our choice of $\mathcal{P}_j$ in Section \ref{sub:parab} 
from choosing orbit representatives of parabolic fixed points, 
we obtain 
\[\gamma(\mathcal{P}_{j}) = \mathcal{P}_{l} \hbox{ for some } l=1, \dots, m_{1}.\] 
Clearly, (2) implies (3) by Lemma \ref{lem:paraf}.

Now we show that (3) implies (2):  
Suppose that $\gamma(\mathcal{P}_j) \cap \mathcal{R} \ne \emp$.
Now, $\gamma(p_{j}) \in \eta(\clo(\mathcal{R}))$ for some $\eta \in \Gamma$, 
and $\gamma(\mathcal{P}_{j})$ meets $\eta(\mathcal{R})$ nicely
and $\gamma(\mathcal{P}_{j}) = \eta(\mathcal{P}_{k})$ for some $k$
by Proposition \ref{prop:intnice}.  
Moreover, 
\begin{equation}\label{eqn:etaPj}  
\eta(\mathcal{P}_k) \subset 
\Paren{\bigcup_{l \in \bZ} \eta \eta_{k}^{l}
	\left(\bigcup_{r=1}^{k_0}\kappa_r(\mathcal{R})\right)}^{o}
\end{equation}
by Proposition \ref{prop:intnice}.  
Hence, $\gamma(\mathcal{P}_{j})$ meets with only the images of $\mathcal{R}$ of the form $\eta\eta_{k}^{l}\kappa_j(\mathcal{R})$. 
If 
\[\eta\eta_{k}^{l}\left(\kappa_j(\mathcal{R})\right) = \mathcal{R} \hbox{ for some } k, l,j, \]
then 
$\gamma(\mathcal{P}_{j})$ meets $\mathcal{R}$ nicely since 
we can check Definition \ref{defn:nice}. 
If  $\eta\eta_{k}^{l}\left(\kappa_j(\mathcal{R})\right) \ne \mathcal{R}$ for all $k, l, j$, 
then $\gamma(\mathcal{P}_{j})$ does not meet $\mathcal{R}$
by \eqref{eqn:etaPj}. 
 (2) implies (1) by Lemma \ref{lem:paraf}.
We proved the first item.  

The second item follows from it. 


Suppose that two respective images $P'_{i}$ and $P'_{j}$  of some $\mathcal{P}_{k}$ and $\mathcal{P}_{l}$ 
for $k, l=1, \dots, m_{0}$ meets in a nonempty set. 
Hence, they meet in $\gamma(\mathcal{R})$ for some $\gamma \in \Gamma$. 
Thus, \[\gamma^{-1}(P'_{i}) \cap \gamma^{-1}({P}'_{j}) \cap \mathcal{R} \ne \emp.\] 
The first item implies that 
\[\mathcal{P}_{j}=\gamma^{-1}(P'_{i}) \hbox{ and } \mathcal{P}_{l}=\gamma^{-1}({P}'_{j}) \hbox{  for some } l, m = 1, \dots, m_1\] 
However, $\mathcal{P}_{j}\cap \mathcal{P}_{l} \cap \mathcal{R} = \emp$
or $P'_i= P'_j$ by our construction of parabolic regions.

\end{proof} 


\subsection{Relative compactification}\label{sub:compactification}

\subsubsection{Proof of Theorem \ref{thm:main}} \label{subsub:thmmain}
{Proposition \ref{prop:FundD} proves the first part of the theorem.}
First, we recall our bordifying surface as defined by \eqref{eqn:tSigma}:
\begin{equation*}
\tilde \Sigma_0 := \Ss_+ \cup \Ss_- \cup \bigcup_{i \in {\mathcal I}} (A_{i} \cup a_{i} \cup \mathcal{A}(a_{i})).
\end{equation*} 
$\Sigma := \tilde \Sigma_{0}/\Gamma$ and $N := (\Lspace\cup \tilde \Sigma)/\Gamma $, which is a manifold by
Proposition \ref{prop:FundD}. 






By Proposition \ref{prop:Pdisj}, we define $P$ to be a union of mutually disjoint parabolic regions of the form $\gamma(\mathcal{P}_{i})$ 
for $\gamma \in \Gamma$, $i=1, \dots, m_{0}$. 
Since the boundary of their union in $\Ss$ is the union of mutually disjoint 
closed horodisks, their closures in $\mathcal{H} = \clo(\Lspace)$ are mutually disjoint. 
Now, we take the closure $\clo(P)$ of $P$ and take the relative interior $P'$ in the closed hemisphere $\mathcal{H}$. 
Let $\partial_{\Lspace} P'$ denote $\Bd P' \cap \Lspace$. 
Then define $\tilde N':= (\Lspace \cup \tilde \Sigma )\setminus P'$. 
$\Gamma$ acts properly discontinuously on $\tilde N'$ since $\tilde N'$ is a  $\Gamma$-invariant proper subspace of $\tilde N$. We note that $\partial_{\Lspace} P'$ is 
transversal to $\Ss$. 
Thus, $N':= \tilde N'/\Gamma$ is a manifold. 

The manifold boundary $\partial N'$ of $N'$ is
 \[((\tilde \Sigma\setminus P' )\cup \partial_{\Lspace } P')/\Gamma.\] 
Define $P'' = P'/\Gamma$.
Also, $(\partial_{\Lspace} P')/\Gamma$ is a union of a finite number of disjoint annuli.
$\partial N'$ is homeomorphic to $(\Sigma\setminus P'') \cup (\partial_{\Lspace} P')/\Gamma$.

Recall that the union of facial disks $\mathcal{D}_{i}$, $i=1, \dots, 2{\mathbf{g}}$, 
bounds the fundamental domain $\mathcal{R}$ in $\mathcal{H}$. 
Then \[\bigcup_{i=1}^{2{\mathbf{g}}}\clo(\mathcal{D}_{i}) \cap  ((\Lspace \cup \tilde \Sigma)\setminus P')\]
bounds a fundamental domain 
\[\clo(\mathcal{R}) \cap ((\Lspace \cup \tilde \Sigma)\setminus P').\] 
The boundary is homeomorphic to a $2$-sphere and, hence, 
the fundamental domain is homeomorphic to a compact $3$-cell.
Since this fundamental domain is compact, $N'$ is compact. 

Since we pasted disjoint disks on a cell,  $N'$ is homotopy equivalent to a bouquet of circles. 
Now, $N'$ has no fake-cell since $\tilde N'$ is a subset of $\Lspace$. 
It follows that $N'$ is homeomorphic to a compact handlebody of genus $\bg$ by Theorem 5.2 of \cite{Hempel04}.

Let $\hat P$ be the closure of $P'$ in $\tilde N$. 
We realize that 
$N'$ is a deformation retract of $N$ by collapsing $\hat P/\Gamma$, homeomorphic to a disjoint union of copies of 
$A^{2}\times [0, 1)$,  to its boundary in $N$ 
homeomorphic to a disjoint union of embedded images of $A^{2}$
for a compact annulus $A^2$ with boundary.  
This completes the proof of Theorem \ref{thm:main}.

\subsubsection{Proof of Corollary \ref{cor:main2}}
If $\mathcal{L}(\Gamma) \subset \SO(2,1)^o$, we are done
by Theorem \ref{thm:main}. 

Suppose not.
We have an index-two subgroup $\Gamma'$ of $\Gamma$ acting on 
$\Ss_+$ with $\mathcal{L}(\Gamma') \subset \SO(2,1)^o$.  
Then
$\Gamma'$ acts on 
$(\Lspace\cup \tilde \Sigma)\setminus P'$ 
where we construct $\tilde \Sigma$ and $P'$ as above for $\Gamma'$. 
There exists an element $\phi$ of $\Gamma-\Gamma'$ so that 
$\phi(\Ss_+) = \Ss_-$ and 
$\phi^2 \in \Gamma'$ and $\phi$ normalizes $\Gamma'$. 
Since $\phi$ acts as an orientation-preserving map of $\Ss$, 
and \[\mathcal{L}(\phi) \circ \mathcal{L}(\Gamma') \circ 
\mathcal{L}(\phi)^{-1} =  \mathcal{L}(\Gamma'),\] 
it follows that 
$\phi$ induces a diffeomorphism 
$\Ss_+/\Gamma'$ with $\Ss_-/\Gamma'$ preserving orientations. 
Since $\Ss_-$ is a Klein model also, we can define 
a limit set $\Lambda_{\Gamma', \Ss_{-}}$.
Hence, for the limit sets, we have 
\[\phi(\Lambda_{\Gamma', \Ss_{+}}) 
= \Lambda_{\Gamma', \Ss_{-}}, 
\hbox{ and } \phi(\partial \Ss_+ \setminus \Lambda_{\Gamma', \Ss_{+}}) =
\partial \Ss_-\setminus \Lambda_{\Gamma', \Ss_{-}}.\] 
Since each element of $\mathcal{L}(\Gamma')$ 
commutes with $\mathcal{A}$, we obtain 
\[\mathcal{A}(\Lambda_{\Gamma', \Ss_{+}}) = 
\Lambda_{\Gamma', \Ss_{-}} \hbox{ and }
\mathcal{A}(\partial \Ss_+\setminus \Lambda_{\Gamma', \Ss_{+}}) = 
\partial \Ss_-\setminus \Lambda_{\Gamma', \Ss_{-}}.\]
Let $\mathcal{I}$ denote the collection of open intervals of 
$\partial \Ss_+\setminus \Lambda_{\Gamma', \Ss_{+}}$. 
We define $\tilde \Sigma$ for $\Gamma'$ as in \eqref{eqn:tSigma}, 
\[\Ss_+ \cup \Ss_- \cup 
\bigcup_{a \in \mathcal{I}} \left(a\cup \mathcal{A}(a) \cup \bigcup_{x \in a} \zeta_x \right).\] 
Since $\phi$ is orientation-preserving, 
$\phi$ sends the disk $A_{a} = \bigcup_{x\in a} \zeta_{x}$, 
$a\in \mathcal{I}$,   to $A_{\mathcal{A}(\phi(a))}$. 
Since $a \mapsto \mathcal{A} \phi(a)$ 
gives us an automorphism of $\mathcal{I}$, 
$\phi$ acts on $\tilde \Sigma$. 

Given a component $P_1$ of $P'$, there is a parabolic primitive element $\gamma_1$ 
acting on it. Then $\gamma_2:= \phi \circ \gamma_1 \circ \phi^{-1}$ 
acts on $\phi(P_1)$. Since $\gamma_2 \in \Gamma'$ also, 
$\gamma_2$ acts on a component $P_2$ of $P'$. 
We denote $\gamma_2^\ast (P_1) = P_2$ where 
we may not yet have $\gamma(P_1) = P_2$. 

Let $\widetilde{\mathcal{P}}$ denote the set of parabolic fixed points of 
$\partial \Ss_+$. Then let a finite $\hat{\mathcal{P}}$ denote the collection of 
the $\Gamma'$-orbit classes of $\widetilde{\mathcal{P}}$.  
The above action of $\phi$ induces an automorphism of $\hat{\mathcal{P}}$. 

\begin{lemma} 
There is no fixed point in $\hat P$ under this action of $\phi$ on $\hat{\mathcal{P}}$. 
\end{lemma} 
\begin{proof} 
Suppose not. Then using orbit equivalence under $\Gamma$,  
there exists an isometry 
$\psi \in \Gamma\setminus \Gamma'$ so that 
$\mathcal{A}\circ \mathcal{L}(\psi)(q) = q$ for a parabolic fixed point $q$. 
$\mathcal{A} \circ \mathcal{L}(\psi)$ acts on 
$\Ss_+$ acting on a component ${\mathscr{H}}_i$ for some $i$. 
Since $\mathcal{A} \circ \mathcal{L}(\psi)$ acts as an orientation reversing isometry on $\Ss_+$, 
$\mathcal{A} \circ \mathcal{L}(\psi)$ acts on 
a complete geodesic $l_q$ ending at $q$. 
Since it {must fix} the point 
$\partial_h {\mathscr{H}}_i \cap l_q$, 
it fixes each point of $l_q$. 
Hence, $\mathcal{L}(\psi)$ acts as $-\Idd$ on a time-like vector subspace $P_{l_q}$ corresponding to $l_q$,
and is the identity on a space-like vector subspace. 
Since $\psi$ cannot have a fixed point on $\Lspace$, 
$\psi^2$ cannot be the identity on $\Lspace$, and 
it is a Lorentzian translation on a space-like 
geodesic $l$ orthogonal to $P_{l_q}$, 
and $\psi^2 \in \Gamma'$ since $[\Gamma:\Gamma']= 2$. 
However, $\Gamma'$ does not have a translation element 
as it is an affine deformation of $\mathcal{L}(\Gamma')$. 
\end{proof} 



{Since there is no} fixed point of the action, 
we divide the collection $\mathcal{\hat P}$ of components of $P'$ into 
equivalence classes of orbits under $\Gamma'$. This is a finite set
$\hat P_1, \dots, \hat P_{2m}$.
Now $\phi$ acts on this set. We may assume that 
$\phi$ sends $\hat P_i$ to $\hat P_{m+i}$. 

We replace each element of $\hat P_{m+i}$ with 
$\phi(P'')$ for the corresponding 
element $P''$ of $\hat P_i$ for $i= 1, \dots, m$. 
We obtain a new set $P'$. 
Here, for the parabolic element $\gamma'''$ 
corresponding to $\phi(P'')$, we have 
$\gamma''' = \phi \circ \gamma '' \circ \phi^{-1}$ 
for a parabolic element $\gamma''$ acting on $P''$. 
Since $\gamma'''$ is in {the} unique one-parameter subgroup $\gamma^{\prime \prime \prime t}, t\in \bR,$ of 
parabolic isometries, $\gamma^{\prime \prime \prime t}, t\in \bR,$ acts on 
$\phi(P'')$. Therefore, the boundary $\partial \phi(P'') \cap \Lspace$ 
is a parabolic ruled surface for $\gamma'''$ as defined by 
Definition \ref{defn:para}.

Obviously, $\Gamma$ acts on $P'$. 
Also, we may assume that 
elements of $P'$ are mutually disjoint:
{we} take a finite set of components of $P'$ that 
meets the fundamental domain $\mathcal{R}$. 
We can make these disjoint by taking them sufficiently 
far away. Proposition \ref{prop:Pdisj} shows that 
these are mutually disjoint.

Therefore, 
$N':= ((\Lspace \cup \tilde \Sigma)\setminus P')/\Gamma$ is compact and is homeomorphic 
to a handlebody of genus $\bg$ by Theorem 5.2 of \cite{Hempel04}
as in Section \ref{subsub:thmmain}. 
Since $\phi$ does not act on any component of $P'$, we can show that 
$N$ deformation retracts to $N'$ as {above.}



\appendix 

\gdef\thesection{\Alph{section}} 
\makeatletter
\renewcommand\@seccntformat[1]{Appendix \csname the#1\endcsname.\hspace{0.5em}}
\makeatother

\section{Parabolic ruled surfaces}\label{app:A}

We will be using the \hyperlink{term-pcs}{parabolic coordinate system} obtained in Section \ref{sub:parabolic}. 
These constructions are canonical except for the ambiguity in the $x$-coordinates up to translations.
(See Remark \ref{rem:canonical}.)

\subsection{Proper embedding of ruled surfaces}

Only prerequisites are Sections \ref{sec:prelim} and \ref{sub:parabolic}.
{Our purpose is to prove Theorem \ref{thm:ruled} using Lemma \ref{lem:ufix}, Proposition  \ref{prop:HR+}, and Lemma \ref{lem:orbitout}.}

\begin{lemma}\label{lem:ufix} 
	{Assume as in Theorem \ref{thm:ruled}.} 
	Every $g^t$-orbit in $\mathcal{H}_{\kappa_{1}, \kappa_{2}}$ 
	starts and ends at ${\zeta_{\llrrparen{ 1,0,0,0}}}$. 
\end{lemma}
\begin{proof} 
	Let $l$ be a segment so that 
	$\clo(l) \in \mathcal{H}_{s_0, \kappa_{1}, \kappa_{2}}$.
	An endpoint of $l_{\infty}$ must be $\llrrparen{ 1, 0, 0, 0}\,\, $ since $g^{t_{i}}(q) \ra \llrrparen{ 1, 0, 0, 0}$ 
	for each point $q\in l \cap P_{T}$. 
	Since $\clo(l)$ has a pair of antipodal points, the other endpoint of $\clo(l_{\infty})$ is $\llrrparen{ -1,0,0,0}$. 
	We compute the intersection of an arbitrary image of 
	$g^t(l)$ at the plane given by $x= 0$
	\begin{multline}\label{eqn:ufix} 
		\llrrparen{0,-\frac{c t \left(\mu  t^3+6 t y_{0}\right)}{6 a+3 c t^2}
			+\frac{\mu  t^2}{2}+ y_{0}, \mu  t-\frac{c \left(\mu  t^3+6 t y_{0}\right)}{6 a+3 c t^2}, 1} =\\
		\llrrparen{0,-\frac{ \mu  t^2 + 6 y_0}{ \frac{6a}{ct^2} + 3}
			+\frac{\mu  t^2}{2} + y_0, \mu  t-\frac{\mu  t + 6y_0/t}{\frac{6a}{ct^2} +3}, 1}
		\ra \llrrparen{ 0, 1, 0, 0}
		\in \SI^{3}
	\end{multline}
	as $t \ra \infty$ or $t \ra -\infty$.  See \cite{convergencesII}.
	Since this point is in $\clo(l_{\infty})$, we showed that 
	$l_{\infty} = \zeta_{\llrrparen{ 1, 0, 0, 0}}$.
	
\end{proof} 


\begin{figure}[h]

	\subfloat[]{\includegraphics[height=8.5cm, trim={1.0cm 0.3cm 1.0cm 1.0cm}]{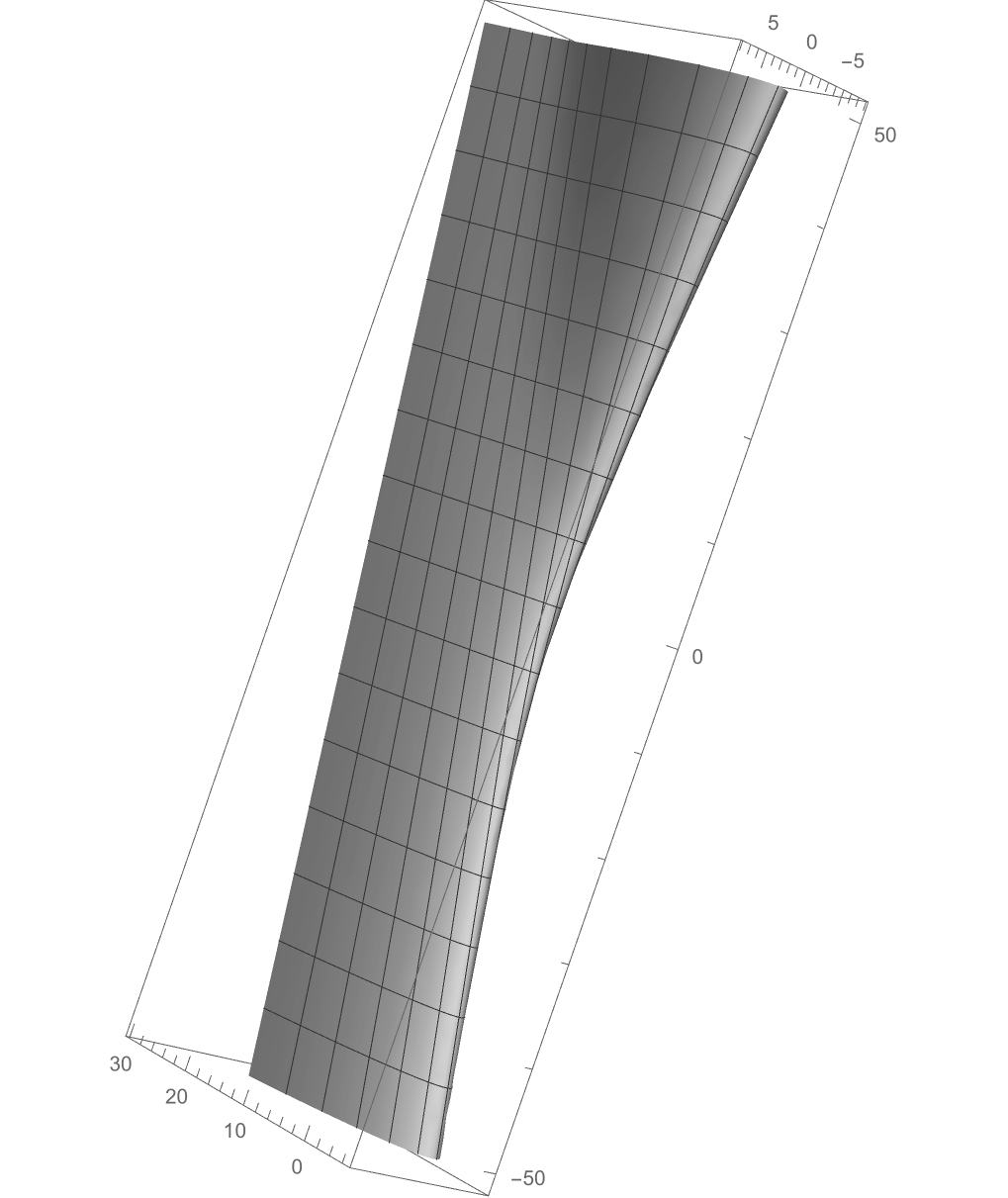}}
	\hfill
	\subfloat[]{\includegraphics[height=8.5cm, trim={1.0cm 0.3cm 1.0cm 1.0cm}]{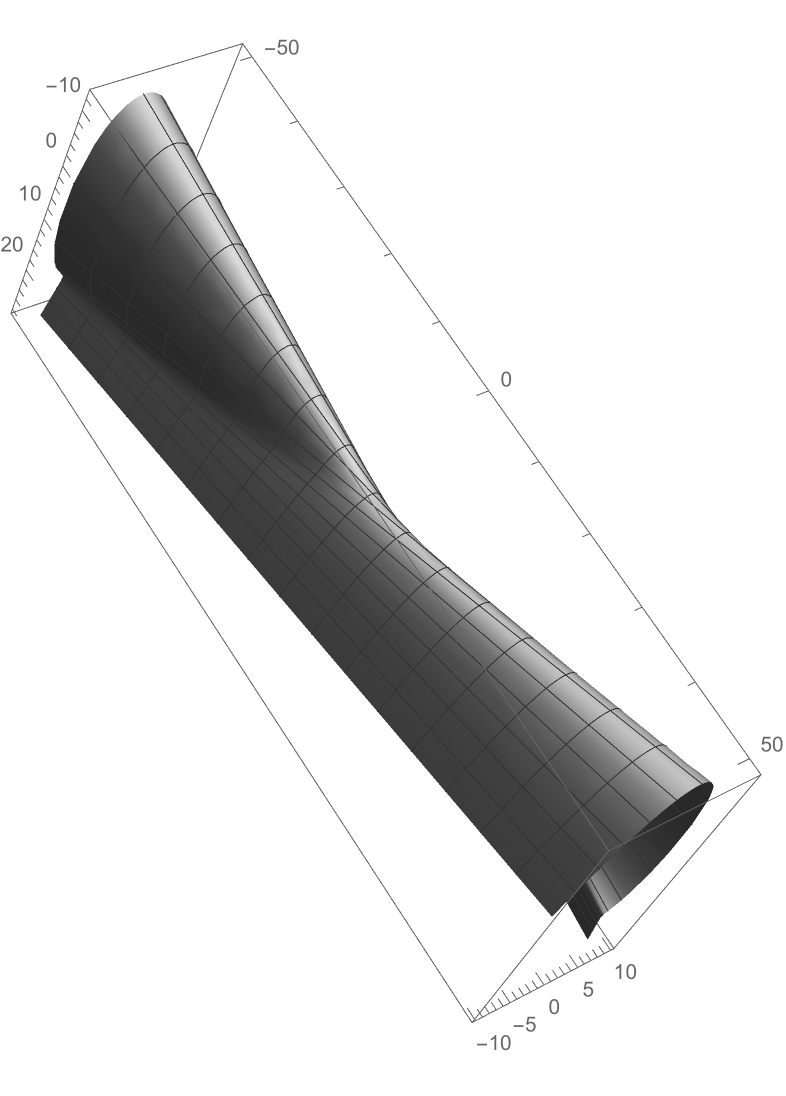}}

	\caption{Two parabolic ruled surfaces. See \cite{convergencesII}.}   
	\label{fig:Ruled2}
\end{figure}

\begin{proposition} \label{prop:HR+} 
	{Assume as in Theorem \ref{thm:ruled}.}
	Choose $\kappa_{1}$ and $\kappa_{2}$ satisfying
	$0 < \kappa_{1}\leq \kappa_{2} < 1$.
	The closure of $\mathcal{H}_{s_0, \kappa_{1}, \kappa_{2}}$ under $\bdd_H$ 
	is a compact set
	$\mathcal{H}_{s_0, \kappa_{1}, \kappa_{2}} \cup 
	\{\zeta_{\llrrparen{ 1,0,0,0}}\}$.
\end{proposition} 
\begin{proof} 
	{The space of open geodesic segments of $\bdd$-length $\pi$ in the $3$-hemisphere $\mathcal{H}$ forms a compact metric space under the Hausdorff metric $\bdd_H$.} 
	We show this by showing that every sequence of elements of 
	$\mathcal{H}_{s_0, \kappa_1, \kappa_2}$ has an accumulation point in $\mathcal{H}_{s_0, \kappa_1, \kappa_2}$ or 
	accumulates to $\clo(\zeta_{\llrrparen{ 1,0,0,0}})$. 
	
	Given a sequence of segments $\{u_i\}$ in $\mathcal{H}_{s_0, \kappa_1, \kappa_2}$,
	$u_i = g^{t_i}(l_i),$ where $l_i \cap \Lspace$ is given by 
	\begin{multline} 
		l_i(s) = (sa_i, y_{0, i}, sc_i)
		\hbox{ for } y_{0, i}\geq s_0, a_{i}, c_{i}> 0, \\ \frac{\kappa_{1}a_i}{c_i} \leq \frac{y_{0, i}}{\mu} \leq \frac{\kappa_{2}a_{i}}{c_{i}}, 
		a_i^2 + c_i^2 = 1.
	\end{multline}
	The boundedness of one of $y_{0, i}$ or $\frac{a_i}{c_i}$ implies that of 
	the other. 
	If $y_{0, i}$ or $\frac{a_i}{c_i}$ is
	bounded above, then $l_i$ geometrically converges to 
	an element of $\mathcal{H}_{s_0, \kappa_1, \kappa_2}$ up to a choice of a subsequence.
	If $t_i \ra \pm \infty$, 
	then $u_i \ra {\zeta_{\llrrparen{ 1,0,0,0}}}$ 
	since the estimates in \eqref{eqn:ufix} in 
	the proof of Lemma \ref{lem:ufix} hold in this case.  
	If $t_i$ is bounded, then $u_i \ra u_0\in \mathcal{H}_{s_0, \kappa_1, \kappa_2}$. 
	
	Hence, we are left with the case where 
	\[ y_{0, i} \ra \infty, \frac{a_i}{c_i} \ra \infty,
	\hbox{ and } t_i \ra \pm\infty.\]
	
	We will show that {$u_i \ra \zeta_{\llrrparen{ 1,0,0,0}}$}:
  { Suppose not. Then 
	$u_i$ converges to a line $u_\infty$ passing $\Lspace$ under the metric $\bdd_H$.} 
	Then $u_\infty$ has the direction $(1, 0,0)$ since
	$\llrrparen{\mathcal{L}(\Phi_t)(\vv)} \ra 
	\llrrparen{1, 0, 0}$ for a generic vector $\vv$. 
	
	{By applying an element of $g^t$ to $u_\infty$ and the sequence $u_i$,} 
	we may assume that $u_\infty\cap \Lspace$ is given as the line 
	\[x = s, y = C, z=0, s\in \bR:\]
	Since $u_i$ geometrically converges to $u_\infty$, 
	$u_i$ intersected with $x = 0$ is near $(0, C, 0)$. 
	By changing $u_i$ by a bounded $g^{s_i}$ with $s_{i} \ra 0$, we may assume without loss of generality that 
	$u_i$ passes $(0, C_i, 0)$ while we still have $u_i \ra u$ under $\bdd_H$. 
	Here $C_{i} \ra C$. 
	
	By our construction, $u_i$ is contained in a hyperplane $P_i$ tangent 
	to {a} parabolic cylinder $S_i$ given by the equation $2\mu y= z^2 + 2 \mu C_{1,i}$ { for some $C_{1, i}\in \bR$.}
	The line $u_i$ meets $S_i$ at {the} unique point $(x^*_i, y^*_i, z^*_i)$. 
	Project $P_{i}$ and $S_{i}$ to the $yz$-plane. Then the image of $P_{i}$ passes $(C_{i}, 0)$ and 
	tangent to the parabola $2\mu y = z^{2} + 2\mu C_{1, i}$.
	We compute by elementary geometry 
	\[ z^*_i = \pm \sqrt{2\mu C_{1, i} - 2\mu C_i}
	\hbox{ and } y^*_i = 2C_{1, i} -  C_i.\] 
	
	Now, we wish to compute {$t_i$ so that $g^{t_i}(l)=u_{i}$} as in \eqref{eqn:ls}    
	where $l(s)$ passes $(0, C_{1, i}, 0)$. 
	We compute $t$ satisfying 
	\[\Phi_t(0, C_{1, i}, 0) = \left(tC_{1, i} + \frac{\mu t^3}{3}, C_{1, i} + \frac{\mu t^2}{2}, \mu t\right)
	= (x^*_i, y^*_i, z^*_i) \]
	recalling \eqref{eqn:Phit}.
	We let $t_i$ denote the answer
	\begin{equation} \label{eqn:tiyi} 
		y^*_i  = 2C_{1, i}  - C_i = C_{1, i} + \frac{\mu t_i^2}{2} \hbox{ and } t_i = \pm \sqrt{\frac{2(C_{1, i} - C_i)}{\mu}}.
	\end{equation}  
	The vector $\alpha_i$ tangent to $u_i$ is given by 
	\begin{align*} 
		\left( t_i C_{1, i} + \frac{\mu t_i^3}{6}, C_{1, i} + \frac{\mu t_i^2}{2}, \mu t_i \right) - (0, C_i, 0)\\ 
		= \left( t_i C_{1, i} + \frac{\mu t_i^3}{6}, C_{1, i} - C_i + \frac{\mu t_i^2}{2}, \mu t_i \right).
	\end{align*} 
	Since the sequence of the directions of the vectors converges to 
	$(1,0,0)$ by our assumption on $u_{i}$,  we obtain $t_{i} \ra \pm\infty$. 
	
	Recall
	\[ 
	\mathcal{L}(\Phi_{t_i})^{-1} = 
	\begin{pmatrix} 
		1 & -t_i & \frac{t_i^2}{2} \\ 
		0 & 1    & - t_i \\ 
		0  & 0  & 1 
	\end{pmatrix}.
	\]
	We compute $\mathcal{L}(\Phi_{t_i}^{-1})(\alpha_i)$ to be 
	\begin{multline*} 
		\left(t_i C_{1, i} + \frac{\mu t_i^3}{6} - t_i(C_{1, i} - C_i) - \frac{\mu t_i^3}{2}+ \frac{\mu t_i^3}{2} , 
		C_{1, i} - C_i + \frac{\mu t_i^2}{2} - \mu t_i^2, \mu t_i\right) \\ 
		=\left(\frac{\mu t_i^3}{6} + t_i C_i, 0, \mu t_i\right).
	\end{multline*} 
	
	Recall the condition  \eqref{eqn:ls} to {$g^{-t_i}(u_i) = l$}, which yields 
	$  C_{1, i}/\mu \leq \kappa_{2} (t_i^2/6 + C_i /\mu )$,
	and we obtain 
	\[C_{1, i} \leq \kappa_{2}\frac{1}{3} {|C_{1,i}- C_i|} + \kappa_{2} C_i  \hbox{ and } \kappa_{2} < 1 \]
	by \eqref{eqn:tiyi}. 
	Since $t_i \ra \pm \infty$, we obtain $C_{1, i} \ra +\infty$ and $C_{i} \ra C$. 
	This contradicts the above inequality. 
	
	We conclude that $u_i$ can converge only to points of 
	$\mathcal{H}_{s_0, \kappa_1, \kappa_2}$ or 
	${\zeta_{\llrrparen{ 1,0,0,0}}}$. 
	This gives us a sequential convergence property. 
	The closure of $\mathcal{H}_{s_0, \kappa_1, \kappa_2}$ is a compact metric space 
	$\mathcal{H}_{s_0, \kappa_1, \kappa_2} \cup \{{\zeta_{\llrrparen{ 1,0,0,0}}} \}$. 
	
\end{proof} 


\begin{lemma}\label{lem:orbitout} 
	Let $M$ be a compact metric space. {Suppose that} there exists a one-dimensional 
	flow $\phi_t: M \ra M$, $t \in \bR$,	with a fixed point $p$. 
	Suppose that the orbit of every point starts and 
	ends at $p$, and the orbit space $(M\setminus \{p\})/\sim$ is not compact. 
	Then every $\epsilon$-ball of $p$ contains an orbit starting and 
	ending at $p$. 
\end{lemma}
\begin{proof} 
	Choose a compact set $K = M\setminus B_\epsilon(p)$ for an open $\epsilon$-ball
	of $p$.  Since \[(\bigcup_{t\in \bR} \phi_t(K))/\sim\, \, \,  = K/\sim\] is compact, 
	and $(M\setminus \{p\})/\sim$ is not compact, 
	it follows that 	
	\[M\setminus \bigcup_{t\in \bR} \phi_t(K) = \bigcap_{t\in \bR} \phi_{t}(B_{\epsilon}(p))
	\subset B_\epsilon(p)\]
	is not empty.
	Then a point here gives us an example of the closed orbit.
\end{proof}

\begin{proof}[Proof of Theorem \ref{thm:ruled}]

	We will first show that $\Psi: \bR^{2} \ra \Lspace$ is a proper injective map. 
	
	Since $g^{t}$ acts on  $P_{T'}$ for each $T'$, 
	we have a self-intersection of $\Psi$ if 
	\[g^{t}(l(s) \cap P_{T'}) = l(s') \cap P_{T'} \hbox{ for some } t> 0, s, s'\in \bR \hbox{ and } T' \in \bR.\] 
	{The following hold:}
	\begin{itemize} 
		\item $l(s) \cap P_{T'}$ is a pair of points provided $T' > - 2\mu y_{0}$, or 
		\item $T'= -2\mu y_{0}$ and $l(s) \cap P_{T'}$  is $(0, y_{0}, 0)$, or else 
		\item $l(s) \cap P_{T'}$  is empty for $T' < -2\mu y_{0}$. 
	\end{itemize} 
	Thus, only in the first case, we can have a self-intersection of the image of $\Psi$ under the quotient space $\Lspace/\langle g \rangle$. 
	Now $l(s) \cap P_{T'}$ can be computed as follows\,: 
	\[(sc)^{2} - 2\mu y_{0} = T' \hbox{ and } s_{0} = \sqrt{T'+ 2\mu y_{0}}/c\] 
	and the points are $(\pm s_{0}a, y_{0}, \pm s_{0}c)$. 
	And we obtain 
	\[F_{3}(\pm s_{0}a, y_{0}, \pm s_{0}c)  = \pm (s_{0}^{3}c^{3} - 3 \mu y_{0} s_{0} c + 3 \mu^{2} s_{0}a).\]
	These are distinct unless the value is $0$. 
	Since $F_{3}$ is invariant under $g$, it follows that if $F_{3}$-values of two points are distinct, 
	then they cannot be in the same orbit of $\langle g \rangle$. 
	If $F_{3} = 0$, we must have 
	\[\frac{a}{c} = -\frac{T'}{3\mu^{2}} + \frac{ y_{0}}{3 \mu }.\]
	Since $-T' < 2\mu y_{0}$, we obtain 
	\begin{equation}\label{eqn:acy0} 
	\frac{a}{c} <  \frac{2\mu y_{0}}{3\mu^{2}} + \frac{ y_{0}}{3 \mu}= \frac{y_{0}}{\mu}.
	\end{equation} 
	Thus, if we choose $ y_{0}<  \mu \frac{a}{c}$, the self-intersection of $\Psi$ never happens.
	For example, choosing $y_{0}$ sufficiently small or choosing $\frac{a}{c}$ sufficiently large would satisfy the condition. 
	This proves the injectivity of $\Psi$. 
	
	By \eqref{eqn:Phicoor}, 
	$g^{t}$ acts properly on each parabolic cylinder $P_{T}$ since $F_{1}$ and $F_{2}$ are invariants of 
	the vector field on $\phi$, and each intersection of $F_{1} \cap F_{2}$ is a complete flow line. 
	
	We now prove the properness of $\Psi$. 
	Suppose that there is a compact set $K \subset \Lspace$ and 
	$g^{t_{i}}(l) \cap K$ is not empty for a sequence $\{t_{i}\}$ of real numbers such that $t_{i} \ra  \infty$. 
	(The case when $t_{i} \ra -\infty$ is entirely similar.) 
	However, $K$ is in the region $B$ in $\Lspace$ bounded by 
	two parabolic cylinders $P_{T_{1}}$ and $P_{T_{2}}$ for some pair $T_{1}$ and $T_{2}$. 
	Then $l$ meets $P_{T_{j}}, j=1$ at most two points. 
	If $l$ does not meet {$B$, then 
	$g^t(l)$, $t\in \bR$ is disjoint from $K$ since the region bounded by $P_{T_{j}}$ is $g^t$-invariant.} 
	If $l \cap B \ne \emp$, 
	\[g^{t_{i}}(l\cap B) \ra  \{\llrrparen{ 1, 0, 0, 0}\}\,\,  \hbox{ or } \{\llrrparen{ -1, 0, 0, 0}\}\,\, \hbox{ as } t_{i}  \ra \pm \infty\] 
	by convexity since the endpoints of $l \cap B$ do this.
	This proves the properness of $\Psi: \bR^{2}\ra \Lspace$ and 
	that $g^{t_{i}}(l)$ can have limit points only in $\Ss$. 
	
	
	
{The first  item is proved by Lemma \ref{lem:ufix}. }
	
	Choose $\kappa_{1}$ and $\kappa_{2}$ satisfying
	$0 < \kappa_{1} \leq \kappa_{2} < 1$.
	There is a continuous map 
	$\iota_R: \mathcal{H}_{s_0, \kappa_1, \kappa_2} \ra \Ss_+$ by taking 
	the endpoints in $\Ss_+$. The image is a horodisk $\mathcal{E}$. 
	Since $\mathcal{E}/\sim$ is not compact, 
	$\mathcal{H}_{s_0, \kappa_1, \kappa_2}/\sim$
	is not compact under the orbit equivalence relation under $g^t, t\in \bR$ 
	
	By Proposition \ref{prop:HR+}, $\mathcal{H}_{s_0, \kappa_1, \kappa_2}\cup \{ {\zeta_{\llrrparen{ 1,0,0,0}}}\}$ is compact. 
	In any $\eps$-$\bdd_H$-neighborhood $\hat N$, $\eps > 0$, of 
	$\clo(\zeta_{\llrrparen{ 1,0,0,0}})$
	in $\mathcal{H}_{s_0, \kappa_1, \kappa_2}$, we can find a $g^{t}$-orbit in $\hat N$ by  Lemma \ref{lem:orbitout}.
	Take any neighborhood $N$ of $\clo(\zeta_{\llrrparen{ 1,0,0,0}})$
	in $\SI^3$. Since we are using the Hausdorff metric, 
	we can find an $\eps$-$\bdd_H$-neighborhood $\hat N$ in 
	$\mathcal{H}_{s_0, \kappa_1, \kappa_2}\cup 
	\{ {\zeta_{\llrrparen{ 1,0,0,0}}}\}$ 
	so that 
	any segment in $\hat N$ is a segment in $N$.
	Then the $g^{t}$-orbit as above will give us the desired ruled surface in $N$. 
	This proves the {second} item. 
	
{The first and second items imply the fact on the boundary of $S_{f, r}$.
	Clearly,$ S_{f, r}$ bounds a domain in $\Lspace$ with boundary $\clo(S_{f, r})$. 
This domain is homeomorphic to a $3$-cell by Lemma 1.12 of \cite{Marden74}.
Also, $g$ sends the disk leaves of the foliation $\mathcal{D}_{f, r_0}$ of the domain to a disjoint disk leaf in 
Theorem \ref{thm:Sr}. Hence, the quotient space is homeomorphic to a solid torus. }
\end{proof}

\subsection{Two transversal foliations}

\begin{proof}[Proof  of Theorem \ref{thm:Sr}]
The fact that $S_{f, r}$ is a properly embedded surface is proved in Theorem \ref{thm:ruled}. 
We defined
$l_{f, r}(s) = (sr, f(\rho), s\sqrt{1-r^2})$. We define
\[l_f: [r_0,1)\times \bR  \ra \Lspace \hbox{ given by } l_f(r, s) =(sr, f(\rho), s\sqrt{1-r^2}).\] 
Let $u_{l_{f, r}}$ denote the vector field $(r, 0, \sqrt{1-r^2})$ tangent to $l_{f,r}(s)$. Also, 
the vector field $\phi$ generating $g^{t}$ is given by $(y, z, \mu)$. 
\[\frac{\partial l_f}{\partial r} = Y_f=\left(s, f'(\rho), \frac{-sr}{\sqrt{1-r^2}}\right)\] is 
tangent to $D_{f, r_0, 0}$ obtained by taking a tangent vector along the direction
of $\frac{\partial }{\partial r}$. 
{A {\em triple product of three vectors} is the volume of the span of three vectors in $\Lspace$.} 
We compute the triple product on the line $l_{f, r}$
\begin{equation}\label{eqn:triple} 
(u_{l_{f,r}}, Y_f, \phi) = \sqrt{1-r^{2}} \left( \frac{\mu r}{\sqrt{1-r^{2}}} - f(\rho) \right) f'(\rho) + s^2 > 0,
\end{equation}
which follows by our condition on $f$ and $r$. 
It follows that
$u_{l_{f,r}}, Y_f, \phi$ form always an independent frame in the standard orientation on $l_{f,r}$, 
and so are their images under $g^{t}$ since $g^t$ is volume-preserving.
Thus, \[Dg^t(u_{l_{f,r}}), Dg^t(Y_f), Dg^t(\phi)\] form an independent frame
at each point of $S_{f, r}$. 

We claim that $S_{f, r}$ is disjoint from $S_{f, r'}$ for $r_{0}\leq r < r' < 1$: 
By \eqref{eqn:triple}, $Y_f$ is transversal to $S_{f, r}$ on 
$l_{f, r}$. 
We define the vector field $Y_f$ on $S_{f, r}$ 
so that 
\[Y_f(g^t(sr, f(\rho), s\sqrt{1-r^2})) = 
Dg^t(Y_f(sr, f(\rho), s\sqrt{1-r^2})).\] 
The extended $Y_f$ is transversal to $S_{f, r}$ since the triple product is 
invariant under the Lorentzian isometries. 
Define $\Xi_{f}(r, t, s) = g^t(l_f(r, s)))$, {which gives us a parametrization of $S_{f, r}$.}
{We obtain the partial derivative with respect to $r$ by chain-rules:} 
\[  \frac{\partial{\Xi_{f}(r, t, s)}}{\partial r} = Dg^t(Y_f(l_f(r, s))) 
= Y_f(\Xi_f(r, t, s)). \]

Solving the {following ordinary differential equation with respect to the variable $r$}
\[\frac{\partial \Xi_f(r, t, s)}{\partial r} = Y_f(\Xi_f(r, t, s))\]
gives us a flow $\Xi_f(r, t, s)$ for $r$ in some interval with fixed $t, s$.
Using the {quasi-linear Cauchy theorem (Theorem 9.52 of \cite{Lee})} and the transversality, we obtain the disjointness. 

Also,  {for each point $x$ of $R_{f, r_{0}}$,} there is a leaf $S_{f, r'}$ containing it:  
Let $x_{i}$ be a sequence converging to $x$ and $x_{i} \in S_{f, r_{i}}$, $r_{i} > r_{0}$. 
Then let $L_{i}$ be the line in $S_{f, r_{i}}$ containing $x_{i}$. 
Since we showed that $\mathcal{H}_{s_0, \kappa_1, \kappa_2} \cup \clo(\zeta_{\llrrparen{ 1,0,0,0}})$ is compact by Proposition \ref{prop:HR+},
$\clo(L_{i})$ geometrically converges to an element of $\mathcal{H}_{s_0, \kappa_1, \kappa_2}$ or 
to $\clo(\zeta_{\llrrparen{ 1,0,0,0}})$ by choosing a subsequence if necessary. 
Proposition \ref{prop:HR+} shows that $L_i = g^{t_i}(l_f(r_i))$ for bounded $t_i$
in the first case. 
Hence, $x$ is in $S_{f, \lim_{i} r_{i}}$. 
In the other case, $x_{i}$ does not have $x$ as a limit.   
This proves the closedness of the foliated subset in $R_{f, r_{0}}$. 


Using the flows, we can prove the openness of 
the set $\bigcup_{r_{0} \leq r < 1} S_{f, r'}$.   
Hence, $R_{f, r_{0}}$ is foliated by leaves $S_{f, r}$, $r \geq r_{0}$.



Since each line in $D_{f, r_0, 0}$ lies on {a} different plane {given by equations of the form} 
$y = c$, $D_{f, r_0, 0}$ is an embedded surface, and so are 
$D_{f, r_0, t}$. Proposition \ref{prop:HR+} implies that $D_{f, r_0, 0}$ is properly embedded since 
$\clo(l_{f, r_{i}})$ geometrically converges to $\clo(\zeta_{\llrrparen{ 1,0,0,0}})$ as $r_{i} \ra 1$. 
Hence, $D_{f, r_0, t}$ is properly embedded for all $t$. 

Since $g^{t_{0}}$ is generated by a vector field $\phi$ transversal to $D_{f, r_0, t}$ for every $t$ by the above paragraph, 
the images under the flows of $D_{f, r_0, t}$ are disjoint from $D_{f, r_0, t}$.  
Also, $g^{t_{0}}(D_{f, r_0, t}) = D_{f, r_0, t+t_{0}}$ follows by our above definition of $D_{f, r_0, t}$. 

Also, $R_{f, r_{0}}$ is foliated by leaves of the form $D_{f, r_0, t}$ as follows: 
$\bigcup_{t\in \bR} D_{f, r_{0}, t}$ is open since we can use the flow generated 
by $g^t$. 
The closedness follows by Proposition \ref{prop:HR+} again as above. 

%
Now, we have a foliation by leaves of the form $S_{f, r}$ for $r \in [r_{0}, 1)$. 
Then $D_{f, r_0, t} \cap S_{f, r}$ contains a geodesic given by $g^{t}(l(s)), s \in \bR$. 
At $t=0$, the tangent space of $D_{f, r_0, 0}$ is generated by $u_{l}, Y_f$, and that of $S_{f, r}$ is generated by $u_{l}, \phi$. The independence above implies the transversality
of $D_{f, r_0, 0}$ and $S_{f, r}$. Thus, the transversality of $D_{f, r_0, t}$ and $S_{f, r}$ follows. 


\end{proof} 

%

\begin{remark} 
There seems to be a vast literature on ruled surfaces on which a one-parameter Lorentzian isometry group acts
but there seems to be no article on the topological properties. 
See Dillen-K\"uhlen \cite{DK99} for a survey of geometric aspects. 
\end{remark}

\section{The flat $\bR^{2, 1}$-bundle valued $1$-forms on a cusp neighborhood} \label{app:1form} 

Only prerequisites are Sections \ref{sec:prelim} and \ref{sub:parabolic} and the notation in Section \ref{sec:orbit}{,} in particular Definition \ref{defn:compt}.

\subsection{Replacing forms by standard cusp $1$-forms in the cusp neighborhoods}  \label{sub:replace} 
Suppose that $\Gamma$ is a discrete Lorentzian isometry group so that 
$\Gamma$ is a Fuchsian group acting on $\Ss_{+}$ with 
a parabolic element $g$ fixing $p \in \partial \Ss_{+}$.
Let $\Sf:=\Ss_{+}/\Gamma$ be a complete hyperbolic surface with  a cusp neighborhood $E$. 
$E$ is covered by a horodisk $P\subset \Ss_{+}$ with $p \in \Bd_{\Ss} P$. 
Then $P/\langle g \rangle$ is isometric to $E$. 

We recall the vector bundle $\mathscr{V}$ given as the quotient of 
$\widetilde{\mathscr{V}} = \Ss_+ \times \bR^{2,1}$ with 
action given by 
\[ \gamma(x,\vv) = (\gamma(x), \mathcal{L}(\gamma)(\vv))
\hbox{ for } \gamma \in \Gamma, \vv \in \bR^{2,1}.\]  
Recall \eqref{eqn:Action} that 
for $\widetilde{\mathscr{V}}$-valued $1$-forms on $\Ss_+$, 
the action is given by 
\[ \gamma^*(\vv \otimes dx) = \mathcal{L}(\gamma)^{-1}(\vv) \otimes dx \circ \gamma.     \]


\begin{proposition}\label{prop:replace} 
Let $\Sf$, $\Gamma$, $P$, $E$, and $\gamma$ be as above in Section \ref{sub:replace}. 
Let $\eta$ be a closed $\mathscr{V}$-valued $1$-form representing a class in $H^{1}(\Sf, \mathscr{V})$. 
Let $\zeta$ be a closed $\mathscr{V}$-valued $1$-form in $E$
so that $\zeta $ is cohomologous to $\eta|E$ in $H^{1}(E, \mathscr{V})$. 
Then we can find a closed $\mathscr{V}$-valued $1$-form $\eta'$ on $\Sf$ cohomologous to 
$\eta$ and a cusp neighborhood $E' \subset E$ 
so that \newch{$\eta'|E' = \zeta|E'$}.
\end{proposition}
\begin{proof} 
Let $E' \subset E$ be a smaller cusp neighborhood so that $\clo(E') \subset E$. 
Consider $\eta - \zeta$ on $E'$. Then $\eta - \zeta = d f$ for a section $f: E' \ra \mathscr{V}$. 
We can extend $f$ to a smooth section $f: \Sf \ra \mathscr{V}$ by a partition of unity 
so that $f = 0$ on $\Sf\setminus E$. 
Then define $\eta' = \eta$ on $\Sf \setminus E$ 
and $\eta' = \zeta$ on $E'$ and $\eta' = \eta - df$ on $E \setminus E'$. 
\end{proof} 

\begin{proposition}\label{prop:cohomology} 
$H^{1}(E, \mathscr{V}) = \bR$. 
\end{proposition}
\begin{proof} 
Recall that $E$ is homotopy equivalent to $\SI^{1}$. Thus, $\pi_{1}(E)$ is an infinite cyclic group.
$\mathscr{V}|E$ is $P\times \bR^{2,1}/\langle g \rangle$.
Recall that $\mathcal{L}(g) = \Idd + N(g) + N(g)^{2}/2$ 
for a nilpotent matrix $N(g)$ of rank 2 from \eqref{eqn:Phit}.
We conclude using the knowledge of Section 3 of \cite{GLM09}:
\begin{align}\label{eqn:v}
Z^{1}(\langle g \rangle, \mathscr{V})  &=  \{  \vv: \bZ \ra \bR^{2, 1}| \mathcal{L}(g^{i}) \vv(g^{j}) - \vv(g^{i+j}) =0 \, \forall i, j \in \bZ \} \cong \bR^{2, 1}, \nonumber \\
B^{1}(\langle g \rangle, \mathscr{V})  &=  \{ \vv: \bZ \ra \bR^{2, 1}| \vv(g) = \mathcal{L}(g) \vv_{0} - \vv_{0}, \vv_{0} \in \bR^{2, 1} \},  \nonumber \\
& = N(g)\left(\left(\Idd+ \frac{1}{2}N(g)\right)(\bR^{2, 1})\right) = N(g)(\bR^{2, 1}), \nonumber\\
 H^{1}(\langle g \rangle, \mathscr{V})  &=  \bR^{2, 1}/N(g)(\bR^{2, 1}) \cong \bR. 
\end{align}
The second to last equation follows since $I+N(g)/2$ is invertible. 
The last follows since $N(g)$ is of rank two by Section \ref{sub:parabolic}.  
\end{proof}

\subsection{The integral of the standard cusp $1$-form} \label{appsub:integral}

We will use the notation of Section \ref{sub:tranv}. 
Let $P$ be the standard horodisk in $\Ss_{+}$ with $\vp$ as 
a null vector in the direction of 
$\clo(P) \cap \partial \Ss_{+}$. 
\oldch{The standard cusp $1$-form for $P$ is given 
where $P$ is given by $y > 1$ in the upper half-space model $\Us^2$ of the hyperbolic plane}. 
A geodesic in $P$ is given by equation $(x\pm R)^2 + y^{2} = R^{2}$ in $\Us^2$ and  parameterized by $\zeta(\theta) := (R \cos \theta \mp R  , R\sin \theta)$. 
The starting point and the endpoint are given by
$R \sin \theta  = 1$. Thus, the beginning $\theta_{0}$ and ending $\theta_{1} = \pi-\theta_{0}$ is one of the values of 
$\sin^{-1}(\frac{1}{R})$.

\oldch{We assume that a complete geodesic $l$ passes a cusp region with 
the cusp point $p= \llrrparen{\vp}$ and the standard cusp $1$-form.} 
We assume that by choice of the coordinates of $\Us^{2}$, 
$p = \infty$ and the geodesic starts at $(0, 0)$ 
and ends at $(2R, 0)$ or at $(-2R, 0)$.
We say that $l$ and any horizontal translation of $l$ in the upper half model
have {\em radius} $R$. 


There is {an isometry} $H: \Us^2 \ra \Ss_+$ 
{ to } the Klein model $\Ss_+$:  
\[ H(x, y) := \llrrparen{\overrightarrow{H(x, y)}} \hbox{ where }
\overrightarrow{H(x, y)}:= 
\left(\frac{2x}{x^{2}+ y^{2} +1 }, 1 - \frac{2}{x^{2} + y^{2} + 1}, 1\right),
 \]  
 (See {Theorem 7.1 of Hongchan Kim}  \cite{HKim06}.)
This extends to the boundary $y= 0$ 
{and induces a homeomorphism from $\Us^2 \cup \{\infty\}$ to the closure of the unit disk
where $\infty$ goes to $\llrrparen{\vj + \vk}$.} 

The standard horodisk in $\Us^{2}$ is given by $y > 1$. The image of this under $H$ is the \hyperlink{term-shd}{standard horodisk} $Q$ of the Klein model.  
The standard horodisk has the point $\llrrparen{\vj + \vk}$ of $\Ss_{+}$ in the boundary
and $\partial_h Q \ni \llrrparen{\vk}$.

This makes things simpler.
\begin{lemma} \label{lem:estimation} 
	Let $g$, $l$, and $\llrrV{\cdot}_{E}$ be as above. 
Let $D$ be the standard horodisk. 
\purple{Set 
	\[\vv(x) =\left( x, -\frac{ x^2}{2\sqrt{2}}+\frac{1}{\sqrt{2}} ,
\frac{x^2}{2}+\frac{1}{\sqrt{2}} \right) \hbox{ for } x \in \bR. \]
}
\purple{
Let $l$ be a complete geodesic passing $D$ of radius $R$ and starts at $H(0,0)$ and ends at $H(\pm 2R,0)$. Assume $R>1$. }
\purple{
Suppose that $l$ corresponds in $\Ss_+$
to the geodesic passing  a point of $\partial_h D$ in the direction of a unit vector $\bu$ away from $H(0, 0) \in {\partial \Us^2}$. 
Then the following hold\/{\rm :}}
\begin{itemize}
\item \purple{for any point $z$ on $l$ with coordinate $x$ in the upper half-space model,}
\newch{
\begin{equation} 
\llrrV{\Pi_{\tilde{\mathbf{V}}_{0}(z, \bu)}(\mu \vv(x))}_E=
\left| {\mu}\left(  x-\frac{\pm \sqrt{2}}{R}\right) \right|
 \end{equation} 
}
\purple{for the cusp coefficient $\mu$, and }
\item 
\newch{
\begin{multline} 
\llrrV{\Pi_{\tilde{\mathbf{V}}_{-}(z, \bu)}(\mu \vv)}_{E} = 
\left| \frac{\mu  \left(4 R^2+1\right)}{4\sqrt{2}R^2} \right|
\hbox{ and } \\
 \llrrV{\Pi_{\tilde{\mathbf{V}}_{+}(z, \bu)}(\mu\vv)}_{E}  =
\left|{\mu} \left(\frac{1 }{4\sqrt{2}R^2} -
\frac{\pm x}{2 R}+\frac{x^2}{2\sqrt{2}} \right) \right|
\end{multline}
}
\end{itemize}
\end{lemma} 
\begin{proof} 
	\purple{
These are simple computations using $H$, and the used frames there
form uniformly bounded matrices in $\GL(3, \bR)$. Hence, 
the estimations are uniformly compatible with the standard Euclidean metric results. 
 (See \cite{NewNormalComp} and \cite{NewNormalComp2}.) 
}
\end{proof}



Let $\zeta= l\cap D$ be a geodesic segment with both endpoints in $\partial_h D$.  
	Suppose that $l$ is in the form of Lemma \ref{lem:estimation}
	parameterized by the angle $\theta$ from the center of the semicircle in
	the upper-half-space model containing $l$. Also, recall
	the geodesic flow $\Psi_t$ acting on $\bR^{2,1}\times \Uu \Ss_+$ 
	from Section \ref{sub:propaffine}. We reparametrize $l$ by 
	$\Psi(z, \theta)$ for $z$ the beginning point of 
	$\zeta$ and $\theta \in (0,\pi)$ with $z = \Psi(z, \theta_0)$. 
	Let $\eta$ denote the standard $1$-form defined on $D$. 
We define 
\[ \vb(\zeta):= \,\int_{\theta_{0}}^{\pi-\theta_{0}} 
\bbD\Psi(z, \theta - \theta_0)^{-1} \left( \eta\left(\frac{d\Psi(z, \theta)}{d\theta}\right) \right) d\theta,\]
where $\theta_0$ and $\pi-\theta_0$ are the start and the end angles of 
the semicircle $l$ parameterized by angles.

We recall $\vb_{\pm}(\zeta)$ from \eqref{eqn-bzeta} as the $\tilde V_{\pm}$-component of $\vb(\zeta)$: that is, 
\[\vb_{\pm} (\zeta):=  \,\int_{\theta_{0}}^{\pi-\theta_{0}} \Pi_{\tilde V_{\pm}}\left(\bbD\Psi(z, \theta-\theta_0)^{-1} 
\left(\eta\left(\frac{d\Psi(z, \theta)}{d\theta}\right) \right)\right)d\theta. \]

We define
\begin{gather} 
\alpha(\zeta):= \int \Bs\left(\bnu, \eta\left(\frac{d\Psi(z, \theta)}{d\theta} \right)\right)d\theta.
\end{gather}

\begin{proposition}\label{prop:mintegrals}
	\purple{ 
	Let $g$, $l$, and $\llrrV{\cdot}_{E}$ be as above. 
	Let $D$ be the standard horodisk. Let 
	$\eta$  be a standard cusp $1$-form for a cusp constant $\mu \newch{> 0}$. 
{\em (}See \eqref{eqn:cuspform}. {\em )}
	Suppose that a complement geodesic $l$ of radius $R$ is in the form of Lemma \ref{lem:estimation}.
	Let $\zeta= l\cap D$ be a geodesic segment with both endpoints in $\partial_h D$.  
	Then we obtain
\newch{
	\begin{multline}\label{eqn:mintegrals} 
	\llrrV{\vb_{- }(\zeta) }_E  = \mu \frac{\sqrt{-1 + R^2}(1 + 4 R^2)}{2 \sqrt{2} R^2} 
	\leq \frac{5}{2\sqrt{2}} \mu R, \\ 
	\mu \left(-\sqrt{2}  + 2  R^2\right)\frac{\sqrt{-1 + R^2}}{R}\leq 
	\alpha(\zeta) = \mu \frac{\sqrt{-1 + R^2}}{R} (\pm \sqrt{2}  + 2  R^2) 
	\leq  \mu (\sqrt{2}  + 2 R^2)  
	\end{multline}
}
where $R \geq 1$. 
}
\end{proposition}
\begin{proof} 
\purple{ 
In this case, we may regard $\bbD\Psi(z, \theta)^{-1}$ as the identity 
since we will work directly over $\Ss_+$ 
(see Remark \ref{rem:dualpict}): 
Since the projection  $\Pi_{\tilde V_{-}}$ to $\tilde V_{-}$ commutes with 
$\bbD\Psi(z, \theta)^{-1}$,
\[\vb_{ -} (\zeta):=  \,\int_{\theta_{0}}^{\pi-\theta_{0}} \bbD\Psi(z, \theta)^{-1} \left(\Pi_{\tilde V_{-}}(\eta)\right) dx\left(
\frac{d\Psi(z, \theta)}{d\theta}\right) d\theta. \]
By computations in \cite{NewNormalComp} or \cite{NewNormalComp2}, we obtain 
\[\llrrV{\vb_{ -} (\zeta)}=\mu \sqrt{-1 + R^2} \frac{(1 + 4 R^2)}{2 \sqrt{2} R^2}.\]
And we evaluate the contribution of $l\cap P$ to $\vb_{0}$\,:
\[
\alpha(\zeta) =\mu\left(\sqrt{-1 + R^2})(\pm \sqrt{2}  + 2 R^2)\right)/R.
\]
}

%

\end{proof}